\documentclass[a4paper,11pt]{amsart}

\DeclareRobustCommand{\SkipTocEntry}[5]{}

\usepackage{amsmath }
\usepackage{marvosym}

\usepackage[T1]{fontenc}
\usepackage{empheq}
\usepackage{relsize}
\usepackage{amsmath}
\usepackage{bm}

\usepackage{upgreek}
\usepackage{ esint }
\usepackage{color}
\usepackage{comment}
\usepackage{amssymb}
\usepackage{amsfonts}
\usepackage{graphicx}

\usepackage{amscd}

\usepackage{slashed}
\usepackage{pdflscape}
\usepackage{enumerate}
\usepackage{multirow}
\usepackage{stmaryrd}

\usepackage{tikz-cd}
\usepackage{tikz}
\usepackage{tkz-graph}
\usetikzlibrary{topaths}
\usetikzlibrary{arrows}
\usepackage{mathdots}
\usetikzlibrary{shapes,snakes}
\usetikzlibrary{decorations.shapes}
\usetikzlibrary{shapes,decorations,shadows}
\usetikzlibrary{decorations.pathmorphing}
\usetikzlibrary{decorations.shapes}
\usetikzlibrary{fadings}
\usetikzlibrary{patterns}
\usetikzlibrary{calc}
\usetikzlibrary{decorations.text}
\usetikzlibrary{decorations.footprints}
\usetikzlibrary{decorations.fractals}
\usetikzlibrary{shapes.gates.logic.IEC}
\usetikzlibrary{shapes.gates.logic.US}
\usetikzlibrary{fit,chains}
\usetikzlibrary{positioning}
\usepgflibrary{shapes}
\usetikzlibrary{scopes}

\usepackage{bbold}
\usepackage[linkcolor=blue]{hyperref}
\usepackage{mathrsfs}
\usepackage[colorinlistoftodos]{todonotes}

\usepackage[T1]{fontenc}

\definecolor{blue}{rgb}{.255,.41,.884} 
\definecolor{red}{rgb}{1, 0, 0} 
\definecolor{green}{rgb}{.196,.804,.196} 
\definecolor{yellow}{rgb}{1,.648,0} 
\definecolor{pink}{rgb}{1,0.5,0.5}

\newtheorem{theorem}{Theorem}[section]
\newtheorem{lemma}[theorem]{Lemma}
  
\newtheorem{proposition}[theorem]{Proposition}
\newtheorem{corollary}[theorem]{Corollary}

\theoremstyle{definition}
\newtheorem{definition}[theorem]{Definition}
\newtheorem{example}[theorem]{Example}

\theoremstyle{remark}
\newtheorem{remark}[theorem]{Remark}

\newtheorem{problem}[theorem]{Problem}

\usepackage{pdfsync}

\usepackage{graphicx} 
\usepackage{amsmath} 
\usepackage{amsfonts}
\usepackage{amssymb}

\newcommand{\be}{\begin{equation}}

\newcommand{\ee}{\end{equation}}

\newcommand{\II}{{ I\hspace{-.8mm}I}}
\newcommand{\IIo}{\mathring{\!{ I\hspace{-.8mm} I}}{\hspace{.2mm}}}

\newcommand{\ba}{\begin{array}}

\newcommand{\ea}{\end{array}}

\newcommand{\beq}{\begin{eqnarray}}

\newcommand{\eeq}{\end{eqnarray}}

\newtheorem{lm}{lemma}

\newtheorem{thee}{theorem}

\newtheorem{proo}{proposition}

\newtheorem{co}{corollary}

\newtheorem{rem}{remark}

\newtheorem{deff}{definition}

\newcommand{\bd}{\begin{deff}}

\newcommand{\ed}{\end{deff}}

\newcommand{\bl}{\begin{lm}}

\newcommand{\el}{\end{lm}}

\newcommand{\bp}{\begin{proo}}

\newcommand{\ep}{\end{proo}}

\newcommand{\bt}{\begin{thee}}

\newcommand{\et}{\end{thee}}

\newcommand{\bc}{\begin{co}}

\newcommand{\ec}{\end{co}}

\newcommand{\brm}{\begin{rem}}

\newcommand{\erm}{\end{rem}}

\usepackage{t1enc}

\def\Cal{\mathcal}

\newcommand{\bS}{\mathbb{S}}

\newcommand{\newc}{\newcommand}

\let\ccdot.
\def\cmdot{\hbox to 2.5pt{\hss$\hh\cdot\hh$\hss}}

\newc{\aR}{\mbox{\boldmath{$ R$}}}
\newc{\aS}{\mbox{\boldmath{$ S$}}}
\newc{\aT}{\mbox{\boldmath{$ T$}}}
\newc{\aW}{\mbox{\boldmath{$ W$}}}

\newc{\aD}{\mbox{\boldmath{$ D$}}\hspace{-.2mm}}

\newc{\aK}{\mbox{\boldmath{$ K$}}}
\newc{\aL}{\mbox{\boldmath{$ L$}}}

\newcommand{\ce}{{\Cal E}}

\newcommand{\bT}{{\Bbb T}}

\usepackage{amssymb}
\usepackage{amscd}

\newcommand{\Rho}{{\it P}}

\newcommand{\Ric}{{\it Ric}}
\newcommand{\Sc}{\it Sc}

\newcommand{\nn}[1]{(\ref{#1})}

\let\t=\tau

\newcommand{\J}{{\mbox{\it J}\hh}}

\newc{\obstrn}[2]{B^{#1}_{#2}}

\newcommand{\rpl}                         
{\mbox{$
\begin{picture}(12.7,8)(-.5,-1)
\put(0,0.2){$+$}
\put(4.2,2.8){\oval(8,8)[r]}
\end{picture}$}}

\newcommand{\lpl}                         
{\mbox{$
\begin{picture}(12.7,8)(-.5,-1)
\put(2,0.2){$+$}
\put(6.2,2.8){\oval(8,8)[l]}
\end{picture}$}}

\usepackage{ifthen}

\newc{\tensor}[1]{#1}
\newc{\Mvariable}[1]{\mbox{#1}}
\newc{\down}[1]{{}_{#1}}
\newc{\up}[1]{{}^{#1}}

\newc{\JulyStrut}{\rule{0mm}{6mm}}
\newc{\midtenPan}{\mbox{\sf S}}
\newc{\midten}{\mbox{\sf T}}
\newc{\midtenEi}{\mbox{\sf U}}
\newc{\ATen}{\mbox{\sf E}}
\newc{\BTen}{\mbox{\sf F}}
\newc{\CTen}{\mbox{\sf G}}

\def\sideremark#1{\ifvmode\leavevmode\fi\vadjust{\vbox to0pt{\vss
 \hbox to 0pt{\hskip\hsize\hskip1em
 \vbox{\hsize2cm\tiny\raggedright\pretolerance10000
  \noindent #1\hfill}\hss}\vbox to8pt{\vfil}\vss}}}

\numberwithin{equation}{section}

\newcommand{\hh}{{\hspace{.3mm}}}

\newcommand{\coker}{\operatorname{coker}}

\newcommand{\cc}{\boldsymbol{c}}

\renewcommand{\=}{\stackrel\Sigma =}

\newcommand{\sss}{\scriptscriptstyle}

\newcommand{\pdot}{{\textstyle\boldsymbol \cdot}\hspace{.05mm}}

\newcommand{\SSmidge}{{\hspace{-.2mm}}}

\renewcommand\geq{\geqslant}
\renewcommand\leq{\leqslant}

\newcommand{\ext}{{\rm d}}
\newcommand{\D}{\mathbf{L}}

\DeclareMathOperator{\Vol}{Vol}
\DeclareMathOperator{\Area}{Area}

\begin{document}

\renewcommand{\today}{}
\title{Conformal Geometry of Embedded Manifolds with Boundary from 
Universal Holographic Formul\ae\  
}
\author{Cesar Arias${}^\flat$, A. Rod Gover${}^\sharp$ \&  Andrew Waldron${}^\natural$}

\address{${}^\flat$
  Departmento de Ciencias Fisicas\\ 
  Universidad Andres Bello\\ 
  Sazie 2212 piso 7, Santiago de Chile and
  the Riemann Center for Geometry and Physics, Leibniz Universit\"at Hannover, Appelstra\ss e 2, Hannover 30167 Germany } \email{cesar.arias@unab.cl}
 
\address{${}^\sharp$
  Department of Mathematics\\
  The University of Auckland\\
  Private Bag 92019\\
  Auckland 1142\\
  New Zealand,  and\\
  Mathematical Sciences Institute, Australian National University, ACT 
  0200, Australia} \email{gover@math.auckland.ac.nz}
  
  \address{${}^{\natural}$
  Center for Quantum Mathematics and Physics (QMAP)\\
  Department of Mathematics\\ 
  University of California\\
  Davis, CA95616, USA} \email{wally@math.ucdavis.edu}

\vspace{3pt}

\renewcommand{\arraystretch}{1}

\begin{abstract}

For an embedded conformal hypersurface with boundary,
we construct critical order local invariants and  their canonically associated differential operators.
These are obtained holographically in a construction that uses a singular Yamabe problem and a corresponding minimal hypersurface with boundary. 
They include an extrinsic $Q$-curvature for the boundary of the embedded conformal manifold and, for its interior, the $Q$-curvature and accompanying boundary  transgression curvatures. 
This gives universal formul\ae\ for extrinsic analogs of Branson $Q$-curvatures that simultaneously generalize the Willmore energy density, including the boundary transgression terms required for conformal invariance. It also gives extrinsic conformal Laplacian power type operators associated with all these curvatures.
 The construction also gives formul\ae\ for the divergent terms and anomalies in the volume and hyper-area asymptotics determined by minimal hypersurfaces having boundary at the conformal infinity.
A main feature is the development of a universal, distribution-based, boundary calculus for the treatment of these and related problems.

\vspace{10cm}


\end{abstract}


\maketitle

\pagestyle{myheadings} \markboth{Arias, Gover \& Waldron}{Conformal geometry of embedded manifolds}

\newpage

\tableofcontents

\newcommand{\balpha}{{\bm \alpha}}
\newcommand{\balphas}{{\scalebox{.76}{${\bm \alpha}$}}}
\newcommand{\bnu}{{\bm \nu}}
\newcommand{\bnus}{{\scalebox{.76}{${\bm \nu}$}}}
\newcommand{\bnuss}{\hh\hh\!{\scalebox{.56}{${\bm \nu}$}}}

\newcommand{\bmu}{{\bm \mu}}
\newcommand{\bmus}{{\scalebox{.76}{${\bm \mu}$}}}
\newcommand{\bmuss}{\hh\hh\!{\scalebox{.56}{${\bm \mu}$}}}

\newcommand{\btau}{{\bm \tau}}
\newcommand{\btaus}{{\scalebox{.76}{${\bm \tau}$}}}
\newcommand{\btauss}{\hh\hh\!{\scalebox{.56}{${\bm \tau}$}}}

\newcommand{\bsigma}{{\bm \sigma}}
\newcommand{\bsigmas}{{{\scalebox{.8}{${\bm \sigma}$}}}}
\newcommand{\bbeta}{{\bm \beta}}
\newcommand{\bbetas}{{\scalebox{.65}{${\bm \beta}$}}}

\renewcommand{\bS}{{\bm {\mathcal S}}}
\newcommand{\bB}{{\bm {\mathcal B}}}
\newcommand{\bC}{{\bm {\mathcal C}}}

\renewcommand{\bT}{{\bm {\mathcal T}}}
\newcommand{\bM}{{\bm {\mathcal H}}}

\newcommand{\go}{{\mathring{g}}}
\newcommand{\nuo}{{\mathring{\nu}}}
\newcommand{\alphao}{{\mathring{\alpha}}}

\newcommand{\Ell}{\mathscr{L}}
\newcommand{\density}[1]{[g\, ;\, #1]}

\newcommand{\ceedot}{{\scalebox{2}{$\cdot$}}}

\newcommand{\Langle}{{\bm \langle}}
\newcommand{\Rangle}{{\bm \rangle}}
\newcommand{\Lodz}{\widehat{\mathbf L}}
\newcommand{\Dhat}{\mbox{\bf \DJ}}

\section{Introduction}

For closed, even dimension conformal manifolds, 
the study of Branson's $Q$-curvature invariant~\cite{BQ}
has been a major focus
in areas including
conformal geometry, geometric analysis and
physics, see for example~\cite{Chang-Qing,Mal,  Henningson,GrahamWitten,Anderson,FGQ,
Fefferman-Hirachi,
whatQ,Chang,
MalDj,
BaumJuhl,FGQJuhl}.
Partly motivating these studies is the result that  the integral of the $Q$-curvature over the manifold
 is a global conformal invariant.  The analog for a dimension $d-1\geq 2$ 
manifold $\widetilde\Sigma$ with boundary~$\partial \widetilde\Sigma=\Lambda$, is given by a $Q$-curvature--transgression pair $(\bm Q,\bm T)$ where 
$$\int_{\widetilde\Sigma}\bm Q + (d-2)\int_{\Lambda} \bm T\, ,$$
is then a global conformal invariant~\cite{Chang-Qing} (see for example~\cite{Chang,Ndi08,Case,GPnew}
for subsequent  $T$-curvature studies).
Any attempt to handle these objects using  standard Levi-Civita calculus is immediately frustrated by the fact that  (i) the formul\ae\  are extremely complicated and (ii) the indirect definition of the  $Q$-curvature. An aim of the current work is to solve this
 problem. In fact, we will treat a generalization.

 It has emerged recently 
 that there is a nice link between  $Q$-curvature and  the 
   Willmore energy 
  as well as its higher 
  dimensional analogs.
  This is in the form of a hierarchy where for conformally embedded hypersurfaces there is a curvature quantity, also denoted $Q$, that gives a Willmore-type energy density for hypersurfaces which
  specializes to the usual Branson $Q$-curvature
   for suitable even-dimensional embeddings. 
    When $\widetilde \Sigma$ is the boundary of any conformal compactification, the $Q$-curvature generalizes canonically~\cite{GW}. 
 In the case when
 the conformal compactification 
 obeys a suitable singular Yamabe problem~\cite{ACF}, 
 this $Q$-curvature is a distinguished, extrinsically coupled curvature invariant, determined by the conformal embedding and a choice of boundary metric with integral again a conformal invariant. 
 (See also~\cite{MazzeoC} for a general study of singular Yamabe problems.)
  When $\widetilde \Sigma$ is a surface,  the integral of this {\it extrinsic $Q$-curvature}
 is the  Willmore energy functional; in general this provides a higher dimensional analog~\cite{CRMouncementCRM,GW15,GW161,Grahamnew,GWvol}. The Willmore functional 
plays an important {\it r\^ole} in both mathematics and physics (see
{\it e.g.}~\cite{Riviere,Polyakov,Alexakis}); recently the celebrated
Willmore conjecture~\cite{Willmore} concerning absolute minimizers of
this energy was settled in~\cite{Marques}.  The Willmore energy is
also linked to  physical
observables~\cite{GrahamWitten} and in particular entanglement entropy
~\cite{RyuT1,RyuT2}.
We expect higher Willmore energies to be similarly important.

Recently the existence of a corresponding canonical, transgression for the extrinsic~$Q$-curvature
 was  established~\cite{GWvolII}. 
 Once again, the Levi-Civita  and  its accompanying Gau\ss--Codazzi--Ricci extrinsic calculus are inadequate for developing a general theory.
We also treat this problem.
  


We attack all these problems {\it holographically}.
This term originated in physics~\cite{'t  Hooft,Susskind} 
where it was concretely realized in  
 the AdS/CFT correspondence of~\cite{Mal}. This relates spacetime theories to theories living on the boundary of spacetime. Here this means that the geometry of $\widetilde\Sigma$ will be studied by embedding it in a conformal manifold $M$  of one dimension higher. The gain is simple formul\ae\ on $M$ that encode the complicated information of $(\bm Q,\bm T)$-pair;
 upon restriction to~$\widetilde\Sigma$
 these simple formul\ae\ produce
  the quantities of interest. 

The Branson $Q$-curvature was initially defined as an intrinsic quantity.
However, the holographic approach is a natural setting to study a generalizing analog 
that includes extrinsic embedding data, and for this extrinsic $Q$-curvature, the corresponding transgression~$T$. For a manifold $\widetilde \Sigma$ with boundary $\Lambda$, we consider a conformal embedding,  as depicted below:
%
%
%

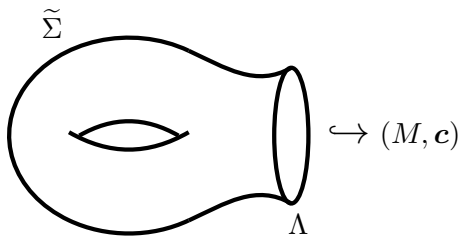
\begin{center}
\begin{tikzpicture}[scale=0.45, ultra thick]
\draw (0,0) ellipse (0.5cm and 2cm);
\draw (0,2) to[out=210,in=-25] (-3,2.5);
\draw (0,-2) to[out=-210,in=25] (-3,-2.5);
\draw (-3,2.5)
arc[start angle=60, end angle=300, x radius=3.5, y radius=2.885];
\draw (-6.2,0) to[bend left] (-3.3,0);
\draw (-6.5,0.1) to[bend right] (-3,0.1);
\node at (-7,3.3) {$\widetilde\Sigma$};
\node at (0.2,-2.6) {$\Lambda$};
\node at (3,0) {{\Large$\hookrightarrow$} $(M, \cc)$};
\end{tikzpicture}
\end{center}

\noindent
Thus we study the sequence of submanifold embeddings
$$
\Lambda\hookrightarrow \Sigma\hookrightarrow (M,\cc)\, ,
$$
where $M$ is a dimension $d$ conformal manifold, $\Sigma$ is a hypersurface (meaning dimension~$d-1$), $\Lambda$ is a closed hypersurface in $\Sigma$
(so $\dim(\Lambda)=d-2$) with interior $\widetilde\Sigma$. 
Throughout this paper we assume $d\geq 3$ (the cases $d=1,2$ are easily handled by more na\"ive methods).
To facilitate the study of these embeddings we consider two classical PDE problems:

First, the data of the conformal embedding $\Sigma \hookrightarrow M$ determines an approximate  metric~$g^o$ that is singular along $\Sigma$.  Given $g\in \cc$, then~$g^o=\sigma^{-2}g$; this is determined canonically up to terms of order $\sigma^d$  where $\sigma$ is a defining function for $\Sigma$ (meaning that~$\Sigma$ is the zero locus of $\sigma$ and $\ext \sigma|_\Sigma\neq 0$).  This is achieved via a singular Yamabe problem 
which requires the scalar curvature to  obey
$$
\Sc^{g^o}=-d(d-1)+{\mathcal O}(\sigma^d)\, .
$$

Second, the singular metric $g^o$ can be used to determine the asymptotics of a hypersurface~$\Xi$ anchored along~$\Lambda$ (so $\partial \Xi=\Lambda$ that is minimal so that its mean curvature~$H_\Xi^{g^o}$ vanishes, again to 
some order in $\sigma$:
$$
H_\Xi^{g^o}=\mathcal O(\sigma^{d-1})\, .
$$
We term $\Xi$ an {\it asymptotically minimal hypersurface}.
Furthermore, the defining function~$\mu$  for $\Xi$
can be canonically determined up to terms of order $\mu^d$ by the singular Yamabe problem for the embedding~$\Xi\hookrightarrow M$.
This geometry is depicted below:

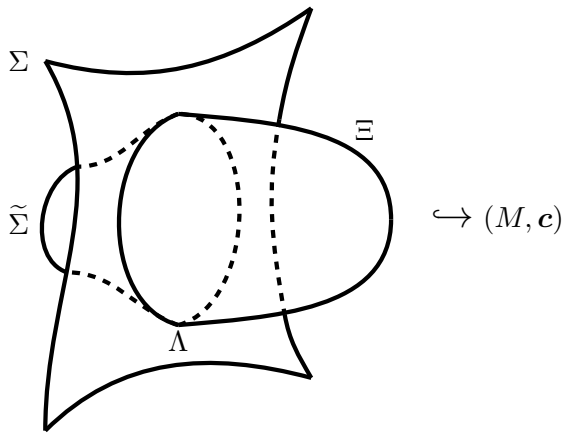
\begin{center}
\begin{tikzpicture}[scale=0.7, ultra thick]
\coordinate (LD) at (-3.5,-4);
\coordinate (RD) at (1.5,-3);
\coordinate (LU) at (-3.5,3);
\coordinate (RU) at (1.5,4);
\coordinate (CU)at (-1,2);
\coordinate (CD)at (-1,-2);
\coordinate(B) at (3,0);
\coordinate(IU) at (0.9,1.9);
\coordinate(ID) at (1,-1.8);
\coordinate(U) at (-2.9,1);
\coordinate(D) at (-3.1,-1);
\draw (CU) to[out=-160,in=165] (CD);
\draw [dashed] (CD) to[out=15,in=0] (CU);
\draw (CU) to[out=-5,in=90] (B);
\draw (CD) to[out=5,in=-90] (B);
\draw (LU) to[out=-15,in=-145] (RU);
\draw (LD) to[out=45,in=165] (RD);
\draw (RU) to[out=-110,in=80] (IU);
\draw (RD) to[out=120,in=-80] (ID);
\draw (LU) to[out=-60,in=90] (LD);
\draw [dashed] (CU) to[out=200,in=0] (U);
\draw[dashed] (CD)to[out=160,in=0](D);
\draw (U) to[out=-160,in=160] (D);
\draw [dashed] (IU) to[out=-100,in=100](ID);
\node at (-4,3) {$\Sigma$};
\node at (-4,0) {$\widetilde\Sigma$};
\node at (-1,-2.3) {$\Lambda$};
\node at (2.5,1.7) {$\Xi$};
\node at (5,0) {{\Large$\hookrightarrow$} $(M, \cc)$};
\end{tikzpicture}
\end{center}
This set-up allows us to study the embedding sequence $
\Lambda\hookrightarrow \Sigma\hookrightarrow (M,\cc)
$ holographically.

The above geometry was partly inspired by 
seminal work of Graham and Witten \cite{GrahamWitten} (see also~\cite{GrahamRiechert})
who study the embedding of closed submanifolds in $\Sigma$ by considering the volume asymptotics of a minimal submanifold in $M$ 
whose boundary is the given submanifold in $\Sigma$.
For this they considered the special case where the 
 metric $g^o$ obeys a Poincar\'e--Einstein condition; this effectively removes the information of the embedding
  $\Sigma\hookrightarrow (M,\cc)$.
In Section~\ref{Sec:MinSurf} we further develop the theory of aymptotically minimal hypersurfaces,
including results of independent interest.

%
%
%
%
%

%

\medskip

Returning to our general setting, 
note that even in the case that $\Xi$ and $\widetilde\Sigma$ close off a compact region $D$, the volume of this is infinite with respect to $g^o$.
Nevertheless, in the case of any suitably singular measure such as the one of $g^o$, there is a notion of a {\it regulated volume} $\Vol_\varepsilon$.  This is defined by the volume of a one parameter family of regions $D_\varepsilon\subseteq D$  approximating that enclosed and such that $D_0=D$. (See
 Equation~\nn{vol_eps} and preceding text for the precise definition of the regulated volume.) 
The asymptotic behavior of $\Vol_\varepsilon$ is then given by
\begin{equation*}\label{reg_vol0}
\Vol_\varepsilon=\sum_{k=d-1}^1 \frac{v_k}{\varepsilon^k}
+{\mathcal V}\, \log\varepsilon+\Vol_{\rm ren}
+\hh{\mathcal O}(\varepsilon)\, ,
\end{equation*} 
with details as follows:
The coefficients of the poles in $\varepsilon$ are local integrals (meaning their integrands are functions determined by finitely many jets of the given data)
along $\widetilde\Sigma =\Sigma \cap D$, these poles are termed {\it divergences}. These depend on the choice of regulating regions $D_\varepsilon$; we give universal holographic formul\ae\ for them in Theorem~\ref{Vdiv}.  
The constant term $\Vol_{\rm ren}$ is called the {\it renormalized volume}. The coefficient~${\mathcal V}$ of the logarithm is the {\it volume anomaly}. 
It is the non-derivative term in an  anomaly operator (see for example~\cite{GWvolII}) that measures the dependence of the renormalized volume on the choice of regulator and hence the metric in~$\cc_\Sigma$. Here  $\cc_\Sigma$  is  the conformal class of metrics induced on $\Sigma$ by $\cc$. The volume
anomaly
 is conformally invariant and expressed as a sum of local integrals along~$\widetilde\Sigma$ and $\Lambda=\partial\widetilde\Sigma$.
 In fact, distinguished integrands are canonically determined by our construction  
 and give the extrinsic $Q$-curvature and corresponding transgression introduced above.
As discussed above,  explicit formul\ae\ for $Q$-curvatures
and their extrinsic generalizations
explode in complexity beyond the simplest low dimensional examples (see~\cite{GoPetCMP}). However, 
our first main theorem establishes that these have  remarkably simple and universal holographic formul\ae\ in all dimensions.


In the following theorem we determine the regulated region $D_\varepsilon$ by introducing a nowhere vanishing positive conformal density $\bm \tau$, termed a  true scale or {\it regulator}. This additional data can be used to  determine a metric in $\cc$, and hence an induced metric along~$\Sigma$.
These notions are explained in Sections~\ref{CONFG} and~\ref{Sec:RenVA}.

\begin{theorem}
\label{Vanomaly} 
The anomaly in the regulated volume $\Vol_\varepsilon$ is determined by the embeddings 
$\Lambda\hookrightarrow \Sigma\hookrightarrow (M,\cc)$
and 
is conformally invariant. For any regulator $\btau$, it is given~by
\begin{equation*}
\mathcal V = \frac{(-1)^{d-1}}{(d-1)!} 
\Bigg[ \frac{1}{(d-2)!} \int_{\widetilde\Sigma} \bm Q_{\Sigma \hookrightarrow (M,\cc)}^{\btaus} 
+ \frac{1}{(d-3)!}\int_{\Lambda=\partial \widetilde\Sigma} \bm T_{\Lambda\hookrightarrow \Sigma \hookrightarrow (M,\cc)}^{\btaus}  \Bigg] ~,
\end{equation*}
where  $(\bm Q , \bm T )$ are local (density-valued) curvatures
depending on the regulator~$\btau$ and the conformal embeddings indicated. Given  
unit and minimal unit conformal defining 
densities $\bsigma$ and $\bmu$ for $\Sigma$ and $\Xi$, respectively, these have holographic formul\ae
\begin{equation*}
\bm Q_{\Sigma \hookrightarrow (M,\cc)}^{\btaus} = \D^{d-1}_\bsigma \log \bm\tau \Big|_{\Sigma} ~, \qquad
\bm T_{\Lambda\hookrightarrow \Sigma \hookrightarrow (M,\cc)}^{\btaus} = \sum_{j=0}^{d-3} (\D^{\!\rm T}_\bsigma)^j\, 
\D^\prime_\bmu \bm \D_\bsigma^{d-j-3}\log\bm\tau\Big|_{\Lambda} \, .
\end{equation*}
\end{theorem}
In the above, the conformal densities 
$\bsigma$ and $\bmu$ obey a unit and minimal unit property. This means that
$\bsigma=\sigma \btau$ and $\bmu= \mu \btau$ with $\btau = [g;1]$, where $g$, $\sigma$ and $\mu$  
obey the singular Yamabe and minimal hypersurface conditions that we just defined.
The logarithm of a conformal density is defined in Sections~\ref{CONFG}. 
The operators  $\D_\bsigmas$, $\D_\bsigmas^{\!\rm T}$ and $\D_\bmus
^\prime$ are all variants of the Laplace--Robin operator; these are distinguished  Laplacian-type operators on the manifold~$M$ that are degenerate along boundaries $\Sigma$, $\Lambda$ or $\Xi$ respectively, where they become the conformally invariant Robin-type combination of Dirichlet and Neumann operators of Cherrier~\cite{cherrier}, see Sections~\ref{LapRob} and~\ref{superlaprob}.
Theorem~\ref{Vdiv} gives analogous results for the divergences in the regulated volume.


A nice feature of this set-up is that it encodes a second geometric problem of independent interest (see~\cite{GrahamWitten}). Namely, the corresponding ``area'' of the minimal hypersurface~$\Xi$
(meaning its volume with respect to the induced singular metric; throughout we use ``volume'' to refer  to volumes of codimension zero regions and ``area'' for  volumes of codimension one regions)
is also infinite, but again can be regulated using the same method as employed above for the volume problem. Here also the  family of regulating regions $\Xi_\varepsilon\subset \Xi$
are described in terms of the true scale $\btau$. The details of this construction may be found in Section~\ref{Sec:RenVA}.
The asymptotic behavior of their areas is then given by
\begin{equation*}
\Area_\varepsilon=\sum_{k=d-2}^1 \frac{a_k}{\varepsilon^k}
+{\mathcal A}\, \log\varepsilon+\Area_{\rm ren}
+{\mathcal O}(\varepsilon)
\, ,
\end{equation*}
where the coefficients of poles are regulator-dependent local integrals over $\Lambda=\Sigma\cap\Xi$;
we give universal holographic formul\ae\ for them in Theorem~\ref{Adiv}.   The constant term $\Area_{\rm ren}$ is the {\it renormalized area}. The coefficient~${\mathcal A}$ of the logarithm is the {\it area anomaly}. It also measures the response of the renormalized area to the choice of regulator and is itself conformally invariant. The integrand also generalizes Branson's $Q$-curvature to this structures to what we shall call a {\it submanifold $Q$-curvature}. Even for the case where $\Lambda$ is a four-manifold and $M$ is specialized to be Poincar\'e--Einstein, the classical treatment of this requires a computational {\it tour de force}~\cite{GrahamRiechert}. Our next theorem gives a universal holographic formula for the area anomaly.

\begin{theorem}\label{Aanomaly}
The anomaly in the regulated area~$\Area_\varepsilon$
is determined by the embeddings 
$\Lambda\hookrightarrow \Sigma\hookrightarrow (M,\cc)$
and 
is conformally invariant. 
For any regulator $\btau$, it is given~by
\begin{equation*}
\mathcal A= \frac{(-1)^{d-2}}{(d-2)!(d-3)!} \int_\Lambda \bm Q_{\Lambda\hookrightarrow \Sigma \hookrightarrow (M,\cc)}^{\btaus}  ~,
\end{equation*}
where the holographic formula for the local (density-valued), $\btau$-dependent, submanifold~$Q$-curvature~$\bm Q_{\Lambda\hookrightarrow \Sigma \hookrightarrow (M,\cc)}^{\btaus}$, is
\begin{equation*}
\bm Q_{\Lambda\hookrightarrow \Sigma \hookrightarrow (M,\cc)}^{\btaus} =(\D_\bsigma^{\!\rm T})^{d-2} \log\bm \t \big|_{\Lambda}\, ,
\end{equation*}
where $\bsigma$ and $\bmu$ are
unit and minimal unit conformal defining 
densities  for $\Sigma$ and $\Xi$, respectively.
\end{theorem}
\noindent
The proofs of Theorems~\ref{Vanomaly} and~\ref{Aanomaly} are given in Section~\ref{Sec:Anomaly}.
Key to these proofs is a calculus, both integral and differential, for conformal geometries coupled to multiple scales; this is described in Section~\ref{Sec:Calculus}.
\medskip

A strong motivation for our work 
is that the volume and area problem not only leads to interesting global conformal invariants of the structure 
$\Lambda\hookrightarrow \Sigma \hookrightarrow (M,\cc)$, but also yields a rich local invariant theory surrounding the 
$\bm Q$  and $\bm T$ integrands. Of particular interest are conformally invariant differential operators  
that measure how 
$Q$ and~$T$ curvatures depend on the choice of metric in the conformal class. 
For the special case of Branson's $Q$-curvature, these are the (critical)  conformally invariant Laplacian powers of~\cite{GJMS}.
Their generalizations follow immediately 
from our construction upon remembering that the choice of regulator $\btau$ also amounts to a choice  of metric $g\in \cc$ and that $\log (e^f \btau) = f + \log\btau$ for any smooth function $f$. For example, the 
extrinsic $Q$-curvature of $\Sigma\hookrightarrow (M,\cc)$ obeys
$$
\bm Q_{\Sigma \hookrightarrow (M,\cc)}^{e^f\btaus}=
\bm Q_{\Sigma \hookrightarrow (M,\cc)}^{\btaus}+ \mathbf P_{\Sigma} f_{\Sigma}\, ,
$$
where $f_{\Sigma}=f|_\Sigma$ and the conformally invariant operator $\mathbf P_{\Sigma}:=\mathbf P^{(d-1)}_{\Sigma\hookrightarrow (M,\cc)}$ is the last member of  a sequence of conformally invariant extrinsically coupled Laplacian powers, as described  by simple holographic formul\ae\ in the next Theorem.
\begin{theorem}\label{oldiebutgoodie}
The operator 
defined by 
\begin{eqnarray*}
\mathbf{P}_{\Sigma\hookrightarrow (M,\cc)}^{(k)}:\Gamma\big(\ce \Sigma\big[\tfrac{k-d+1}2\big]\big)&\to& \Gamma\big(\ce \Sigma\big[\tfrac{-k-d+1}2\big]\big)\, ,
\nonumber\\[2mm]
\stackrel{\rotatebox{90}{$\in$}}
{{\bm f}^{\!\!\!\phantom A}}\quad\qquad &\mapsto&
\quad
\stackrel{\rotatebox{90}{$\in$}}{(\D^{k}_\bsigma {\bm f}_{\rm ext})\big|_\Sigma}\, ,
\end{eqnarray*}
 for $k\in\{1,\ldots, d-1\}$ and $\bm f_{\rm ext}$ any extension of $\bm f$ to $M$, is canonically determined by the embedding data $\Sigma\hookrightarrow (M,\cc)$. When $k$ is even, this has
 leading derivative term $\big((k-1)!!\big)^2\big(\Delta_\Sigma\big)^{\frac k2}$, where $\Delta_\Sigma$ is the intrinsic Laplacian of $\Sigma$ as determined by  a choice of  $g_\Sigma\in  \cc_\Sigma$.
\end{theorem}

Because the extrinsic ${Q}$-curvature 
pairs with a transgression ${\bm T}$ 
for manifolds with boundary, the critical extrinsic Laplacian power $\mathbf P_\Sigma$ pairs with a conformally invariant boundary operator $\mathbf U_{\Lambda\hookrightarrow \Sigma}$ that measures how the $T$-curvature responds to changing the choice of metric in the conformal class:
$$
\bm T_{\Lambda \hookrightarrow \Sigma\hookrightarrow (M,\cc)}^{e^f\btaus}=
\bm T_{\Lambda \hookrightarrow \Sigma\hookrightarrow (M,\cc)}^{\btaus}+ \mathbf U_{\Lambda\hookrightarrow\Sigma} f_{\Sigma}\, .
$$
The sum of the integral of the extrinsic $Q$-curvature~$\bm Q$ along $\widetilde\Sigma$ and 
$(d-2)$ times that of the transgression $\bf T$ along $\Lambda$ is conformally invariant; this implies an integral identity for the $(\mathbf P_\Sigma,\mathbf U_{\Lambda\hookrightarrow \Sigma})$ pair. 
This and the holographic formula for  $\mathbf U_{\Lambda\hookrightarrow \Sigma}$ are given in the next theorem.
\begin{theorem}\label{intPU}
The differential operator 
 defined by 
 \begin{eqnarray*} 
 \mathbf{U}_{\Lambda\hookrightarrow \Sigma}:\Gamma(\ce \Sigma[0])
&\to&
\qquad\quad
 \Gamma\big(\ce \Sigma[2-d]\big)\big|_\Lambda
\, ,
\\[2mm]
\stackrel{\rotatebox{90}{$\in$}}
{{ f}^{\!\!\!\phantom A}}\quad\quad &\mapsto&
\Big(
\sum_{j=0}^{d-3} (
\stackrel{\rotatebox{90}{$\in$}}{
\D^{\!\rm T}_\bsigma)^j\, 
\D^\prime_\bmu \bm \D_\bsigma^{d-j-3}
}\! {f}_{\rm ext}\Big)\Big|_\Lambda\, ,
\end{eqnarray*}
where $ f_{\rm ext}$ is any extension of $ f$ to $M$,
 is canonically determined by the embedding data $\Lambda\hookrightarrow\Sigma\hookrightarrow(M,\cc)$ and for any $f\in C^\infty \Sigma$ obeys
 $$
 \int_{\widetilde\Sigma} \mathbf {P}_\Sigma f + (d-2)
 \int_{\Lambda=\partial \widetilde\Sigma}
\mathbf{U}_{\Lambda\hookrightarrow \Sigma} f = 0\, .
 $$
%
Moreover, when $d$ is odd, $\mathbf{U}_{\Lambda\hookrightarrow \Sigma}$ has leading derivative term
$$
(d-2)!!\hh (d-4)!!
\nabla_{\!\hat m} \big(\Delta^g_\Sigma\big)^{\!\frac{d-3}2}
+{\rm LTOTs}\, ,
$$ 
for any $g\in \cc$,
where $\hat m$ is the unit outward normal to $\Lambda$  and ``{\rm LTOTs}'' denotes terms of lower transverse order, meaning lower order in $\nabla_{\hat m}$.\end{theorem}

The dependence on the choice of metric $g\in \cc$ for  the submanifold $Q$-curvature  is also controlled by a conformally invariant extrinsically coupled Laplacian power $\mathbf P_{\Lambda}:=\mathbf P_{\Lambda\hookrightarrow \Sigma\hookrightarrow (M,\cc)}^{(d-2)}$ (see Equation~\nn{transform}).
This operator is also the last member of a sequence of invariant operators. These are described holographically in our next theorem:
\begin{theorem}\label{Plambdaop}
The differential operator 
defined by
\begin{eqnarray*}
\mathbf{P}^{(k)}_{\Lambda\hookrightarrow \Sigma \hookrightarrow (M,\cc)}:\Gamma\big(\ce \Lambda\big[\tfrac{k-d+2}2\big]\big)
&\to&
\Gamma\big(\ce \Lambda\big[\tfrac{-k-d+2}2\big]\big)\, ,
\nonumber\\[2mm]
\stackrel{\rotatebox{90}{$\in$}}
{{\bm f}^{\!\!\!\phantom A}}\quad\qquad &\mapsto&
\quad
\stackrel{\rotatebox{90}{$\in$}}{\big((
\D^{\!\rm T}_\bsigma\big)^{k} 
 {\bm f}_{\rm ext}\big)\big|_\Lambda}\, ,
 \end{eqnarray*}
 for $k\in\{1,\ldots, d-2\}$, is canonically determined  by the embedding data $\Lambda\hookrightarrow\Sigma\hookrightarrow (M,\cc)$. When~$k$ is even, this has
 leading derivative term $\big((k-1)!!\big)^2\big(\Delta_\Lambda^g\big)^{\frac k2}$ for any $g\in \cc$. Moreover, when $\partial \Lambda=\emptyset$, for any $f\in C^\infty \Lambda$,
 $$
 \int_\Lambda \mathbf P_\Lambda f = 0\, .
 $$
\end{theorem}
\noindent
The proofs of Theorems~\ref{oldiebutgoodie},~\ref
{intPU} and~\ref{Plambdaop} 
are given in Section~\ref{Sec:GJMS}.

\medskip

The simple holographic formul\ae\ for the $({\bm Q},{\bm  T})$-curvature pair of Theorems~\ref{Vanomaly} and~\ref{Aanomaly} explode in complexity when expressed in terms 
 of standard Riemannian 
invariants along~$\Sigma$ and $\Lambda$.
For example, even the intrinsic $Q$-curvature for 
a conformal eight-manifold takes easily a page when 
expressed this way~\cite{GoPetCMP}. 
In the situation when one is given explicit data for the sequence of hypersurface embeddings $\Lambda\hookrightarrow \Sigma \hookrightarrow (M,\cc)$ in higher dimensions but where Riemannian formul\ae\ are unwieldy, one can employ our holographic formul\ae\  and computer software to compute these curvatures and operators. Examples of how to set up this kind of computation are given in~\cite{GWvol,GWvolII}.
When the  bulk $M$ is a three or four-manifold, explicit Riemannian formul\ae\ are still relatively compact; these are given in following pair of theorems,
the first of which gives the volume expansion, its anomaly and related curvatures and invariant operators. 
Any new notations appearing in these two  theorems
as well as their proofs are given in Section~\ref{Sec:Examples}.

\begin{theorem}\label{exvol}
Given the sequence of hypersurface embeddings $\Lambda\hookrightarrow \Sigma \hookrightarrow (M^d,\cc)$ and $g\in \cc$ (which induces $g_\Sigma\in \cc_\Sigma$) 
the 
$(\bm Q^g_{\Sigma\hookrightarrow (M,\cc)}, \bm T^g_{\Lambda\hookrightarrow\Sigma\hookrightarrow (M,\cc)})$ curvature pair and associated 
$({\mathbf P}_{\Sigma\hookrightarrow (M,\cc)}, {\mathbf U}_{\Lambda\hookrightarrow\Sigma\hookrightarrow (M,\cc)})$ operator pair are given by
\begin{alignat*}{4}
&\bm Q^g_\Sigma &&= J^{g_\Sigma}-\frac12\,\bm\IIo_{ab}^\Sigma\,\bm \IIo_\Sigma^{ab}\, , \qquad
&&{\bm T}^g_{\Lambda}&&=-H^{g_\Sigma}_{\Lambda\hookrightarrow \Sigma}\,,\\[2mm]
&\mathbf P_\Sigma&&=\bm \Delta_\Sigma\, ,
&&\mathbf U_{\Lambda} &&= \bm\nabla_{\bm{ \hat m}}\,,
\end{alignat*}
when $d=3$, and
\begin{align*}
&\bm Q_{\Sigma
}^g = -4\nabla^a_\Sigma\nabla^b_\Sigma \,\bm{\IIo}_{ab}^\Sigma
-8\, \bm{\IIo}^{ab}_\Sigma\, \bm{\mathcal F}_{ab}^\Sigma \, ,
&&
\bm T_{\Lambda
}^g
=
-2 \bm {\hat m^a}
\nabla^b_\Sigma \hh\bm \IIo^\Sigma_{ab}
+2
\rm\updelta_{\rm R} ^{\sss\Lambda\hookrightarrow \Sigma} 
\, \bm \IIo^\Sigma_{\bm{\hat m}\bm {\hat m}}\!
-2\hh \hh\bm \IIo^\Sigma_{ab}\hh\hh\bm \IIo_{\Lambda\hookrightarrow \Sigma}^{ab}\, ,
\\[2mm]
&
\mathbf P_\Sigma=8\bm \nabla_a^\Sigma \circ\bm \IIo^{ab}_\Sigma \circ \bm \nabla^\Sigma_b\, ,
&&
\mathbf U_{\Lambda} = 4\bm {\hat m}^a\bm \IIo^\Sigma_{ab} \bm \nabla_\Sigma^b\, ,
\end{align*}
when $d=4$. 

Accordingly, the regulated volume expansion~\eqref{reg_vol} is given by
\begin{multline*}
\Vol_\varepsilon^g= \frac{1}{2\varepsilon^2}\int_{\widetilde\Sigma} \ext V_{g_\Sigma} +\frac{1}{\varepsilon}\int_{\widetilde\Sigma} \ext V_{g_\Sigma}  H^g_{\Sigma} +\log\varepsilon\left[
\pi \chi_{\widetilde\Sigma}
-\frac 14\int_{\widetilde\Sigma} \bm\IIo_{ab}^\Sigma\,\bm \IIo_\Sigma^{ab}\right]
+
\Vol_{\rm ren} + \varepsilon {\mathcal R}(\varepsilon)\, ,
\end{multline*}
when $d=3$, and
\begin{multline*}
\Vol_\varepsilon^g=
\frac{1}{3\varepsilon^3} \int_{\widetilde\Sigma} \ext V_{g_\Sigma}
+
\frac{1}{\varepsilon^2}
\int_{\widetilde\Sigma} \ext V_{g_\Sigma}
H_{\Sigma}^g
-\frac{1}{2\varepsilon}
\int_{\widetilde\Sigma} \ext V_{g_\Sigma}
J^{g_\Sigma}
+\frac{1}{2\varepsilon}\int_\Lambda  
\ext V_{g_\Lambda}H_{\Lambda\hookrightarrow \Sigma}^{g_\Sigma}
\\
+\frac 23
\hh\log\varepsilon\left[
 \int_{\widetilde\Sigma}\bm{\IIo}^{ab}_\Sigma\, \bm{\mathcal F}_{ab}^\Sigma
-
\frac12\int_\Lambda
\Big(
\bm \updelta_{\rm R}^{\sss\lambda\hookrightarrow \Sigma} 
\, \bm \IIo^\Sigma_{\bm{\hat m}\bm {\hat m}}\
-\hh \hh\bm \IIo^\Sigma_{ab}\hh\hh\bm \IIo_{\Lambda\hookrightarrow \Sigma}^{ab}\Big)
\right]
+
\Vol_{\rm ren} + \varepsilon {\mathcal R}(\varepsilon)\, .
\end{multline*}
when $d=4$.
\end{theorem}

%
%
%


The area expansion, its anomaly and related curvatures and invariant operators  are given in our final theorem.

\begin{theorem}\label{exarea}
Given the sequence of hypersurface embeddings $\Lambda\hookrightarrow \Sigma \hookrightarrow (M^d,\cc)$ and $g\in \cc$ (which induces $g_\Lambda$ along $\Lambda$), the 
$\bm Q_{\Lambda\hookrightarrow\Sigma\hookrightarrow (M,\cc)}$ curvature and associated 
${\mathbf P}_{\Lambda\hookrightarrow\Sigma\hookrightarrow (M,\cc)}$ operator  are given by
$$
\bm Q_{\Lambda\hookrightarrow \Sigma}
\stackrel\Lambda
=\bm {\mathcal C}:=
-\hh
\bm\IIo^\Sigma_{\bm{ \hat\bm  m}   {\bm{\hat  m}}}\, ,
\qquad
\mathbf P_{\Lambda}=0\, ,
$$
when $d=3$ and
$$
{\bm Q}^{g}_{\Lambda\hookrightarrow \Sigma}\stackrel\Lambda=
J^{g_\Lambda}
+ \frac12 \bm K_{\Lambda\hookrightarrow \Sigma}
+\frac14 \bm K_\Sigma
- 2\bm {\mathcal F}^\Sigma_{ \bm {\hat m}\bm {\hat m}}
 -2\hh \bm\IIo^\Sigma_{\bm{\hat m}a}\bm g_\Lambda^{ab}\bm \IIo^\Sigma_{\bm{\hat m}b}
 -\bm {\mathcal C}^2\, ,\quad
 \mathbf P_{\Lambda}=\bm \Delta_\Lambda\, ,
$$
when $d=4$.

Accordingly, the regulated area expansion~\eqref{AREA} is given by
$$
\Area_\varepsilon^g=
\frac1\varepsilon
\int_\Lambda \ext V_{g_\Lambda}
+\log\varepsilon \int_\Lambda \bm\IIo^\Sigma_{\bm{ \hat\bm  m}   {\bm{\hat  m}}}+\Area_{\rm ren}
+\varepsilon\, {\mathcal R}(\varepsilon)\, , 
$$
when $d=3$ and
\begin{multline*}
\Area_\varepsilon^g=
\frac1{2\varepsilon^2}
\int_\Lambda \ext V_{g_\Lambda}
+\frac{1}{\varepsilon}\int_\Lambda \ext V_{g_\Lambda}
\big(H_{\Sigma}-\bm\IIo^\Sigma_{\bm{ \hat\bm  m}   {\bm{\hat  m}}} \big)
\\[1mm]
+ \log\varepsilon
\left[\pi \chi_\Lambda + 
 \frac14 
 \int_\Lambda 
 \Big(
 \bm K_{\Lambda\hookrightarrow \Sigma}
   - 4\bm {\mathcal F}^\Sigma_{\bm  {\hat m}\bm {\hat m}}
 -4\hh \bm \IIo^\Sigma_{\bm{\hat m}a}\bm g^{ab}_\Lambda \bm \IIo^\Sigma_{\bm{\hat m}b}
  -2\bm{\mathcal C}^2
\Big)\right]
\\[1mm]
+\Area_{\rm ren}
+\varepsilon\, {\mathcal R}(\varepsilon)\, , 
\end{multline*}
when $d=4$.
\end{theorem}

Many of our constructions and results have origins in tractor calculus techniques for hypersurfaces, see in particular~\cite{GW15,GW161}. 
They allow the Laplace--Robin technology to be extended to any conformally compact structure as well as to higher rank tensors.
Moreover, these techniques suggest a natural  generalization of our method to problems involving higher codimension embeddings; the key is to develop suitable analogs of the tangential Laplace--Robin operator~$\D_\bsigmas^{\!\rm T}$. 
We plan to present the details of a higher codimension tractor calculus for conformally embedded submanifolds and related volume and area problems in a separate manuscript.

%

%
%
%
%
%
%
%

\medskip


%


\subsection{Conventions}
We employ the notation $M^d$ to indicate the dimension $d$  of a smooth manifold $M$ and drop the superscript when this is clear from context. The Euler characteristic of a 
manifold $M$ is denoted by $\chi_M$, and the exterior derivative is denoted by $\ext$.
When $M$ is equipped with a metric $g$, the corresponding  volume element  is denoted $\ext V_g$.
We will often employ an abstract index notation for tensors on $M$, where for example, $v^a$ denotes a section of $TM$ but no choice of coordinates is implied (see, for example~\cite{ot}). Canonical operations such as contraction of a vector field $v^a$ and a one-form $\omega_a\in \Gamma(T^*M)$, are then given by expressions such as~$v^a \omega_a$. In this notation the squared length $|v|_g^2$ of a vector~$v^a$ with respect to a metric~$g$ is~$v^a g_{ab} v^b$, but we will often abbreviate this quantity by $v^2$. In the same vein, $u.v$ denotes $g_{ab} u^a v^b=g(u,v)$.  Throughout we work with Euclidean signature metrics, but many of our results generalize directly to a pseudo-Riemannian setting. 
The Levi--Civita connection of a metric $g$ is denoted $\nabla^g$ 
(again the superscript $g$ will be dropped when context makes this clear). 
The Riemann tensor $R$ of~$\nabla$ is defined by $R(u,v)w=\nabla_u \nabla_v w -\nabla_v \nabla_u w -\nabla_{[u,v]}$, where $[\pdot,\pdot]$ is the Lie bracket and $u,v,w$ are arbitrary vector fields. In the abstract index notation~$R$ is denoted by $R_{ab}{}^c{}_d$ and the Ricci tensor is $R_{ab}{}^a{}_d=:\Ric_{bd}$.
In  $d\geq 3$ dimensions, this is related to the Schouten tensor $\Rho$  by the trace-adjustment 
$\Ric_{ab}=(d-2)\Rho_{ab}+g_{ab} J$, where $J:=\Rho_a^a$ and the scalar curvature $\Sc=\Ric_a^a=2(d-1) J$.
When $d=2$, we define~$J:=\frac12 \Sc$. In dimensions $d\geq 4$, the Weyl tensor is defined by $W_{abcd}=R_{abcd}-P_{ac} g_{bd}+P_{bc} g_{ad}-P_{bd} g_{ac}+P_{ad} g_{bc}$. 
Moreover $W^{\sss \Omega^2 g}{\!\!}_{ab}{}^c{}_d=W^{\sss g}{\!}_{ab}{}^c{}_d$ for any smooth function $\Omega$. As for the quantity zero, we denote $0!=1$ and any operator~$A$ raised to the zeroth power is the identity map: $A^0={\rm Id}$.

\section{The Calculus of Scale}
\label{The Calculus of Scale}
An effective method for handling  geometric problems holographically  is to treat the  bulk as a conformal manifold. In this approach, asymptotically hyperbolic (or 
anti de Sitter) metrics are treated as a coupling of conformal geometry to a singular scale. This allows us to utilize potent hidden symmetries
  because the bulk conformal structure extends smoothly to the boundary even when the hyperbolic metric does not.

\subsection{Conformal geometry}\label{CONFG}

A {\it conformal manifold} $(M^d,\bm c)$ is a manifold equipped with a conformal class of metrics
$$
{\bm c}
:=[\hh g]=[\Omega^2 g]\, ,
$$
where $\Omega
$ is any smooth, strictly-positive function. 
Conformal densities are  fundamental objects on conformal manifolds:
A {\it (conformal) density}  of weight $w\in{\mathbb R}$ is 
a double
equivalence class of metrics and smooth functions  defined by 
$${\bm \nu}:=[\hh g ; \nu]=[\Omega^2g; \Omega^w \nu]\, .$$
Conformal densities may also be treated as sections of the line bundle ${\mathcal E}M[w]$ induced from the corresponding ${\mathbb R}_+$ representation where one  views the conformal structure as a ray subbundle of $\odot^2 T^*M$ and  in turn as an  ${\mathbb R}_+$ principal bundle over $M$~\cite{GoPetCMP,CapGoamb}. Weight zero densities $\bm f=[g;f]=[\Omega^2g;f]$ may be treated as functions and thus we often use  the notation~$f$ for these.
 We will also be interested in operators between sections of line bundles
$\bm {\mathcal A}:\Gamma({\mathcal E}M[w_1])\to \Gamma({\mathcal E}M[w_2])$. In this case we will refer to the difference $w_2-w_1=:w(\bm { \mathcal A})$ as the {\it operator weight}  of $\bm {\mathcal A}$ and also use the terminology ``$\bm {\mathcal A}$~is an operator of weight $w_2-w_1$''.
We also need the notion of a {\it log density}.
These are (see~\cite{GW})   equivalence classes of functions defined by an additive ${\mathbb R}_+$ representation
$$
\bm \lambda =[g;\lambda]=[\Omega^2g;\lambda + \ell \log \Omega]\, .
$$ 
The number $\ell$ is called the {\it (log)-weight} of $\bm \lambda$. A log density is a  section of  a log-density bundle ${\mathcal F}M[\ell]$ for some $\ell\in {\mathbb R}$. A more detailed description of the algebra and calculus of conformal densities and their tensor analogs  may be found in~\cite{GW,GWvol}. For example, $\bm v^a=[g;v^a]\in \Gamma(TM[w])$ denotes a vector-valued density of  weight~$w$ and its corresponding section space. 

When  $\bm \tau=[g;\tau]$  is a weight $w=1$ density and the smooth function $\tau$ is strictly positive, we call ${\bm \tau}$ a {\it true scale}. 
Given this, we can define a weight $\ell=1$ log-density $\log \bm \tau:=[g;\log\tau]$ and any weight $\ell=1$ log density can be written this way for some true scale.


A true scale canonically determines a Riemannian geometry~$(M,g_{\btaus})$ via the  equivalence class representative~$
{\bm \tau}=[g_{\bm \tau}; 1]$. Conversely, given a true scale ${\bm \tau}$ and a density ${\bm f}$, this canonically determines a function $f$ by expressing ${\bm f}=[g_{\btaus},f]$. Thus one may also regard a true scale as a choice of metric $g_{\btaus}\in \bm c$.
Alternatively, defining (tautologically) the symmetric cotensor-valued {\it conformal metric}
$
\bm g_{ab}=[g;g_{ab}]\in \Gamma(\odot^2 T^*M[2])
$
(which we denote by~$\bm g$ in an index-free notation), the metric given by the choice of scale $\btau$ is
$$
g_{\tau}=\frac{\bm g}{\btau^2}\, .
$$
It is often propitious to perform computations involving densities
by making such a choice, we shall 
 term this operation {\it working in a scale}
 and label it either by a choice of metric or a true scale, and
  use unbolded symbols for the corresponding equivalence class representatives for densities. Similarly, an unbolded notation will be used for operators acting on densities when a choice of scale has been made.
  Once a scale been chosen we may (or may not according to convenience) use this to trivialize density bundles.

 For submanifolds embedded in a conformal manifold, density bundles of the induced conformal structure agree with the restriction of ambient density bundles of the same weight. We shall use this frequently without further mention.


\subsection{Defining densities}\label{DefDs}
The data of a conformal manifold and a true scale is equivalent to a Riemannian geometry. The situation becomes far more interesting when one admits scales with a non-trivial zero locus.

\medskip 
Given an embedded  hypersurface $\Sigma\hookrightarrow M$, a {\it defining density} $\bm \sigma$ is a  weight $w=1$ density ${\bm \sigma}=[\hh g ; \sigma\, ]$ with zero locus $${\mathcal Z}({\bm \sigma}):=\{p\in M\, |\, \sigma(p)=0\}=\Sigma\, ,$$ and such that $\ext\sigma (p)\neq 0$, $\forall p\in \Sigma$ (so the function $\sigma$ is a {\it defining function} for $\Sigma$). For a given hypersurface a defining density  always exists locally.
 
 A defining density defines a  Riemannian metric $g^o$ via $\bsigma=[g^o;1]$ 
  away from the zero locus $\Sigma$  where $g^o$ is singular. The scalar curvature of this singular metric 
  is determined by 
   the conformal metric~${\bm g}$
  and the defining density~${\bm \sigma}=[g;\sigma]$
  through the weight zero density
  \begin{equation}\label{Scurvy}
{\bm {\mathcal S}}_\bsigmas:=\Big[ g ; 
|\ext \sigma|_g^2-\frac{2\sigma}d\,\big(\Delta^g \sigma +\sigma \J\, \big)\Big]\, , 
\end{equation}
termed the  {\it ${ {\mathcal S}}$-curvature}. Here $\Delta^g$ denotes the (negative energy) Laplacian of $g$ and  $J$ is the multiple of scalar curvature given by $J:=Sc^g/\big(2(d-1)\big)$ (here and elsewhere we will drop super and subscripts~$g$ and $\bsigma$ when these are clear from context). 
Away from~$\Sigma$ (and up to a negative multiplicative constant) the 
${ {\mathcal S}}$-curvature
is the  of 
$g^o$,
$$Sc^{g^o}=-d(d-1) {\bm {\mathcal S}}_\bsigmas.$$
This shows that the scalar curvature  $Sc^{g^o}$  of the metric $g^o$ extends smoothly to~$\Sigma$, where~$g^o$ is 
singular.

\medskip 
When working  in a scale $g$, we will denote $n:=\ext \sigma$ and
$$\rho_\sigma:=-\frac{1}{d}(\Delta^g + \J)\sigma\, .$$ 
The   
${ {\mathcal S}}$-curvature $
[g;{\mathcal S}]$ is then given by the function
\begin{equation}\label{Snrs}
{\mathcal S}=
n^2+2\rho\sigma\, .
\end{equation}
Hence, along the hypersurface~$\Sigma$, the ${ {\mathcal S}}$-curvature determines the length  of the conormal~$n$ and ${\bm {\mathcal S}}_\bsigmas$ is necessarily non-vanishing in a neighborhood of $\Sigma$. Since our considerations concern such a neighborhood, throughout and without loss of generality  we assume that~${\bm {\mathcal S}}_\bsigmas$
is everywhere non-vanishing.

\subsection{The singular Yamabe problem}~\label{SINGYAM}
Every Riemannian metric $g$ on a closed manifold~$M$ can be conformally rescaled to one of constant scalar curvature~\cite{Trudinger,Schoen,Aubin}, and the problem of finding $\Omega\in C^\infty M$ such that $\Sc^{\Omega^2g}$ is constant is termed the {\it Yamabe problem}~\cite{Yamabe}.
Compact manifolds with boundary and a singular metric $g^o=g/\sigma^2$, where the boundary and singularity of $g^o$ of $M$ are given by the zero locus  of the  defining function $\sigma\in C^\infty M$, are termed {\it conformally compact}. 
A Yamabe-type problem for the metric $g^o$ was  considered in~\cite{ACF}. Such and related problems   were  called the {\it singular Yamabe problem} in~\cite{MazzeoC}. When the metric~$g^o$ is in addition Einstein on~$M\backslash \partial M$, we say that $(M,g^o)$ is {\it Poincar\'e--Einstein}.

Actually it will be convenient to 
consider 
 conformal manifolds $(M,\cc)$ with defining density $\bsigma$ for an embedded,  oriented, separating hypersurface $\Sigma$. On each side of $\Sigma$, the defining density gives a canonical singular metric $g^o$ via $\bsigma=[g^o;1]$; or equivalently 
$$
g^o=\frac{\bm g}{\bm \sigma^2}\, .
$$
The hypersurface $\Sigma$ is  termed a {\it conformal infinity} for the metric $g^o$ and we may again ask whether there exists a defining function for $\Sigma$ such that its singular metric has constant scalar curvature. We shall also use the singular Yamabe problem moniker for  this problem. 

Whether $\Sigma$ is a boundary component or hypersurface, the requirement that $\Sc^{g^o}$ is constant may be phrased in terms of the ${ {\mathcal S}}$-curvature. In particular, for negative constant scalar  curvature equaling $-d(d-1)$ we must solve
\begin{equation}\label{singingYam}
\bS_\bsigmas = 1\, ,
\end{equation}
for $\bsigma$ given $\Sigma\hookrightarrow(M,\cc)$.
Of relevance to us here is that 
this problem is intimately related to the study of invariants of the conformal embedding~$\Sigma\hookrightarrow(M,\cc)$ (see~\cite{CRMouncementCRM,GW15,GW161}). Indeed, as demonstrated in~\cite{ACF} (see also~\cite{GW15}) the singular Yamabe problem cannot in general be solved smoothly, but rather is obstructed. However there does exist an asymptotic solution for~$\bsigma$ such that
\begin{equation}\label{asymWill}
\bS_\bsigma = 1 + \bsigma^d \bm {\mathcal F}\, ,
\end{equation}
where $\bm {\mathcal  F}$ is a smooth, weight $-d$ density  known as the {\it obstruction density}. The quantity $\bm F_\Sigma=\bm {\mathcal F}|_\Sigma$ is a non-trivial  invariant of $\Sigma\hookrightarrow(M,\cc)$.
When $M$ is a conformal three-manifold this gives the {\it Willmore invariant} which is the functional gradient of the Willmore energy functional~\cite{CRMouncementCRM}. A defining density solving 
Equation~\nn{asymWill} is termed a {\it (conformal) unit defining density}. 
Unit  defining densities always exist and are unique up to addition of terms $\bsigma^{d+1}\bm A=:{\mathcal O}(\bsigma^{d+1})$ where $\bm A$ is any smooth weight $-d$ density.
In general, given any weight one density $\bnu$, we use the notation~${\mathcal O}(\bnu^k)$ to denote $\bnu^k \bm A$ where $\bm A$ is any smooth density of the appropriate weight and tensor type determined by context.

\subsection{The Laplace--Robin operator}\label{LapRob}
Natural Laplace-type equations 
on conformal geometries coupled to scale 
enjoy a
hidden solution generating algebra~\cite{GW}. To uncover this, inspired by the tractor calculus based approach of~\cite{GoSigma,GoverS,GW}, we introduce  an operator that combines a bulk Laplace operator and a conformally invariant boundary operator:

Let $\bm \sigma=[g ; \sigma]$ be a weight~1 density. Then
the corresponding {\it Laplace--Robin operator}~
$$\D_{\bsigmas}:\Gamma({\mathcal E}M[w])\longrightarrow \Gamma({\mathcal E}M[w-1])\, ,$$ maps weight $w$   densities to 
weight $w-1$~densities $\bm f=[g;f]$ according to 
\begin{equation}\label{Ldef}
\D_{\bsigmas}{\bm f}:=\big[g ; (d+2w-2)(\nabla_n+w\rho)f-\sigma
(\Delta+w J)f
\big]\, .
\end{equation}
The Laplace--Robin operator  also maps  weight~$\ell$ log-densities~$\bm \lambda$ to weight $-1$ densities:
\begin{equation}\label{Llog1}
\D_\bsigmas\bm \lambda:=
\big[g ; (d-2)(\nabla_n \lambda+\ell\rho) 
 -\sigma (\Delta \lambda+\ell J)\big]\, .
\end{equation} 
It is not difficult to verify that if the log density $\bm \lambda=\log (e^f \btau)$ where $\btau>0$ is a density of any non-vanishing weight, and $f\in C^\infty M$, then
$$
\D_{\bsigmas} \log(e^f \btau)= \D_{\bsigmas} \log \btau + \D_{\bsigmas} f\, ,
$$
where the right hand side is computed using the definitions given in Equations~\nn{Ldef}
and~\nn{Llog1}.

\medskip

The operator $\D_\bsigmas$ is a Laplacian-type operator that is degenerate along the zero locus of~$\bm \sigma$. 
It arises naturally in a tractor calculus description of (pseudo)Riemannian geometries and related physical wave equations from a conformal perspective~\cite{GoSigma,GoverS}. 
In the case that $\bm \sigma$ is a defining density for a hypersurface~$\Sigma$, this operator restricted to~$\Sigma$ is proportional  to a conformally invariant Robin-type ({\it i.e.}, Dirichlet plus Neumann) operator along the corresponding hypersurface~$\Sigma$. When the defining density additionally obeys $\bS_\bsigma = 1 + {\mathcal O}\big(\bsigma^2\big)$
then, on weight $w\neq 1-\frac d2$ densities $\bm f=[g;f]$, along $\Sigma$ one has \begin{equation}\label{ROBIN}
\frac{1}{d+2w-2}\D_\bsigmas \bm f\Big|_\Sigma=
[g; (\nabla_{\hat n} - w H^g)f]\big|_\Sigma
=:\bm \updelta_{\rm R}
^{\sss\Sigma} \bm f
\, .
\end{equation}
The operator $\bm \updelta_{\rm R}
^{\sss\Sigma}$ is well-defined by the above display for all weights $w$, and is termed the {\it conformal  Robin operator}~\cite{cherrier}. 
%
%
%
%
%
%
%
%
%
Here $\hat n$ is the unit conormal to $\Sigma$, and $H^g$ is the mean curvature of $\Sigma$ with respect to the metric $g$. The quantity $\bm \updelta_{\rm R}
^{\sss\Sigma} \bm f$ is an example of a conformal hypersurface invariant.
These are invariants of the conformal embedding~$\Sigma\hookrightarrow (M,\cc)$. The precise definition is given in~\cite[Section 6.1]{GW15}. 
We use a boldface notation for these invariants, other examples include the 
trace-free second fundamental form $\bm \IIo_{ab}^\Sigma\in \Gamma(\odot^2T^* \Sigma[1])$ of  the hypersurface $\Sigma$ and its unit conormal $\bm {\hat n}_a\in \Gamma(T^* \Sigma[1])$.

Away from the hypersurface $\Sigma$, when computed in the scale $g^o$ corresponding to the  scale $\bsigma$, the Laplace--Robin operator gives the 
Laplacian-type operator $$-\Delta^{\!o}-\frac{2J^o}{d}\, w(d+w-1)\, ,$$
where $\Delta^o$ is the Laplacian of the singular metric $g^o$ and similarly for $J^o$.
%
Moreover, acting on a weight $w=1-\tfrac d2$ density $\bm f$, the Laplace--Robin operator obeys
$$
\D_\bsigmas \bm f =- \bm \sigma\,  \square \bm f\, ,
$$
where $\square$ is the conformally invariant {\it Yamabe operator}
\begin{equation}\label{Yamop}
\begin{array}{rcccc}
\square&\!\!\!\!\!\!\!:\!\!\!\!\!\!\!&
\Gamma({\mathcal E}M[1-\tfrac d2])&\!\!\longrightarrow\!\!&\!\!
\Gamma({\mathcal E}M[-1-\tfrac d2])\\
&&\rotatebox{90}{$\in$}&&\rotatebox{90}{$\in$}\\
&&[g;f]&\longmapsto&\!\big[g;\Delta f+(1-\tfrac d2) Jf\big]\, .
\end{array}
\end{equation}
We will refer to the distinguished weight $w=1-\frac d2$ as the {\it Yamabe weight}.

The utility of the Laplace--Robin operator is based on the following solution generating~$\mathfrak{sl}(2)$ triple (see~\cite{GW}): Define the operators 
$$
\begin{array}{rcccc}
x&\!\!\!\!\!\!\!:\!\!\!\!\!\!\!&
\Gamma({\mathcal E}M[w])&\!\!\longrightarrow\!\!&\!\!
\Gamma({\mathcal E}M[w+1])\\
&&\rotatebox{90}{$\in$}&&\rotatebox{90}{$\in$}\\
&&[g;f]&\longmapsto&\!\![g;\sigma f]\, ,
\end{array}
$$
$$
\begin{array}{rcccc}
h&\!\!\!\!\!\!\!:\!\!\!\!\!\!\!&
\Gamma({\mathcal E}M[w])&\!\!\longrightarrow\!\!&\!\!
\Gamma({\mathcal E}M[w])\\
&&\rotatebox{90}{$\in$}&&\rotatebox{90}{$\in$}\\
&&[g;f]&\longmapsto&\!\![g;(d+2w) f]\, ,
\end{array}
$$
and, supposing the ${ {\mathcal S}}$-curvature is nowhere vanishing, 
$$
\begin{array}{rcccc}
y&\!\!\!\!\!\!\!:\!\!\!\!\!\!\!&
\Gamma({\mathcal E}M[w])&\!\!\longrightarrow\!\!&\!\!
\Gamma({\mathcal E}M[w+1])\\
&&\rotatebox{90}{$\in$}&&\rotatebox{90}{$\in$}\\
&&[g;f]&\longmapsto&\!\![g;-{\rm L}_\sigma f/{\mathcal S}]\, ,
\end{array}
$$
where ${\rm L}_\sigma f:=(d+2w-2)(\nabla_n+w\rho)f-\sigma
(\Delta+w J)f$. Note that the operator $y=-(1/{\bm {\mathcal S}}_\bsigmas) \D_{\bsigmas}$ and $x$ is multiplication by $\bsigma$.
Then, a key result of~\cite{GW} is that for any conformal class of metrics $\bm c$, the operator triple $(x,h,y)$ obeys the standard $\mathfrak{sl}(2)$ commutator relations
\begin{equation}\label{thesl2}
[x,y]=h\, ,\quad [h,x]=2x\, ,\quad [y,h]=2y\, .
\end{equation}
In particular, we will rely heavily on the ${\mathcal U}\big(\mathfrak{sl}(2)\big)$  enveloping algebra identities
\begin{equation}\label{envelop}
[y,x^k]=-x^{k-1}k(h+k-1)\, ,\quad
[y^k,x]=-y^{k-1}k(h-k+1)\, .
\end{equation}
Observe that acting on densities of weight $w=\frac{1}2\, (k-d-1)$, the right hand side of the second identity above vanishes, which implies that the operator $y^k$ acting on densities $\bm f$ of weight $\frac{1}2\, (k-d+1)$ has the special property
$$
y^k (\bm f +\bsigma \bm f')|_\Sigma = y^k \bm f|_\Sigma\, ,
$$
where here $\bm f'$ is any smooth density of weight $\frac{1}2\, (k-d-1)$. So along $\Sigma$, the action of~$y^k$ is independent of the  choice of extension $\bm f$ of the boundary data $\bm f|_\Sigma$. We encode this notion in a general definition:

\begin{definition}\label{TangDef}
Let $\bm{\mathcal A}:\Gamma(\ce M[w])\to
\Gamma(\ce M[w'])$. Then we call the operator $\bm{\mathcal A}$ {\it tangential to}  $\Sigma={\mathcal Z}(\bsigma)$ if it obeys
\begin{equation}\label{tangdef}
\bm{\mathcal A} (\bm f +\bsigma \bm f')|_\Sigma = \bm{\mathcal A} \bm f|_\Sigma\, ,
\end{equation}
for any $\bm f \in \Gamma(\ce M[w])$ and $\bm f' \in \Gamma(\ce M[w-1])$. 

For operators acting on log-densities, we say 
$\bm{\mathcal A}:\Gamma({\mathcal F} M[1])\to
\Gamma(\ce M[w])$ is {\it tangential to}  $\Sigma$ when for any true scale $\bm \tau$,
$$
\bm{\mathcal A} \big(\log(f\btau)\big) |_\Sigma = 
\bm {\mathcal A} \big(\log(\btau)\big) |_\Sigma\, ,
$$
for all $0<f\in C^\infty M$ such that $f|_\Sigma = 1$.

\end{definition}

\begin{remark}\label{moretang}
Since $\bm f$ and $\bm f'$ are smooth and $\bm {\mathcal A}$ is linear, the requirement in Equation~\nn{tangdef}
can be restated as
 $$
\bm{\mathcal A} (\bsigma \bm f' )= \bsigma \bm h\, ,
$$
for some smooth $\bm h\in \Gamma(\ce M[w'-1])$, and analogously for operators on log densities.
Tangential operators are  useful for expressing operators along~$\Sigma$ holographically because, as a consequence of the above definition,
a tangential operator~$\bm {\mathcal A}$ defines an operator 
$$\mathbf { A}:\Gamma(\ce \Sigma[w])\to
\Gamma(\ce \Sigma[w'])$$
according to 
$$
\mathbf { A}
 {\bm f}:= \bm{\mathcal A} {\bm f}_{\rm ext}\big|_\Sigma\, ,
$$
where $\bm f_{\rm ext}$ is any smooth extension of $  {\bm f}\in\Gamma(\ce \Sigma[w]) $ to $M$.
\end{remark}

Definition~\ref{TangDef} and its consequences remarked upon above,  extend naturally to higher codimension embedded submanifolds. In particular for a codimension  two embedded submanifold $\Lambda$ defined as the zero locus 
of a pair of defining densities $\bsigma$ and $\bmu$, we require
$$
\bm {\mathcal A} (\bm f +\bsigma \bm f'+\bmu \bm f'')|_\Lambda = \bm{\mathcal A} \bm f|_\Lambda\, ,
$$
for any $\bm f',\bm f'' \in \Gamma(\ce M[w-1])$. We call such operators {\it tangential to} $\Lambda={\mathcal Z}(\bsigma,\bmu)$; they define operators $\Gamma(\ce \Lambda[w])\to
\Gamma(\ce \Lambda[w'])$.
We will also encounter the intermediate case where 
 tangentiality of an operator only holds along a hypersurface $\Lambda\hookrightarrow \Sigma$
  with respect to a single defining function $\bm \sigma$, so that
$$
\bm {\mathcal A} (\bm f +\bsigma \bm f')|_\Lambda = \bm{\mathcal A} \bm f|_\Lambda\, ,
$$
for any $\bm f' \in \Gamma(\ce M[w-1])$.
Here we will say that $\bm {\mathcal A}$ is {\it tangential to~$\Sigma$ along~$\Lambda$}. This defines a conformally invariant operator along $\Lambda$ that may take derivatives in directions normal to the embedding of $\Lambda$  in~$\Sigma$.
When $\Lambda$ is the intersection of two hypersurfaces~$\Sigma$ and~$\Xi$, tangentiality to~$\Sigma$ along $\Lambda$ {\it and}  
tangentiality to $\Xi$ along $\Lambda$ together imply tangentiality to $\Lambda$.

\subsection{Leibniz rules}\label{LRs}
Because it is a scale-coupled conformal analog of the Laplace operator, the Laplace--Robin operator does not obey the Leibniz rule when acting on products of densities. 
To handle this feature we proceed as follows.  For weight $w\neq 1-\tfrac d2$ densities~$\bm f$ and weight $\ell$ log densities $\bm \lambda$ we define
\begin{equation}\label{LODZ}
\left\{
\begin{array}{l}
\Lodz_\bsigmas \bm f:=
\frac{1}{d+2w-2}
\D_\bsigmas \bm f\, , \\[2mm]
\Lodz_\bsigmas \bm\lambda:=\frac1{d-2}\, \D_\bsigmas\bm \lambda\, .
\end{array}
\right.
\end{equation}
The first of these operators obeys a 
 generalized Leibniz rule.
\begin{equation}\label{Leibniz}
\Lodz_\bsigmas  \mbox{$\L$}(\bm f \bm f')=
(\Lodz_\bsigmas \bm f) \bm f' + \bm f \, \Lodz_\bsigmas \bm f' -
\frac{2\bm \sigma}{d+2 (w+w')-2}\,  \Langle\bm f, \bm f'\Rangle \, .
\end{equation}
In the above $\bm f$ and $\bm f'$ are densities of weight $w\neq 1-\tfrac d2\neq w'$ and $w+w'\neq 1-\tfrac d2$. The  bracket $\Langle\cdot,\cdot\Rangle$ is defined by
\begin{equation}
\begin{split}
\Langle \bm f,\bm f'\Rangle&=
\Big[g;(\nabla f).\nabla f'
-\tfrac{w}{d+2w'-2}f\Delta f' -
\tfrac{w'}{d+2w-2}f'\Delta f 
-\tfrac{2ww'(d+w+w'-2)}{(d+2w-2)(d+2w'-2)}Jff'\Big]\\[2mm]&
\in\,  \Gamma({\mathcal E}M[w+w'-2])\, .\label{thebracket}
\end{split}
\end{equation}

\noindent
In the case that $w=w'=1$ we will often use an unbolded notation for the resulting  function-valued symmetric bracket. Indeed, when $\bm \sigma$ and $\bm \mu$ are defining densities for hypersurfaces~$\Sigma$ and~$\Xi$ intersecting along a codimension two submanifold  $\Sigma\cap\Xi$, then
$$
\left.
\frac{\langle\bm \sigma,\bm \mu\rangle}{\sqrt{\langle\bm \sigma,\bm \sigma\rangle\langle\bm \mu,\bm \mu\rangle}}\, \right|_{\Sigma\cap\Xi}
$$ 
computes the cosine of the angle between the respective conormals to $\Sigma$ and $\Xi$. Note also that $$\langle\bm \sigma,\bm \sigma\rangle=\bm {\mathcal S}_\bsigmas\, .$$ 
Moreover, when $\bm f$ has weight $w\neq 1-\tfrac d2$, it follows that 
\begin{equation}\label{lapland}\Langle \bsigma,\bm f\Rangle=\Lodz_\bsigma \bm f\, .
\end{equation}

We will also need analogs of the generalized Leibniz rule~\nn{Leibniz} at certain critical weights. First, when $w+w'=1-\tfrac d2$ but $w\neq 1-\tfrac d2\neq w'$ we have 
\begin{equation}\label{critLbox}
\D_\bsigmas (\bm f\bm f')=-2\bsigma \Langle \bm f,\bm f'\Rangle\, .
\end{equation}
The case when $w=1-\tfrac d2\neq w',w+w'$ is more delicate. For that we focus on the case that $\bsigma$ is a defining density for a hypersurface~$\Sigma$. Then we observe that the space of equivalence classes with respect to the equivalence $
\bm f\sim \bm f+\bsigma \bm f'$
for any $\bm f'\in
\Gamma({\mathcal E}M[w-1])$,
denoted by
\begin{equation}\label{funnycod}
\Gamma_\Sigma({\mathcal E}M[w]):=\big\{[\bm f]
:\, 
\bm f\in\Gamma({\mathcal E}M[w]) \big\}\, ,
\end{equation}
is congruent to the space $\Gamma({\mathcal E}M[w])$ of conformal densities along the hypersurface with conformal class of metrics $\bm c_\Sigma$ induced by $\bm c$.
Observe that operators that are tangential along $\Sigma$ are canonically well-defined on the space $\Gamma_\Sigma({\mathcal E}M[w])$.

Now we note that along the hypersurface $\Sigma$,   the operator $\nabla_n + (1-\frac d2)\rho_\sigma$ is the conformal Robin operator $\bm \updelta_{\rm R}^{\sss\Sigma}$ at the weight $1-\frac d2$.
Therefore we have the following well-defined 
operator
\begin{equation}\label{Lodzhat}
\begin{array}{rcccc}
\widetilde{\D}_\bsigmas\!\!&\!\!\!\!\!\!\!\!\!\!:\!\!\!\!\!\!\!\!\!\!&
\!\!\Gamma({\mathcal E}M[1-\tfrac d2])&\!\!\!\!\longrightarrow\!\!\!\!&\!\!
\Gamma_\Sigma({\mathcal E}M[-\tfrac d2])\\
&&\rotatebox{90}{$\in$}&&\rotatebox{90}{$\in$}\\
&&[g;f]&\!\!\longmapsto\!\!&\!\big[g;\nabla_n f + (1-\tfrac d2) \rho_\sigma f\big]\, .
\end{array}
\end{equation}
In the last expression in the above display we have used the square brackets to indicate equivalence classes with respect to both metric rescalings and the addition of smooth terms proportional to $\sigma$.

Given a  defining density 
$\bmu=[g;\mu]$ for a second hypersurface  $\Xi$,
it is possible to build invariant operators along $\Xi$ from 
combinations and compositions of representatives of the  operators $\widetilde\D_\bsigma$ 
and $\widetilde{\D}_\bmu$, as well as  analogs for the bracket of Equation~\nn{thebracket}.
%
A key example  is the combination (which by a slight abuse of notation will be) denoted by 
$
\widetilde \D_\bsigmas
-\bsigma\widehat \D_\bmu \widetilde\D_\bmu
$ with  domain~$\Gamma({\mathcal E}M[1-\tfrac d2])$.
 This can be made well-defined with codomain $\Gamma_\Sigma({\mathcal E}M[-\tfrac d2])$ by suitably interpreting the second term, which in any case does not contribute.
 In fact, one can also make it well-defined with codomain $\Gamma_\Xi({\mathcal E}M[-\tfrac d2])$
 via a continuation argument in the weight $w$ by simultaneously using the corresponding representatives of the two tilded operators.
 Indeed there is a pole $1/(d+2w-2)$ when   extending the operator~$\widehat \D_{\hh_{\textstyle\boldsymbol \cdot}}$ defined in  Equation~\nn{LODZ} to the critical weight $w=1-\frac d2$, but this then cancels for the particular combination~$
\widetilde \D_\bsigmas
-\bsigma\widehat \D_\bmu \widetilde\D_\bmu
$. Hence, along $\Xi$, we may invariantly define
\begin{equation}
\label{funnycombination}
\begin{array}{rcccc}
\widetilde \D_{\hh\bsigmas\!}
-\!\bsigma\widehat \D_\bmu \widetilde\D_\bmu\!
\!\!&\!\!\!\!\!\!\!\!\!\!:\!\!\!\!\!\!\!\!\!\!&
\!\!\!\!\Gamma({\mathcal E}M[1-\tfrac d2])&\!\!\!\!\!\!\!\longrightarrow\!\!\!\!&\!\!
\Gamma_\Xi({\mathcal E}M[-\tfrac d2])\\
&&\rotatebox{90}{$\in$}&&
\!\!\!\!\!\!\!\!\!\rotatebox{90}{$\in$}\\
&&[g;f]&\!\!\!\!\!\longmapsto\!\!&\!\big[g;\big(\nabla_n \! + \!(1-\tfrac d2) \rho_\sigma\!-\!\sigma 
(\nabla_m \!-\!\tfrac d2 \rho_\mu)
(\nabla_m \!+\!(1-\tfrac d2) \rho_\mu)\big)f
 \big] ,
\end{array}
\end{equation}
where 
$m:=\ext \mu$. 
 \medskip

Along similar lines, we define a modified ``bracket'', whose codomain depends on its second argument:
\begin{equation}\label{takeoff}
\Langle\hspace{-3.4mm}- \hspace{-.1mm}
\!\bm f,\bm f'\Rangle:=\Big[g;
(\nabla f).\nabla f' 
+\tfrac{d-2}{2(d+2w'-2)}\, f\big(\Delta+w' J \big)f'\Big]
\in
 \coker \bm f'\, .
\end{equation}
In the above, $\coker \bm f'$ is the cokernel of $\bm f'$ viewed as the linear operator mapping $\Gamma({\mathcal E}M[-1-\tfrac d2])\to \Gamma({\mathcal E}M[-1-\tfrac d2+w'])$ that
acts by multiplication. 
%
%
%
%
%
%
In these terms, we then have the Leibniz-type rule for the critical case $w=1-\tfrac d2\neq w',w+w'$:
\begin{equation}\label{Leibniz'}
\Lodz_\bsigmas(\bm f \bm f')=
\big((\widetilde\D_\bsigmas+\frac{1}{2w'} \D_\bsigmas) \bm f\big) \bm f' + \bm f \, \Lodz_\bsigmas \bm f' -
\frac{\bm \sigma}{w'}\,  \Langle\hspace{-3.25mm}- \hspace{-1mm}\bm f, \bm f'\Rangle\in \Gamma({\mathcal E}M[-\tfrac d2+w']) \, .
\end{equation}
It is not difficult to use a weight continuation  argument similar to that discussed above (and employing a similar abuse of notation)  that the operator $\widetilde\D_{\hh_{\textstyle\boldsymbol \cdot}}$ and the modified bracket $\Langle\hspace{-3.25mm}- \hspace{-1mm}\cdot,\cdot\Rangle$ combine in this formula to produce a density-valued result.

In the doubly critical case, $w=1-\frac d2= w'$ one has
\begin{align}\label{Leibniz''}
\Lodz_\bsigmas(\bm f \bm f')
&=
\Big(\widetilde\D_\bsigmas
\bm f-\tfrac1{d-2}\hh\D_\bsigmas \bm f\Big) \bm f' + \bm f \, \Big(\widetilde\D_\bsigmas\bm f'-\tfrac1{d-2}\hh\D_\bsigmas \bm f' \Big)+\tfrac{2}{d-2}\hh \
 \bsigma\,  \Langle\hspace{-3.25mm}- \hspace{-1mm}\bm f, \bm f'\hspace{-1mm}- \hspace{-3.25mm}\Rangle
 \\[1mm]
 &\in \Gamma({\mathcal E}M[1-d]) \, ,\nonumber
\end{align}
where the doubly-modified ``bracket'' is defined similarly to above by
$$
  \Langle\hspace{-3.25mm}- \hspace{-1mm}\bm f, \bm f'\hspace{-1mm}- \hspace{-3.25mm}\Rangle=
\big[g;
(\nabla f).\nabla f' 
\big]
\in
\coker(\bm f , \bm f')\, ,
$$
where the $\coker$ notation above means we quotient by the linear span of the images of   $\bm f$ and~$\bm f'$.

Given a pair of scales $\bsigma$ and $\btau$ we may also form the invariant differential operator 
$
\D_{\bsigmas,\btaus}:\Gamma(\ce M[w])\to \Gamma(\ce M[w])
$
defined by
\begin{equation}\label{DD}
\D_{\bsigmas,\btaus}\bm f = \big[g;
\tau(\nabla_n+w\rho_\sigma) f -\sigma (\nabla_k+ w \rho_\tau)f\big]\, ,
\end{equation}
where $k:=\ext \tau$. At weight $w=0$, in the $\btau$ scale this operator maps functions $f$ to $\nabla_n f$, and thus was denoted by $\nabla_{\bm n^{\btaus}}$ 
and dubbed the coupled conformal gradient operator in~\cite{GWvolII}.
Note also that at weight $w\neq 1-\frac d2$ one has
$$
\D_{\bsigma, \btau}= \btau \widehat\D_\bsigma - \bsigma \widehat\D_\btaus\, .
$$
At the critical weight $w=1-\frac d2$,
the above identity still holds (abusing notation as above) upon  replacing $\widehat\D_{\hh_{\textstyle\boldsymbol \cdot}}$ with~$\widetilde\D_{\hh_{\textstyle\boldsymbol \cdot}}$.

\subsection{Distributions and integral theory}

A weight $w=-d$ density $\bm f=[g;f]$ can be invariantly integrated over a conformal $d$-manifold (or some region $D\subset M$ thereof) since the volume element $dV^g$ of $g\in\bm c$ defines a weight $d$ measure-valued density $[g;dV^g]$. Thus we define the conformally invariant integral of $\bm f$ over $D$ by
$$
\int_D\bm f :=\int_D dV^g f\, .
$$

Given a hypersurface $\Sigma\hookrightarrow M$ and a function $f_\Sigma\in C^\infty \Sigma$,  it is propitious to treat the integral of $f_\Sigma$ over $\Sigma$ 
in terms of a defining function $\sigma$ for $\Sigma$. In particular, given $g\in \bm c$, we have (see, for example,~\cite{Osher} or~\cite{GGHW15})
\begin{equation}\label{anumber}
\int_\Sigma dV^{g_\Sigma} f_\Sigma=\int_M dV^{g} |\ext \sigma|_g\, \delta(\sigma) f\, .
\end{equation}
In the above display, $g_\Sigma$ is the metric along $\Sigma$ induced by $g$, $f$ denotes any (smooth) extension of $f_\Sigma$ to $M$ and $\delta(\sigma)$ is the Dirac-delta distribution.

The distributional identity $\delta(\Omega \sigma)=\Omega^{-1}\delta(\sigma)$ (valid for any $0<\Omega\in C^\infty M$) implies that if $\bsigma=[g;\sigma]$ is any weight $w=1$ density, then
$$
\delta(\bsigma):=[g;\delta(\sigma)]
$$
is a weight $w=-1$ distribution-valued conformal density (see~\cite{GWvol} for details). Thus $\delta(\bsigma-\varepsilon \btau)$ where $\varepsilon\in{\mathbb R}$ and $\btau\in\Gamma(\ce M[1])$ gives a one parameter family of weight $w=-1$ densities. Successively differentiating this $k$ times with respect to $\varepsilon$ and subsequently setting $\varepsilon$ to zero establishes that 
$\delta^{(k)}(\bsigma)=[g;\delta^{(k)}(\sigma)]$ is a weight $w=-k-1$ distribution-valued density.
Moreover, the conformally invariant integral of a weight $1-d$ density ${\bm f}_\Sigma=[ g_\Sigma; f|_\Sigma]$ along~$\Sigma$ may be expressed in terms of any smooth extension $\bm f$ of this density via
\begin{equation}\label{deltasurface}
\int_\Sigma\bm f_\Sigma = \int_M \delta(\bm \sigma) \sqrt{\bm{\mathcal S}} \bm f\, .
\end{equation}
This relation reduces to Equation~\nn{anumber} upon expressing it in a choice of metric.


We will  employ 
standard distributional identities (on ${\mathbb R}$) for the Dirac delta $\delta(x)$ and Heaviside step function $\theta(x)$ such as
$$
\theta'(x)=\delta(x)\, ,\quad
x\delta(x) = 0\, ,\quad
x\delta'(x)=-\delta(x)\, ,$$
and
$$
x\delta^{(n)}(x):=x\tfrac{d^n\delta(x)}{dx^n}
=-n\delta^{(n-1)}(x)\, ,\quad n\in {\mathbb Z}_{\geq 1}\, ,
$$
where $n\in {\mathbb Z}_{\geq 1}$.
These are valid when integrating against suitable test functions. In particular we will need to consider the situation where the coordinate $x$ is replaced by a defining function~$\sigma$. Again, this is discussed in detail in~\cite{GWvol,GWvolII}, the key maneuvre is to assume that in a neighborhood of $\Sigma$, the bulk manifold $M$ can be treated as a product $\Sigma\times I$ where $I$ is an open interval about $0$ and the defining function $\sigma$ pulls back to a coordinate $x$ on $I$. Thus, integrals over such neighborhoods can be handled using Fubini's theorem.

\begin{remark}\label{Robinrem}
The distributional calculus is also well adapted hypersurface computations. For example, when $\bsigma$ is a defining density for a hypersurface~$\Sigma$, we have the operator  identity relating the Laplace--Robin and Robin operators 
$$
\delta(\bsigma)\hh \Lodz_\bsigmas = \delta(\bsigma) \bm \updelta^{\sss \Sigma}_{\rm R}\, ,
$$
valid acting on any density of weight $w\neq 1-\frac d2$. At the critical weight we may use the operator $\widetilde \D_\bsigmas$
of Equation~\nn{Lodzhat} to write the identity
$$
\delta(\bsigma)\hh \widetilde \D_\bsigmas= \delta(\bsigma) \hh\bm \updelta^{\sss\Sigma}_{\rm R}\, ,
$$
because $\delta(\bsigma) \Gamma({\mathcal E}M[w])=\delta(\bsigma) \Gamma_\Sigma({\mathcal E}M[w])$.
\end{remark}

One  integral result will play a key role, namely that the Laplace--Robin operator is formally self-adjoint~\cite{GWvol}. Hence, if $M$ is a closed conformal manifold, $\bm f$ a weight $1-d-w$ density and $\bm g$ a weight $w$ density, then
\begin{equation}\label{FSA}
\int_M \bm f\D_\bsigmas \bm g=
\int_M \bm g \D_\bsigmas\! \bm f\, .
\end{equation}
The same conclusion holds if $\bm f$ or $\bm g$ have compact support.
We will use the notation~$\dagger$ for the formal adjoint along $M$, which ignores boundary terms, so that the above equation reads
$$
\D_\bsigmas=\D_\bsigmas^\dagger\, .
$$
The boundary terms are given in~\cite[Theorem 4.3]{GW161}.


\section{Minimal Hypersurfaces for Singular Metrics}
\label{Sec:MinSurf}
Here we begin  with  the data of a sequence of conformal hypersurface embeddings
\begin{equation}\label{embedseq}
\Lambda\hookrightarrow
\Sigma\hookrightarrow (M,\bm c)\, ,
\end{equation}
meaning that $\Lambda$ is a hypersurface in $\Sigma$ and in turn  $\Sigma$ is a hypersurface in $M$.
Then, given a choice of defining density $\bm \sigma$ for $\Sigma$ and thus  a metric $g^o$ that is singular along~$\Sigma$, we consider the problem of determining, at least asymptotically, an oriented hypersurface~$\Xi$ (with boundary)
that meets~$\Sigma$ transversely with intersection $\Lambda=\partial\Xi$,
and such that~$\Xi$ is minimal with respect to $g^o$. 
We will often refer to $\Lambda$ as the anchoring hypersurface/submanifold. This situation is depicted below:

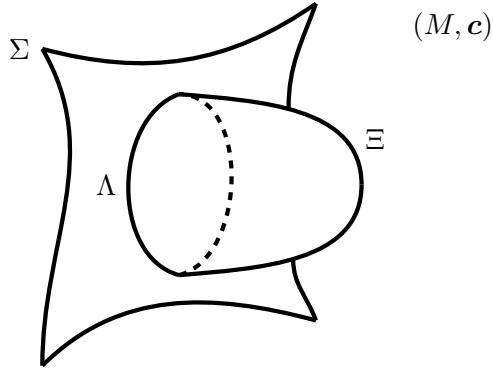
\begin{center}
\begin{tikzpicture}[scale=0.6, ultra thick]
\coordinate (LD) at (-4,-4);
\coordinate (RD) at (2,-3);
\coordinate (LU) at (-4,3);
\coordinate (RU) at (2,4);
\coordinate (CU)at (-1,2);
\coordinate (CD)at (-1,-2);
\coordinate(B) at (3,0);
\coordinate(IU) at (1.4,1.7);
\coordinate(ID) at (1.5,-1.65);
\draw (CU) to[out=-160,in=165] (CD);
\draw [dashed] (CD) to[out=15,in=0] (CU);
\draw (CU) to[out=-5,in=90] (B);
\draw (CD) to[out=5,in=-90] (B);
\draw (LU) to[out=-15,in=-145] (RU);
\draw (LD) to[out=45,in=165] (RD);
\draw (RU) to[out=-110,in=90] (IU);
\draw (RD) to[out=110,in=-90] (ID);
\draw (LU) to[out=-60,in=90] (LD);
\node at (5,3.5) {$(M, \cc)$};
\node at (-4.5,3) {$\Sigma$};
\node at (-2.6,0) {$\Lambda$};
\node at (3.3,1) {$\Xi$};
\end{tikzpicture}
\end{center}
As discussed earlier, 
in the case that the singular metric $g^o$ is Poincar\'e--Einstein, this problem has been studied by Graham and Witten~\cite{GrahamWitten} (see also~\cite{GrahamRiechert})
using a different approach.

A minimal hypersurface $\Xi$ is characterized by the vanishing of its mean curvature~$H_\Xi^g$ with respect to the ambient metric $g$. Our treatment of minimal surfaces relies on formulating this condition in terms of defining densities.
This is achieved by the next proposition. In what follows we use the notation $\bm A\stackrel{\bm \mu}\sim \bm B$ when smooth densities $\bm A$, $\bm B$ and~$\bmu$ obey $\bm A=\bm B+\bmu \bm C$ for some smooth density $\bm C$.
\begin{proposition}\label{MEAN}
Let $\bmu$ and $\bsigma$ be defining densities for  embedded hypersurfaces~$\Xi$ and~$\Sigma$, respectively, 
Then, away from $\Lambda$, the mean curvature of $\Xi$ 
with respect to the 
metric
$g^o=\bm g/\bsigma^2$ is given~by
\begin{equation}
\label{HXI}
H_\Xi^{g^o}=- \left. \frac{\bM^\bsigmas_\bmus}{\sqrt{\bS_\bmus}}\, \right|_\Xi\, ,
\end{equation}
where
\begin{equation}\label{minimalM}
\bM^\bsigmas_\bmus=\langle\bsigma,\bmu\rangle +
\frac{1}{2(d-1)(d-2)}\,\bsigma \D_\bmus \log\bS_\bmus\, ,
\end{equation}
and hence $H_\Xi^{g^o}$ is extended smoothly to  $\Lambda$ by the right hand side of Equation~\nn{HXI}.
Moreover, if $f$ is any smooth, strictly positive 
function, 
then
$$
\bM^{\hh\hh\bsigmas}_{f\!\hh\bmus} \stackrel\bmuss\sim  f\bM^\bsigmas_\bmus\:
\mbox{ and }\:
\bS_{f\!\hh\bmus}
\stackrel\bmuss\sim
f^2 \bS_{\bmus}\, .
$$
\end{proposition}

\begin{proof}
Given any defining function $\mu$ and metric $g^o$, the mean curvature of $\Xi\hookrightarrow (M,g^o)$ is given by 
\begin{equation}\label{mean}
H^{g^o}_\Xi=\frac{\: \,  \nabla^{g^o}\!\!.\, \hat m\hh \big|_\Xi\: }{d-1}\,   ,
\end{equation}
where $\hat m$ is the extension of the unit conormal to $\Xi$  given by
$$
\hat m=
\frac{\ext \mu\,\: }{\, |\ext \mu|_{g^o}}\, .
$$
Now, the divergence of $\hat m$ with respect to the Levi--Civita connection of $g^o=g/\sigma^2$ is related to that of the metric $g$ via 
$$
g_o^{ab}\nabla^{g^o}_a\hat m_b^{\phantom{G}}=\sigma^2\, g^{ab}\nabla_a\hat m_b -(d-2)\sigma\,  g^{ab}n_a\hat m_b \, .
$$
Calling $m=\ext \mu$, we have $\hat m=\sigma^{-1} m/|m|_g$ so that
$$
g_o^{ab}\nabla^{g^o}_a\hat m_b^{\phantom{G}}=\frac1{\, |m|_g}\, \Big(\sigma g^{ab} \nabla_a m_b - (d-1)g^{ab} n_a m_b
-\frac12\, \sigma g^{ab} m_a  \nabla_b \log|m|^2_g\Big)\, .
$$
Hence we have (using an index-free notation now that all explicit metric dependence is through $g$)
$$
H^{g^o}_\Xi\, \stackrel\mu\sim\, -\, \frac{\, m.n+\sigma\rho_\mu+\mu\rho_\sigma
+\frac1{2(d-1)(d-2)}\, \sigma \big((d-2)\nabla_m-\mu  \Delta\big) \log\big(|m|^2_g+2\rho_\mu\mu\big)\, 
}{\sqrt{|m|_g^2+2\rho_\mu \mu\, }}\, .
$$
 Comparing the above display with the definition of the bracket $\langle\cdot,\cdot\rangle$ in Equation~\nn{thebracket},
the ${ {\mathcal S}}$-curvature in Equation~\nn{Scurvy} and the Laplace--Robin operator in Equation~\nn{Ldef},
gives the result claimed in the first two displays of the proposition.

To prove  the second claim, we could  rely on the fact that  Equation~\nn{mean} gives the mean curvature for {\it any} defining function $\mu$ and only compute the homogeneity of the ${ {\mathcal S}}$-curvature. Here, we give a  detailed proof to further develop our hypersurface calculus. First we note that
for any weight $w=1$ densities $\bm \mu$ and $\bm \sigma$, given a smooth function $f$, it is not difficult (using that, away from critical weights, $\Langle \bsigma,\cdot\hh\Rangle=\Lodz_\bsigma\cdot$ and the Leibniz rule~\nn{Leibniz})  to verify that
$$
\langle\bsigma ,f\bmu\rangle=
f\langle\bsigma ,\bmu\rangle
+\bmu \Langle \bsigma,f\Rangle
-\tfrac{2}{d}\, \bsigma\Langle f,\bmu\Rangle
\stackrel{\bmuss}\sim
f\big(\langle\bsigma ,\bmu\rangle-\tfrac2d \bsigma\,  \Lodz_\bmus \!\log f\big)
\, .
$$
Thus,
since $\bm{\mathcal S}_\bmus=\langle\bmu,\bmu\rangle$ it follows, applying the above display twice, that
$$\bS_{f\!\hh\bmus}:=
\langle f\bmu ,f\bmu\rangle
=
f^2 \langle\bmu ,\bmu\rangle
+\tfrac{2(d-2)}d \bmu f
\Langle \bmu, f\Rangle
+{\mathcal O}(\bmu^2)
\stackrel\bmuss\sim
f^2 \bS_{\bmus}\, ,$$
which gives the claimed result for the ${ {\mathcal S}}$-curvature. 
The above display and the additive property for logarithms of products of densities allows us to compute
\begin{equation*}
\begin{split}
\D_{f\bmus} \log \bS_{f\!\hh\bmus}&\stackrel{\bmuss}\sim
f\D_\bmus \Big[\log \bS_\bmus+2\log f+\log\Big(1+
\tfrac{2(d-2)}{d} \, \tfrac{\bmu}{\bS_\bmus} \Lodz_\bmus\log f\Big)\Big]\\
&\stackrel{\bmuss}\sim
f\D_\bmus\log  \bS_\bmus +\tfrac{4(d-1)(d-2)}d\, f\Lodz_\bmu \log f\, .
\end{split}
\end{equation*}
Here we have  used that $\log \bS_{f\!\hh\bmus}$ is a weight zero density and that $\D_\bmus$ is a derivation along the zero locus of $\bmu$.
Combining the last and next to last displays gives the required result.
\end{proof}

Given a hypersurface $\Sigma\hookrightarrow (M,\bm c)$, we may always find a conformal unit defining density $\bmu$ solving the singular Yamabe problem $\bS_\bmus=1+{\mathcal O}(\bmu^d)$. 
In that case the minimal surface condition 
\begin{equation}
\label{minimal}\bM^\bsigmas_\bmus\stackrel\bmuss\sim0
\end{equation} 
simplifies to 
$$\langle\bsigma,\bmu\rangle\stackrel\bmuss\sim\ 0\, .
$$
This implies that if a minimal hypersurface $\Xi$ for a singular metric $g^o=\bm g/\bsigma^2$ intersects the zero locus $\Sigma$  of $\bsigma=[g;\sigma]$, it does so at right angles. 

\medskip

In general the minimal surface condition~\nn{minimal} for a singular metric $g^o$ determined by~$\bsigma$ cannot be solved smoothly, so instead we solve this problem asymptotically, in the following sense:

\begin{problem}\label{minprob}
Let $\bsigma$ be a defining function for $\Sigma={\mathcal Z}(\bsigma)\hookrightarrow (M,\bm c)$ and let  
$\Lambda\hookrightarrow \Sigma$ be an embedded hypersurface. 
Find a 
hypersurface $\Xi$ such that
$$
H_\Xi^{g^o}=\bsigma^k \bm A|_\Xi\, ,
$$
where $g^o$ is the singular metric determined by $\bsigma$, 
$\Lambda=\Xi\cap \Sigma$, the density $\bm A$ is smooth, and the order $k\in \mathbb Z_{\geq 1}$
is as high as possible.
\end{problem}


\noindent
The key to solving this problem is the following lemma:
\begin{lemma}\label{keylemma}
Let $k\in \mathbb Z_{\geq 1}$ and
$\bmu'=\bmu+\bsigma^{k+1}\bmu_{k+1}$, for any
$\bmu_{k+1}\in \Gamma(\ce M[-k])$.
Then
$$
\bM^\bsigmas_{\bmus'}
\stackrel{\bmuss'}\sim
\bM^\bsigmas_\bmus \, \big(1+{\mathcal O}(\bsigma^{k})\big)
+\frac{(k+1)(d-1-k)}{d-1}\,
\bS_\bsigmas\hh \bsigma^k\hh  \bmu_{k+1}+{\mathcal O}(\bsigma^{k+1})\, .
$$
\end{lemma}
\begin{proof}
The proof is an elementary application of the Leibniz rules developed in Section~\ref{LRs}. The details are as follows: First we use the 
$\mathfrak{sl}(2)$ identity obeyed by the Laplace--Robin operator in~\nn{envelop}
to compute
\begin{equation}\label{newangle}
\langle\bsigma,\bmu'\rangle-\langle\bsigma,\bmu\rangle=
\tfrac1d\D_\bsigma (\bsigma^{k+1}\!\bmu_{k+1})
=\tfrac{(k+1)(d-k)}d\, \bsigma^k\bS_\bsigmas \bmu_{k+1}
+{\mathcal O}(\bsigma^{k+1})
\, .\end{equation}
Then we note that acting on a weight $w=0$ 
density $\bm f$, we have
$$
\D_{\bmus'}\bm f
-\D_\bmus \bm f=
(k+1)\, \bsigma^k\bmu_{k+1} \D_\bsigmas
 \bm f
 +{\mathcal O}(\bsigma^{k+1})
\, .
$$
This identity is easily established by examining the definition of the Laplace--Robin operator in Equation~\nn{Ldef}.
Thus, because $\bsigma \D_{\bmus} \log \bS_{\bmu}$ is itself a weight $w=0$ density,
\begin{equation}\label{Llog}
\bsigma \D_{\bmus'} \log \bS_{\bmu'}
\!-\bsigma \D_{\bmus} \log \bS_{\bmu}
=
\bsigma \D_{\bmus} \log\Big( \frac{\bS_{\bmu'}}{\bS_\bmu}\Big)
+{\mathcal O}(\bsigma^{k+1})\, .
\end{equation}
Moreover, using
$$
\Lodz_\bmu \bsigma^{k+1}=(k+1)\,  \bsigma^k \langle \bsigma,\bmu\rangle-\tfrac{k(k+1)}{d+2k} \,  \bsigma^{k-1}\bS_\bsigmas\,  \bmu\, ,
$$
and
$$
\Langle
\bsigma^{k+1},\bmu_{k+1}\Rangle
=(k+1)\, \bsigma^k \, \Lodz_\bsigmas \bmu_{k+1}
+\tfrac{k^2(k+1)}{d+2k}\, \bsigma^{k-1}  
\bS_\bsigma\, \bmu_{k+1}\, ,
$$
we have
\begin{equation}
\begin{split}
\bS_{\bmus'}-\bS_\bmu&=
2\Langle\bmu,\bsigma^{k+1}\bmu_{k+1}\Rangle
+
\langle\bsigma^{k+1}\bmu_{k+1},\bsigma^{k+1}\bmu_{k+1}\rangle
\\[1mm]
&=
2\bmu_{k+1}
\big((k+1)\,  \bsigma^k \langle \bsigma,\bmu\rangle-\tfrac{k(k+1)}{d+2k} \,  \bsigma^{k-1}\bS_\bsigmas\,  \bmu\big)
+
2\bsigma^{k+1}\Lodz_\bmu \bmu_{k+1}
\\&
-\tfrac4d\, 
\bmu
\big(
(k+1)\, \bsigma^k \, \Lodz_\bsigmas \bmu_{k+1}
+\tfrac{k^2(k+1)}{d+2k}\, \bsigma^{k-1}\, 
\bS_\bsigma\,   \bmu_{k+1}
\big)
+{\mathcal O}(\bsigma^{2k})
\\[1mm]
&=
-\bmu\big(\tfrac{2k(k+1)}{d} \,  \bsigma^{k-1}\bS_\bsigmas\, 
\bmu_{k+1} + {\mathcal O}(\bsigma^k)\big)
+\bM^\bsigmas_\bmus \, {\mathcal O}(\bsigma^{k})
+{\mathcal O}(\bsigma^{k+1})\, .
\end{split}
\label{S'}
\end{equation}
Now, acting on weight $w$ densities---again thanks to the Laplace--Robin $\mathfrak{sl}(2)$ algebra---we have 
the operator identity 
$$\D_\bmus \circ\,  \bmu \stackrel{\bmuss'}\sim (d+2w)\, \bS_\bmus
+{\mathcal O}(\bsigma^{k+1})\, ,$$
so that Equations~\nn{Llog} and~\nn{S'} imply
$$
\bsigma \D_{\bmu'}\log \bS_{\bmu'}
\stackrel{\bmuss'}\sim
\bsigma \D_{\bmu}\log \bS_{\bmu}
-\tfrac{2k(k+1)(d-2)}{d}\, \bsigma^k \bS_\bsigma\,  
\bmu_{k+1} 
 +\bM^\bsigmas_\bmus \, {\mathcal O}(\bsigma^{k})
 +{\mathcal O}(\bsigma^{k+1})\, .
$$
Here we used that $\D_\bmu\circ \, \bsigma^k
\stackrel{\bmuss'}\sim {\mathcal O}(\bsigma^{k-1})$ and the identity
$\log (A/B)=\log(1+\frac{A-B}B)$.
Employing Equation~\nn{newangle} and the above display  to compute $\bM^\bsigmas_{\bmus'}$ as defined by Equation~\nn{minimalM} gives the quoted result.
\end{proof}

The above lemma is the basis for an iterative solution to the minimal surface condition~\nn{minimal}.
First consider a defining density for a hypersurface that meets $\Sigma$ transversely along $\Lambda$. Working locally, it is straightforward to improve this to 
a 
defining function~$\bmu_0$ for a hypersurface~$\Xi_0$
such that $\Sigma$ and $\Xi_0$ 
intersect along $\Lambda$ at right angles. Moreover, without loss of generality, assume  that~$\bmu_0$ is a conformal unit defining density, so that $\bS_{\bmus_0}=1+{\mathcal O}(\bmu_0^d)$. Then
$$
\bM^\bsigmas_{\bmu_{_0}}
\stackrel{\bmuss_{0}}\sim 
\langle\bsigma,\bmu_0\rangle=\tfrac1d \D_\bsigmas \bmu_0
\stackrel{\bmuss_{0}}\sim {\mathcal O}(\bsigma)\, ,
$$
since $\langle\bsigma,\bmu_0\rangle$ vanishes along $\Lambda$. Hence, we consider an improved defining density
$$
\bmu=\bmu_0+\bsigma^2
\bmu_2 \, .
$$
 By the above lemma 
 we should choose $\bmu_2$ that solves 
$$
\tfrac 1d \D_\bsigmas \bmu_0+\tfrac{2(d-2)}{d-1} \bS_\bsigmas \bsigma \bmu_2
\stackrel{\bmuss}\sim {\mathcal O}(\bsigma^2)\, .
$$
It is not difficult to verify that $\D_\bsigmas (\bmu \bm f)\stackrel{\bmuss}\sim {\mathcal O}(\bsigma)$ for any density $\bm f$. Thus, dividing the above display by $\bS_\bsigma$
(as remarked earlier this is   well defined, at least in some neighborhood of $\Sigma$) and then acting with $\D_\bsigmas$, with the help of the $\mathfrak{sl}(2)$ algebra (see Section~\ref{LapRob}) we find
$$
\tfrac 1d\hh
\D_\bsigma  \circ \, \bS_\bsigma^{-1}\circ \D_\bsigma \bmu_0+\tfrac{2(d-2)^2}{d-1}\, \bS_\bsigmas
\bmu_2
\stackrel\bmuss\sim {\mathcal O}(\bsigma)\, .
$$
Hence we have proved the following Lemma:
\begin{lemma}\label{rightangs}
Let $\bmu_0$ be a defining density for a hypersurface $\Xi_0$ that intersects $\Sigma={\mathcal Z}(\bsigma)$ at right angles. Then the density 
$$
\bmu=\bmu_0-\tfrac{d-1}{2d(d-2)^2}\, \bsigma^2\, (\bS_\bsigma^{-1}\circ \D_\bsigma)^2 \bmu_0\, 
$$
obeys
$$
\bM^\bsigmas_{\bmu}
\stackrel{\bmuss}\sim {\mathcal O}(\bsigma^2)\, .
$$
\end{lemma}

%

%

The preceding two lemmas are the induction and base steps that establish the following theorem:

\begin{theorem}\label{MINSOL}
Given the conformal embedding data $\Lambda\hookrightarrow \Sigma\hookrightarrow (M,\cc)$ and a  defining density $\bsigma$ for~$\Sigma$, there exists a conformal unit defining density~$\bmu$ such that  
\begin{equation}\label{minimu}
\bM^\bsigmas_{\bmus}
\stackrel{\bmuss}\sim \bsigma^{d-1} \bB\, .
\end{equation}
Moreover, the weight $w=1-d$ density along  
$$
{\bm B}_\Lambda:=\bB|_\Lambda
$$
is uniquely determined by the above data.
\end{theorem}


The quantity $
{\bm B}_\Lambda$ obstructs smooth solutions to the singular minimal hypersurface problem and is therefore termed the {\it minimal obstruction density}. It is an invariant of the conformal embedding data $\Lambda\hookrightarrow \Sigma\hookrightarrow (M,\cc)$ and the defining density $\bsigma$. 
Then if~$\bsigma$ 
is determined,  to sufficient order, in terms of these embeddings via a suitable problem it follows that ${\bm B}_\Lambda$ is determined in terms of the  conformal embeddings.
For example, the singular Yamabe problem determines $\bsigma$ 
modulo terms of order $\bsigma^{d+1}$, and so achieves this.
In the special case that $(M,g^o)$ is Poincar\'e--Einstein,~$
{\bm B}_\Lambda$ is an invariant of $\Lambda\hookrightarrow (\Sigma,\cc_\Sigma)$. This is another hierarchy in the same vein as that discussed in the introduction. Moreover, for Poincar\'e--Einstein structures, it is known that ${\bm B}_\Lambda$  is the functional gradient of the area anomaly for the minimal surface~$\Xi$~\cite{GrahamRiechert}.


\begin{remark}
A  useful tool for computations with Poincar\'e--Einstein structures is to choose a canonical metric $g\in \cc$. 
Writing $\sigma$ to denote the function in $\bsigma$ determined by $g$, which may be  taken to equal the arc length $x$ for a suitable geodesic shooting problem, one obtains the Graham--Lee normal form~\cite{GL} for the singular metric
$$
g^o=\frac{dx^2 + h(x)}{x^2}\, .
$$
The Poincar\'e--Einstein condition then gives that $h(x)$ has an even expansion in $x$. This facilitates a simple proof of vanishing theorems for anomalies and obstructions in appropriate dimension parities. For example, this implies the area anomaly and hence the minimal obstruction density $
{\bm B}_\Lambda$ vanishes in this case when $d$ is odd~\cite{GrahamWitten,GrahamRiechert}. Noting that the Laplace--Robin operator is odd under the interchange $x\leftrightarrow -x$, when expressed in the Graham-Lee normal form, it is not difficult to check that  Lemmas~\ref{keylemma} and~\ref{rightangs} also lead to an even expansion in $x$ for the function  $\mu$ corresponding to $\bmu$, and in turn vanishing obstruction for $d$ odd.
\end{remark}

We will employ the term {\it minimal defining density}  for densities $\bmu$ obeying the minimal condition~\nn{minimu} given the data
 $\Lambda\hookrightarrow \Sigma\hookrightarrow (M,g^o)
 $ where the singular metric $g^o$ is determined by the unit defining density $\bsigma$ (given the conformal class). When we need to emphasize  the choice of $\bsigma$, we will use the term {\it $\bsigma$-minimal}.
 This condition  restricts the zero locus~${\mathcal Z}(\bmu)$
 to be a hypersurface~$\Xi$ that solves Problem~\ref{minprob} to order $k=d-1$. 
 We will term such a hypersurface an {\it asymptotically minimal hypersurface}. 
 When~$\bmu$ is also chosen to further obey the singular Yamabe condition
$\bS_\bmu=1+{\mathcal O}(\bmu^d)$, we term~$\bmu$ a {\it minimal unit defining density}. A minimal defining density may always be improved to a minimal unit one (while keeping~$\Xi$ in the  zero locus) with the same zero locus~$\Xi$ (see Section~\ref{SINGYAM} and~\cite{GW15,GW161}).  

Equation~\nn{minimu} implies that (generically) the failure of $\bM^\bsigmas_{\bmus}$ to vanish along $\Sigma$ is proportional to the minimal defining density~$\bmu$;
%
%
%
 this leads to another invariant of the embedding data~$\Lambda\hookrightarrow \Sigma\hookrightarrow (M,\cc)$, which we record in the following lemma. 
 
\begin{lemma}\label{C}
Let  the conformal data $\Lambda\hookrightarrow \Sigma\hookrightarrow (M,\cc)$ be given and let  $\bsigma$ be a corresponding unit defining density for $\Sigma$.
Moreover let $\bmu$ be a $\bm \sigma$-minimal unit defining density for $\Xi$  with~$\partial \Xi =\Lambda$. Then 
$$
\bM^\bsigmas_{\bmus}
= \bmu \bC + \bsigma^{d-1}  \bB
$$
for some $\bC\in \Gamma(\ce M[-1])$,
and $\bm C_\Lambda:=\bC|_\Lambda$ is a uniquely defined invariant of the embedding data $\Lambda\hookrightarrow \Sigma\hookrightarrow (M,\cc)$ given by
$$
\bm C_\Lambda= \, \bm\IIo^\Sigma_{\bm{ \hat\bm  m}   {\bm{\hat  m}}}|_\Lambda\, , 
$$
where ${\bm{ \hat m}}$ is the unit conormal for the hypersurface embedding $\Lambda \hookrightarrow \Sigma$. 
%
%

\end{lemma}

\begin{proof}
Together, the minimal unit and unit defining density definitions for $\bmu$ and $\bsigma$ imply
\begin{equation}\label{thebr}
\bM^\bsigmas_{\bmus}
= \langle\bsigma,\bmu\rangle +\bm \sigma \bmu^{d-1} \bC^\prime
=
\bmu \bC + \bsigma^{d-1}  \bB
\, ,
\end{equation} 
for some smooth, weight $w=-1$ density $\bC$ and some weight $-d$ density $\bC'$ (coming from the log term in Equation~\nn{minimalM}). 
Uniqueness of $\bm C_\Lambda=\bm {\mathcal C}|_\Lambda$ is guaranteed by Theorem~\ref{MINSOL} and the uniqueness property of  unit defining density solutions to the singular Yamabe problem. In particular this determines $\bmu$ uniquely to ${\mathcal O}(\bmu^d)$, which suffices for uniqueness of $\bC|_\Lambda$.

Now let us denote   $\bmu=[g;\mu]$, $\bsigma=[g;\sigma]$, $m:=\ext \mu$, $n:=\ext \sigma$. From the minimal unit and unit defining density definitions (see Equations~\nn{Scurvy},~\nn{thebracket} and~\nn{minimalM})
we then have
\begin{equation}\label{ones}
m^2+2\mu \rho_\mu = 1 + {\mathcal O}(\mu^d)\, ,\quad
n^2+2\sigma \rho_\sigma = 1+ {\mathcal O}(\sigma^d)\, ,
\end{equation}
and
\begin{equation}\label{mn}
m.n+\mu \rho_\sigma + \rho_\mu \sigma =\mu\hh  {\mathcal C} +{\mathcal O}(\sigma\mu^{d-1}) 
+{\mathcal O}(\sigma^{d-1})\, .
\end{equation}
Here ${\mathcal C}$ 
denotes  $\bC$ 
 in the scale $g$. 
Differentiating the last display along the conormal $m$ and restricting to $\Lambda$ (which we indicate by a superscript $\Lambda$ above the equals sign) we have
$$
\nabla_m(m.n) + \rho_\sigma \stackrel\Lambda= {\mathcal C}\,  .
$$
Here we used $\nabla_m \mu = m^2 \stackrel\Lambda = 1$
and $\nabla_m \sigma = m.n \stackrel\Lambda = 0$.
Now in general for unit defining densities (see~\cite[Lemma 3.3]{GW161}) one has
\begin{equation}
\label{singyamIIo}
\rho_\sigma \stackrel\Sigma= -H^\Sigma\, ,\qquad \nabla_a n_b \stackrel\Sigma= \IIo^\Sigma_{ab} + H^\Sigma g_{ab}
\, ,
\end{equation}
and thus $\nabla_m m_a\stackrel\Xi= m_a H^\Xi$.
Hence, since $m.n\stackrel\Lambda=0$ we have
$$
{\mathcal C}\stackrel\Lambda = \IIo_{mm}^\Sigma\, .
$$
Along $\Lambda$ we have that $m=\hat m$; this gives the quoted result for ${\mathcal C}|_\Lambda$.
\end{proof}

There are various conformally invariant relations obeyed by extrinsic quantities associated to an asymptotically minimal hypersurface~$\Xi$ along its boundary $\Lambda$. The first of these was discussed above, namely that unit conormals $\bm{\hat m}$ and $\bm {\hat n}$ of $\Xi$ and $\Sigma$ obey~$ \bm{\hat m}.\bm {\hat n}\stackrel\Lambda=0$. Another example is given in the following Lemma:

\begin{lemma}\label{IIonn}
Let $\Xi$ be an asymptotically minimal hypersurface determined by the sequence of embeddings~\nn{embedseq} where $\hat {\bm n}$  is the unit conormal to $\Sigma$. Then  the trace-free second fundamental form of $\Xi$ obeys 
\begin{equation}\label{aftervector}
\bm\IIo^\Xi_{\bm{ \hat\bm  n}   {\bm{\hat  n}}}\big|_\Lambda=0\, .
\end{equation}

\end{lemma}

\begin{proof}
This is a corollary of Equation~\nn{Rodsequation} below, but a direct proof can also be given:
Without loss of generality we can take $\Xi$ to be the zero locus of a minimal unit defining density $\bmu$. This means that we can use Equations~\nn{ones} and~\nn{mn} of the previous proof. So now 
we compute
$$
n^a n^b (\nabla_a m_b + \rho_\mu g_{ab})
=\nabla_n(m.n) -n^a m^b \nabla_a n_b + n^2 \rho_\mu\, .
$$
Along $\Lambda$, the left hand side above equals $\IIo^\Xi_{\hat n\hat n}$ while for the right hand side we find
$$
\nabla_n\big (-\sigma \rho_\mu - \mu \rho_\sigma + \mu {\mathcal C} +
{\mathcal O}(\sigma\mu^{d-1})+
{\mathcal O}(\sigma^{d-1})\big)+\rho_\mu
\stackrel\Lambda= 0\, .
$$
Here we have used Equation~\nn{singyamIIo}
and $m.n|_\Lambda=0$ to show that $n^a m^b \nabla_a n_b |_\Lambda=0$. Similarly $\nabla_n \mu|_\Lambda=0$. This completes the proof.
\end{proof}

\begin{remark}
When the singular metric $g^o=\bm g/\bsigma^2$ is Poincar\'e--Einstein, the hypersurface~$\Sigma$ is necessarily umbilic~\cite{LeB-Heaven,Goal}, so the invariant $ \bC|_\Lambda$ then vanishes.
Also note that Equation~\ref{aftervector} implies that $\bm{\hat n}^a\bm\IIo^\Xi_{ab}|_\Lambda$ is a covector tangent  to $\Lambda$.
\end{remark}

The following Lemma explains the  significance of the 
result in Equation~\nn{aftervector}, and in particular demonstrates that
along $\Lambda$, the mean curvatures of $\Xi$  equals that of $\Lambda$ when~$\Xi$ is an asymptotically minimal hypersurface determined by the embedding sequence~\nn{embedseq}.

\begin{lemma}\label{H2H}
Let $\Xi$ and $\Sigma$ be hypersurfaces in a Riemannian $d$-manifold $(M,g)$ that intersect at right angles along a submanifold $\Lambda$. Then the following relation on mean curvatures holds along $\Lambda$, 
$$
H_{\Lambda\hookrightarrow (\Sigma,g_\Sigma)}=H_{\Xi\hookrightarrow(M,g)}-\tfrac1{d-2}\, \IIo^{\Xi\hookrightarrow(M,g)}_{\hat n \hat n}\, ,
$$
where $g_\Sigma$ is the induced metric on $\Sigma$, and $\hat n$ is the unit conormal to $\Sigma$.
\end{lemma}

\begin{proof}
Let $m_a$ be any extension to $M$ of the unit conormal $\hat m_a$ of $\Xi$. Then the mean curvature of $\Lambda\hookrightarrow (\Sigma,g_\Sigma)$ is given along by
$$
\frac{1}{d-2}\, \nabla^\Sigma_a \Big(\frac{m^a}{|m|}\Big|_\Sigma\Big)=\frac{1}{d-2}\left.\left((\nabla_a - n_a\nabla_n)
\Big(\frac{m^a}{|m|}\Big)\right)\right|_\Sigma\, ,
$$
where $n$ is any extension of the unit conormal of $\Sigma$. Without loss of generality (see~\cite[Proposition 2.5]{GW15}) we may choose $m=d\mu$ where $|m|_g=1$ in $M$
({\it i.e.}, $\mu$ is a {\it unit defining  function} for $\Xi$). This further implies that $\II^\Xi_{ab}$ equals $\nabla_a m_b$ along $\Xi$. 
Hence the above display becomes
$$
\frac{1}{d-2}\Big((d-1)H_{\Xi\hookrightarrow(M,g)}
-\II_{\hat n\hat n}^{\Xi\hookrightarrow(M,g)}\Big)\, .
$$
The proof is completed by using that $\II^{\Xi\hookrightarrow(M,g)}_{ab}=\IIo^{\Xi\hookrightarrow(M,g)}_{ab}+g^\Xi_{ab}H_{\Xi\hookrightarrow(M,g)}$.

\end{proof}

\medskip

We now turn to the problem of computing 
the minimal obstruction density. 
This may be achieved using a variant the idea  behind the recursion developed in~\cite{GW15} for calculating the obstruction to smoothly solving the singular Yamabe problem. 
The key is to compute $d-1$ derivatives normal to $\Sigma$ of the canonical extension $\bM_\bmus^\bsigmas$ of the 
mean curvature of the singular metric defined in Equation~\nn{HXI}.
When $\bmu$ is chosen to solve the singular Yamabe problem we may instead study $\langle \bsigma,\bmu\rangle$ (see Equation~\nn{thebr}).
The following result is underlies a recursion for that computation.

\begin{lemma}\label{recurse}
Let $\bsigma=[g;\sigma]$ be a unit defining density, $\bmu=[g;\mu]$ be a $\bsigma$-minimal unit defining density, and suppose $k\in \{1,\ldots,d-1\}$. Then for a given $g\in \cc$ and $k>1$,
$$
(\nabla_n^\top)^k\langle \bsigma,\bmu\rangle + (d-k-1) (\nabla_n^\top)^{k-1}\rho_\mu \stackrel\Lambda = -(\nabla_n^\top)^{k-1} \big(I^{ab}\nabla_a m_b\big)-(k-1)(\nabla_n^\top)^{k-2}A+ {\rm LOT} ,
$$
where $m=\ext \mu$, $n=\ext \sigma$, $I_{ab}:=g_{ab}-m_a m_b-n_a n_b$ smoothly extends the first fundamental form $I^{\Lambda}_{ab}=g_{ab}^\Lambda$ to $M$, $\nabla^\top = \nabla - m \nabla_m$, 
$A=\nabla_m \rho_\sigma-\rho_\mu\rho_\sigma
- \rho_\mu m^a\nabla_m n_a$, and ``${\rm LOT}$'' 
denotes terms involving at most $k-2$ powers of the operator $\nabla_n^\top$ acting on $\rho_\mu$.
When $k=1$, the result is
$
(d-2) \rho_\mu \stackrel\Lambda= -I^{ab}\nabla_a m_b
$.
%
%
 \end{lemma}
 
 \begin{proof}
 The weight zero density $\langle \bsigma,\bmu\rangle$ is given in any
 scale $g\in \cc$ by
 $$
 m.n +\sigma \rho_\mu + \mu \rho_\sigma\, .
 $$
 However, along $\Xi$ we have 
 $
 \nabla_n^\top \mu |_\Xi = 0$.
 Hence we only need study powers of $\nabla^\top_n$ acting on $m.n+\sigma \rho_\mu$. 
 Moreover, by assumption, Equations~\nn{ones} and~\nn{mn} hold.
 It therefore follows that
 \begin{multline}
\nabla_n^\top (m.n) - n^a n^b \nabla_a^\top m_b = m^a \nabla_n^\top n_a
\stackrel \mu \sim
\frac 12 \nabla_m n^2
+\sigma \rho_\mu m^a\nabla_m n_a
+{\mathcal O}(\sigma^{d-1})
\\[1mm]
\stackrel \mu \sim
-\sigma (\nabla_m \rho_\sigma
-\rho_\mu\rho_\sigma
- \rho_\mu m^a\nabla_m n_a)
+{\mathcal O}(\sigma^{d-1})
\, .
\end{multline}
The second term on the left hand side of the previous display gives
$$
-n^a n^b \nabla_a^\top m_b =  I^{ab} \nabla_a^\top m_b
- \nabla_a^\top m^a + m^a m^b \nabla_a^\top m_b\, .
$$
Observe that
$
\nabla_a^\top m^a = \nabla_a m^a - \frac 12 \nabla_m m^2
\stackrel\mu\sim-(d-1)\rho_\mu
\mbox{ 
and }
m^a m^b \nabla_a^\top m_b \stackrel\mu\sim 0\, ,
$
 because $m^a \nabla_a^\top = \nabla_m - m^2 \nabla_m$ and $m^2=1-2\mu\rho_\mu + {\mathcal O}(\mu^d) \stackrel\mu\sim 1$. Thus we have
 $$
 \nabla_n^\top (m.n)
 \stackrel\mu\sim
- I^{ab} \nabla_a^\top m_b
 -(d-1)\rho_\mu-\sigma 
 (\nabla_m \rho_\sigma
-\rho_\mu\rho_\sigma
- \rho_\mu m^a\nabla_m n_a)
+{\mathcal O}(\sigma^{d-1})
\, .
 $$
 So far we have found that
 \begin{equation}\label{sofarsogood}
 (\nabla_n^\top)^k\langle \bsigma,\bmu\rangle\stackrel\Lambda=
  (\nabla_n^\top)^{k-1}
  \big(- I^{ab} \nabla_a^\top m_b
 -(d-1)\rho_\mu-\sigma A
  \big)+
(\nabla_n^\top)^k(\sigma\rho_\mu)
  \, .
 \end{equation}

Now $\nabla_n^\top \sigma=
n^2 -(m.n)^2\stackrel\mu\sim1-2\sigma\rho_\sigma-\sigma^2 \rho_\mu^2
+{\mathcal O}(\sigma^d)$.
 Hence along $\Lambda$,
$(\nabla_n^\top)^{k-1}(\sigma A)$
can be replaced with $ (k-1)(\nabla_n^\top)^{k-2} A$ modulo terms involving fewer than $k-2$ powers of~$\nabla_n^\top$ acting on $\rho_\mu$. Similarly, $(\nabla_n^\top)^k (\sigma \rho_\mu)$  can be replaced with
$k (\nabla_n^\top)^{k-1} \rho_\mu$  modulo terms involving at most $k-2$ powers of $\nabla_n^\top$ acting on $\rho_\mu$. 
The $k=1$ case follows directly from the above display and explanation.
 \end{proof}

\begin{remark}
Observe, that when $k<d-1$, the first term displayed in the above lemma $(\nabla_n^\top)^k\langle \bsigma,\bmu\rangle\stackrel\Lambda = 0$, because $\bmu$ is a minimal unit defining density. Hence, for these $k$,  the lemma  determines then $(\nabla_n^\top)^{k-1}\rho_\mu $ in terms of terms involving lower powers of $\nabla_n^\top$ acting on $\rho_\mu$. When $k=d-1$, we have
$$
(\nabla_n^\top)^k\langle \bsigma,\bmu\rangle\stackrel\Lambda = (d-1)! \, \bm B_\Lambda\, ,$$ so the lemma instead determines the minimal obstruction density.
\end{remark}

The above recursion can be used to compute  the minimal obstruction density for three-manifolds, as given in the following proposition.

\begin{proposition}\label{B3}
Given the data $\Lambda\hookrightarrow \Sigma\hookrightarrow (M,\cc)$, where $M$ is a three-manifold,  
%
%
the minimal obstruction density
${\bm B}_\Lambda$ is given by 
$$
 {\bm B}_\Lambda=
\left[g; \left(\nabla^{b \Lambda}\big(\hat m^a\IIo^{\Sigma}_{ac}I^c_b\big)
+\tfrac12 \big(\hat m^a \hat m^b\nabla_{\hat m}^\Sigma \, \IIo^{\Sigma}_{ab}
+H_{\Lambda\hookrightarrow \Sigma}\hh\hh \IIo^\Sigma_{\hat m\hat m} \big)
\right)\!\Big|_\Lambda\, \right]
\, .
$$
\end{proposition}

\begin{proof}
Without loss of generality we may always choose $\bmu$ to be a minimal { unit} defining density.
When $k=1$, Lemma~\ref{recurse}
simply says that 
$$
\rho_\mu
\stackrel\Lambda=-I^{ab}\nabla_a m_b
\stackrel\Lambda=-H_\Xi\, ,
$$
where the last equality uses Equations~\nn{aftervector} and~\nn{singyamIIo}.
This is a particular case of the more general result $\rho_\mu|_\Xi=-H_\Xi$ (see~\cite[Proposition 3.5]{Goal}).
Now we proceed to the case $k=2$ which determines the obstruction.
We must now 
be  careful to record the terms labeled ``LOT'' in Lemma~\ref{recurse}.
For that we return to Equation~\nn{sofarsogood} which for $k=2$ gives 
\begin{eqnarray*}
2{\mathcal B}\stackrel\Lambda=
(\nabla_n^\top)^2\langle \bsigma,\bmu\rangle&\stackrel\Lambda=&
  -\nabla_n^\top(
   I^{ab} \nabla_a^\top m_b)
- \nabla_m \rho_\sigma+\rho_\mu\rho_\sigma
+ \rho_\mu m^a\nabla_m n_a
  +
\rho_\mu(\nabla_n^\top)^2\sigma  \\[1mm]
&\stackrel\Lambda=&
-\nabla_n(
   I^{ab} \nabla_a m_b)
- \nabla_m \rho_\sigma
-H_\Xi\hh \IIo^\Sigma_{\hat m\hat m}
-
2H_\Sigma H_\Xi
\ ,
\end{eqnarray*}
where in the last line we used that
$ \nabla_n \big(n^2-(m.n)^2\big)\stackrel\Lambda= -2\rho_\sigma$ as well as Equation~\nn{singyamIIo}.

Now we note that
 \begin{equation}\label{superclaim}\nabla_n n_a\stackrel \Sigma= n_a H_\Sigma\stackrel \,, \qquad
 \nabla_n m_a\stackrel\Xi=
n^b\IIo^\Xi_{ab}+ n_a H_\Xi\, ,\qquad
\nabla_m m_a \stackrel\Xi= m_a H_\Xi\, .
 \end{equation}
 These identities follow  from
Equation~\nn{singyamIIo}.
Applying them  and remembering that $\IIo^\Xi_{\hat n\hat n}|_\Lambda=0$,
we find that along $\Lambda$, the quantity $-\nabla_n(
   I^{ab} \nabla_a m_b)$ equals
$$
-I^{ab}\nabla_n \nabla_a m_b
+2H_\Sigma\,  n_a \nabla_n m^a
\stackrel \Lambda =
-I^{ab}\nabla_n \nabla_a m_b
+2H_\Sigma H_\Xi 
\, . $$
So far we have found
$$
2{\mathcal B}\stackrel \Lambda =
-I^{ab} \nabla_n \nabla_a m_b
-\nabla_m\rho_\sigma-H_\Xi\hh \IIo^\Sigma_{\hat m\hat m}
\, .
$$
We now attack the first term on the right hand side along $\Lambda$ and find
$$
-I^{ab} \nabla_a \nabla_n m_b -I^{ab} n^c R_{cabd} m^d+
I^{ab}(\nabla_a n^c) \nabla_c m_b \, .
$$
Each term above is easily handled.
For the first we note that, because the operator $I^{ab}\nabla_a$ is tangential along $\Lambda$, by virtue of Equation~\nn{superclaim} we may replace $\nabla_n m_b$ in the first term by any extension of $\hat n^a\IIo^\Xi_{ab}+ \hat n_b H_\Xi$.
Then remembering that $\hat n^a\IIo^\Xi_{ab}$ gives a tangent covector to~$\Lambda$, and using that for any smooth extension $v_a$ of 
a covector $v^\Lambda_a\in T^*\Lambda$ it holds that $I^{ab}\nabla_a v_b\stackrel\Lambda=\nabla^{b\Lambda} v_b^\Lambda$,
we have
$
-I^{ab} \nabla_a \nabla_{ n} m_b
\stackrel\Lambda=
-\nabla^{b\hh\Lambda}(\hat n^a\IIo^\Xi_{ab})
+(\IIo^\Sigma_{\hat m\hat m}-H_\Sigma)H_\Xi\, .
$
Along $\Lambda$, the second term of the above display equals $\Ric_{\hat n\hat m}=\Rho_{\hat n\hat m}$. 
For the third term, using Equation~\nn{singyamIIo},  we find
$g^{ac}g^{bd}\IIo^\Sigma_{ab} \IIo^\Xi_{cd}-H_\Sigma\,  \IIo^\Xi_{\hat n\hat n}
 -H_\Xi\,  \IIo^\Sigma_{\hat m\hat m}
 + H_\Sigma H_\Xi$.
Using Equation~\nn{aftervector}, the identity
 \begin{equation}\label{nablamH}
 \nabla_m \rho_\sigma \stackrel\Lambda= - \nabla^\Sigma_{\hat m} H_\Sigma\, ,
 \end{equation}
  and orchestrating these results gives
 $$
 2{\mathcal B}\stackrel\Lambda=
 -\nabla^{b\hh\Lambda}(\hat n^a\IIo^\Xi_{ab})
-H_\Xi\hh \IIo^\Sigma_{\hat m\hat m} 
 +
 \IIo^\Sigma_{ab}\hh\hh  \IIo^{ab \hh\Xi}
+\nabla^\Sigma_{\hat m} H_\Sigma+ \Rho_{\hat n\hat m }\, .
 $$
By the Codazzi--Mainardi equation specialized to $d=3$ (see for example~\cite[Equation (2.9)]{GW15})  the last two terms obey the identity
\begin{equation}\label{nablamH1}
\nabla^\Sigma_{\hat m} H_\Sigma+ \Rho_{\hat n\hat m }\stackrel\Sigma=\hat m^a \nabla^{b\hh\Sigma} \IIo^\Sigma_{ab}\, .
\end{equation}
We now want to  rewrite 
the above in a form similar to the first term in 
$2{\mathcal B}$, this is done as follows:
First we note that it equals the restriction to $\Lambda$ of
$$
m^a(\nabla^b -n^b \nabla_n)\, \IIo^{\Sigma,\rm ext}_{ab}
= 
(\nabla^b -n^b \nabla_n)\, \big(m^a\IIo^{\Sigma,\rm ext}_{ab}\big)
-
\IIo^{\Sigma}_{ab}
\nabla^b m^a
\, ,
$$
where $\IIo^{\Sigma,\rm ext}_{ab}$ is any smooth extension to $M$  of the trace-free second fundamental form of~$\Sigma$.
Using Equation~\nn{singyamIIo}, the derivatives of $m$ in the last term can be replaced by $\IIo^\Xi$ while for first term we convert the leading derivative to one tangential to $\Lambda$. Then the  the above display becomes
\begin{multline*}
I^b_c\nabla^c \big(m^a\IIo^{\Sigma,\rm ext}_{ab}\big)
+m^b \nabla_m\big(m^a\IIo^{\Sigma,\rm ext}_{ab}\big)-\IIo^\Xi_{ab}\hh \IIo^{\Sigma \hh ab}\\[1mm]
\stackrel\Lambda=
\nabla^{b \Lambda}\big(\hat m^a\IIo^{\Sigma}_{ac}I^c_b\big)
+I^{ab}(\nabla_a m_b)\,  \IIo^\Sigma_{\hat m \hat m}
+\hat m^a \hat m^b\nabla_{\hat m}^\Sigma \, \IIo^{\Sigma}_{ab}+H_\Xi\,  \IIo^\Sigma_{\hat m\hat m}-\IIo^\Xi_{ab}\hh \IIo^{\Sigma \hh ab}\, .
\end{multline*}
Noting that along $\Lambda$ we have $I^{ab}\nabla_a m_b=H_\Xi$, collecting terms gives
$$
 2{\mathcal B}\stackrel\Lambda=
 -\nabla^{b\hh\Lambda}(\hat n^a\IIo^\Xi_{ab})
+\nabla^{b \Lambda}\big(\hat m^a\IIo^{\Sigma}_{ac}I^c_b\big)
+\hat m^a \hat m^b\nabla_{\hat m}^\Sigma \, \IIo^{\Sigma}_{ab}
+H_\Xi\hh \IIo^\Sigma_{\hat m\hat m} 
\, .
 $$
The proof is completed upon using Lemmas~\ref{IIonn} and~\ref{H2H}, and then  demonstrating the following identity (which in  fact holds in any dimension $d$) along $\Lambda$
\begin{equation}\label{Rodsequation}
\hat n^a \IIo_{ab}^\Xi + \hat m^a \IIo_{ac}^\Sigma I^c_b = 0\, .
\end{equation}
This is established by computing $\nabla^\Lambda_a (\hat m.\hat n)$ and noting that this quantity must vanish since $\hat m$ and $\hat n$ are perpendicular along $\Lambda$.

\end{proof}

\begin{remark}
The above proposition shows that the three-manifold minimal obstruction vanishes when $\Sigma$ is umbilic. In particular this is the case when $(M,g^o)$ is Poincar\'e--Einstein. 
Also, we note that the quantity  $\nabla^{b \Lambda}\big(\hat m^a\IIo^{\Sigma}_{ac}I^c_b\big)
+\tfrac12 \big(\hat m^a \hat m^b\nabla_{\hat m}^\Sigma \, \IIo^{\Sigma}_{ab}
+H_{\Lambda\hookrightarrow \Sigma}\hh \IIo^\Sigma_{\hat m\hat m} \big)$
is a  weight $-2$ boundary invariant for a conformally embedded surface $\Sigma$, which is new to the best of our knowledge. 
A straightforward computation shows that under the transformation $g\mapsto \Omega^2 g$, invariance requires all three terms to be present.
That this invariant is non-trivial can be verified by checking that the obstruction is non-vanishing for a sample geometry; an example is given below.\end{remark}

\begin{example}
Let $M={\mathbb R}^3$ with conformal class of metrics $\cc$ containing the metric 
$$
ds^2 = dx^2 + (1+\alpha x) dr^2 + (1+\beta x) r^2 d\theta^2\, ,
$$
where $(x,r,\theta)$ are cylindrical coordinates and $\alpha,\beta$ are constant parameters inserted to deform the geometry away from Euclidean 3-space. Take $\Sigma$ to be the plane $x=0$ and $\Lambda$ to be the circle $x=0=r-1$. Then for the asymptotically singular Yamabe metric $g^o=ds^2/(x[1+\frac18(\alpha+\beta)x-\frac1{24}(2\alpha^2-\alpha\beta+2\beta^2)x^2])^2$, the asymptotically minimal surface~$\Xi$ is given by $r=1-\frac12 x^2$ and the obstruction $B_\Lambda=\frac18 (\alpha-\beta)$. The trace-free second fundamental form $\IIo=\frac14 (\alpha-\beta)(dr^2 - r^2 d\theta^2)$, the mean curvature $H_{\Lambda\hookrightarrow \Sigma}=1$ and the unit conormal to $\Lambda$ in $\Sigma$ is $\hat m=dr$. Only the last term in the result for $B_\Lambda$ in Proposition~\ref{B3} is non-vanishing for this simple example; it produces the quoted result for the obstruction.
\end{example}

\section{Renormalized Volumes and Areas}
\label{Sec:RenVA}
Here we consider the sequence of conformal embeddings $\Lambda\hookrightarrow \Sigma \hookrightarrow (M,\cc)$ where
$\Sigma$ is an orientable hypersurface in $M$
and $\Lambda$ is a closed orientable hypersurface in~$\Sigma$. The closure of the interior of $\Lambda$ in $\Sigma$ is denoted by $\widetilde\Sigma$.
 We assume that $\widetilde \Sigma$ is compact and that $\widetilde\Sigma\cap \partial M=\emptyset$, this may involve a choice of interior for $\Lambda$.
 In addition we consider a second orientable hypersurface $\Xi\hookrightarrow (M,\cc)$ anchored along $\Lambda=\partial \Xi$ such that the union $\Xi \cup\widetilde\Sigma $ is closed with interior denoted by $D$. For technical reasons we require some subset of $D$ to be a collar neighborhood of $\widetilde\Sigma$. At this point we do not require any of these hypersurfaces to be minimal. This geometry is depicted below:

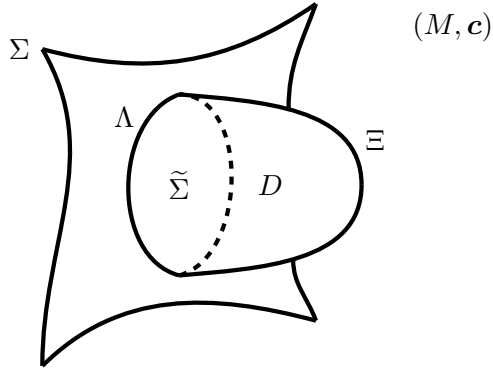
\begin{center}
\begin{tikzpicture}[scale=0.6, ultra thick]
\coordinate (LD) at (-4,-4);
\coordinate (RD) at (2,-3);
\coordinate (LU) at (-4,3);
\coordinate (RU) at (2,4);
\coordinate (CU)at (-1,2);
\coordinate (CD)at (-1,-2);
\coordinate(B) at (3,0);
\coordinate(IU) at (1.4,1.7);
\coordinate(ID) at (1.5,-1.65);
\draw (CU) to[out=-160,in=165] (CD);
\draw [dashed] (CD) to[out=15,in=0] (CU);
\draw (CU) to[out=-5,in=90] (B);
\draw (CD) to[out=5,in=-90] (B);
\draw (LU) to[out=-15,in=-145] (RU);
\draw (LD) to[out=45,in=165] (RD);
\draw (RU) to[out=-110,in=90] (IU);
\draw (RD) to[out=110,in=-90] (ID);
\draw (LU) to[out=-60,in=90] (LD);
\node at (5,3.5) {$(M, \cc)$};
\node at (-4.5,3) {$\Sigma$};
\node at (-2.2,1.5) {$\Lambda$};
\node at (3.3,1) {$\Xi$};
\node at (-1,0) {$\widetilde\Sigma$};
\node at (1,0) {$D$};
\end{tikzpicture}
\end{center}

We now input the additional data of a defining density $\bsigma=[g;\sigma]$ for $\Sigma$ which determines a metric $g^o$ on $M\backslash\Sigma$
that is singular along $\Sigma$. We assume $\sigma$ is positive on the side of~$\Sigma$ occupied by $D$.
We wish to study the ``volume'' of~$D$ and the ``area'' of $\Xi$ with respect to this singular metric.
These diverge so we must regulate them suitably.

To begin with, we introduce a unit defining density $\bmu=[g;\mu]$ for $\Xi$ where
$$
D=\{p\in M\, |\,  \mu(p)\leq0\leq \sigma(p)\} \, ,
$$
and 
$$
\Xi={\mathcal Z}(\bmu)\cap \{
p\in M\, |\,   \sigma(p)\geq0\}\, .
$$
Then we pick any true scale $\btau=[g;\tau]$ and define a {\it regulating hypersurface}
$$
\Sigma_\varepsilon = {\mathcal Z}(\bsigma/\btau-\varepsilon)\, ,
$$
where $1\gg\varepsilon \in {\mathbb R}_+$.
Then the metric $g^o$ is non-singular in the region
$$
D_\varepsilon:=\{p\in M\, |\,  \mu(p)<0< \tau(p)\hh\varepsilon<\sigma(p)\} \, .
$$
The choice of $\btau$ determines the regulating hypersurface, hence we call it the {\it regulator}.
 Of key interest are quantities that are independent of the  choice of regulator.

The {\it regulated volume} $\Vol_\varepsilon$ of $D$ is defined as the volume of the region $D_{\varepsilon}$ with respect to the singular metric $g^o$. This is regulator-dependent
 and given by
 \begin{equation}\label{vol_eps}
\Vol_\varepsilon=\int_{D_\varepsilon}\frac1{\bsigma^d}\, .
\end{equation}
It turns out to be  advantageous to write this result in terms of Heaviside step functions and an integral over the manifold $M$:
$$
\Vol_\varepsilon=\int_M \frac{\theta(\bmu)\theta(\bsigma/\btau-\varepsilon)}{\bsigma^d}\, .
$$
In what follows we will assume that $\partial M=\emptyset$ (alternatively one can multiply the measure of integration by a smooth unit cut-off function that vanishes far from the region $D$). This leads to no loss of generality for our purposes.


\medskip

The regulated area of $\Xi$ is computed by integrating the volume element of the pullback of $g^o$ along the  hypersurface $\Xi_\varepsilon:=\{p\in \Xi\, |\, \sigma(p)>\varepsilon \tau(p)\}\, .$
This could be computed (inefficiently)  by recycling the  renormalized volume computations of~\cite{GWvol} for the conformal manifold $(\Xi,\cc_\Xi)$ where $\cc_\Xi$ is the conformal class of metrics induced by $\cc$.
Instead it is  propitious to
also write the regulated area in terms of distributions; indeed using Equation~\nn{deltasurface} and that $\bmu$ is a unit defining density, we have
\begin{equation}\label{Area}
\Area_\varepsilon=\int_M\frac{\delta(\bmu)\theta(\bsigma/\btau-\varepsilon)}{\bsigma^{d-1}}\, .
\end{equation}

\subsection{The $\varepsilon$ expansion}

By using the distributional identity
$$
\frac{d\theta(\bsigma/\btau-\varepsilon)}{d\varepsilon} =-\delta(\bsigma/\btau-\varepsilon)\, ,
$$
in~\cite{GWvol,GWvolII} it was proved that
\begin{equation}\label{reg_vol}
\Vol_\varepsilon=\sum_{k=d-1}^1 \frac{v_k}{\varepsilon^k}
+{\mathcal V}\, \log\varepsilon+\Vol_{\rm ren}
+\varepsilon\, {\mathcal R}(\varepsilon)\, ,
\end{equation}
where ${\mathcal R}(\varepsilon)$ is smooth. 
The $\varepsilon$-independent term $\Vol_{\rm ren}$ defines the {\it renormalized volume}. The poles in $\varepsilon$ (which we denote ${\rm Poles}({\rm Vol}_\varepsilon)$) are called {\it divergences} and have coefficients  given by local (along $\Sigma$), but regulator-dependent, formul\ae 
\begin{equation}\label{volumedivergences}
v_k=\frac{(-1)^{d-1-k}}{(d-1-k)!k} \, \int_M\frac{
\theta(\bmu)\, \delta^{(d-1-k)}(\bsigma)}{\btau^k}\, .
\end{equation}
The coefficient of the $\log\varepsilon$  term is called the {\it anomaly}. It is again local but in addition regulator-independent and given by
\begin{equation}\label{volomaly}
{\mathcal V}=\frac{(-1)^d}{(d-1)!}\, \int_M\theta(\bmu)\, \delta^{(d-1)}(\bsigma)\, .
\end{equation}

Analogous results may be derived for the regulated area using the methods developed in~\cite{GWvol,GWvolII}, a sketch goes as follows: 
First we differentiate Equation~\nn{Area} with respect to $\varepsilon$ to obtain
$$
-\varepsilon^{d-1}\, \frac{d{\rm Area}_{\varepsilon}}{d\varepsilon}=
\int_M\frac{\delta(\bmu)\delta(\bsigma-\varepsilon\btau)}{\btau^{d-2}}\, .
$$
The right hand side is manifestly an analytic function of~$\varepsilon$, and hence we may compute its Taylor series coefficients by taking successive $\varepsilon$ derivatives and setting $\varepsilon$ to zero. This gives
\begin{equation}\label{dA/de}
-\varepsilon^{d-1}\, \frac{d{\rm  Area}_{\varepsilon}}{d\varepsilon}=
\int_M\frac{\delta(\bmu)\delta(\bsigma)}{\btau^{d-2}}
-
\varepsilon\int_M\frac{\delta(\bmu)\delta'(\bsigma)}{\btau^{d-3}}+\cdots+\frac{(-\varepsilon)^\ell}{\ell!}
\int_M\frac{\delta(\bmu)\delta^{(\ell)}(\bsigma)}{\btau^{d-\ell-2}}\, +\cdots
\, ,
\end{equation}
whence we have 
\begin{equation}\label{AREA}
\Area_\varepsilon=\sum_{k=d-2}^1 \frac{a_k}{\varepsilon^k}
+{\mathcal A}\, \log\varepsilon+\Area_{\rm ren}
+\varepsilon\, {\mathcal R}(\varepsilon)\, ,
\end{equation}
where ${\mathcal R}(\varepsilon)$ is smooth and $\Area_{\rm ren}$ is the constant of integration. The coefficients of the divergences are given by 
\begin{equation}\label{areadivergences}
a_k=
\frac{(-1)^{d-2-k}}{(d-2-k)!k}
\int_M \frac{\delta(\bmu)\delta^{(d-2-k)}(\bsigma)}{\btau^{k}}\, ,
\end{equation}
while the anomaly is 
\begin{equation}\label{arnomaly}
{\mathcal A}=\frac{(-1)^{d-1}}{(d-2)!}\int_M\delta(\bmu)\delta^{(d-2)}(\bsigma)\, . 
\end{equation}
The renormalized area $\Area_{\rm ren}$ is not determined by Equation~\nn{dA/de}
and thus is not forced  to be local along $\Sigma$. A main aim of  this article is to compute holographic formul\ae\   for the area and volume anomalies ${\mathcal A}$ and ${\mathcal V}$ as well as the divergence coefficients~$a_k,v_k$.

\section{Distributional Calculus}
\label{Sec:Calculus}

%
In this 
section we develop the general calculus required for handling 
 products of distributions involving multiple defining densities such as those appearing in~\nn{volumedivergences},~\nn{volomaly},~\nn{areadivergences} and~\nn{arnomaly}.
Just as for  single scale problems, main ingredients are the $\mathfrak{sl}_2$
algebra involving the Laplace--Robin operator described in Section~\ref{LapRob},
and  formal self-adjointness of the 
Laplace--Robin operator \eqref{FSA}.

\medskip

\subsection{Multiscale calculus}

In the previous section, we showed that the divergences and anomalies we are interested in computing are written in terms of integrals of the form 
\begin{equation}
\label{Intpq}
\int_M \theta(\bmu)\delta^{(k+1)} (\bsigma)  \bm f\, 
\mbox{ and }
\int_M \delta(\bmu)\delta^{(k)} (\bsigma)  \bm f\, ,
\end{equation}
where $k$ is a non-negative integer and $\bm f$ is some weight $-d+k+2$ density, 
or possibly log density when $k=d-2$.
Our aim is to rewrite integrals of the above form as
\begin{equation}
\label{Intpq1}
\int_M \theta(\bmu)\delta(\bsigma)  \bm F\, \quad
\mbox{ and }\quad
\int_M \delta(\bmu)\delta (\bsigma)  \bm G\, ,
\end{equation}
where $\bm F$ and $\bm G$ have weight 
$1-d$ and $2-d$, respectively.
The integrands in Equation~\nn{Intpq} have support along
the compact hypersurface and submanifold $\widetilde\Sigma$ and $\Lambda$, respectively. This will allow us to discard boundary terms when integrating by parts because we uniformly assume either that $\partial M\cap \widetilde\Sigma=\emptyset$ or that we are dealing with integrals where a cut-off function has been inserted to ensure vanishing support there.
Moreover, 
when~$\bsigma$ is a unit  defining density, the first integral in Equation~\nn{Intpq1} becomes $\int_{\widetilde\Sigma} \bm F$, and when in addition the density $\bmu$ is unit defining,  the second integral is $\int_\Lambda \bm G$. Hence $\bm F$ and $\bm G$
are holographic formul\ae~in the sense discussed in the introduction. 

%

%

We will mainly focus, partly because it is canonical and simplifying, on the case that~$\bsigma=[g;\sigma]$ is a  unit defining density; the method to handle to general $\bsigma$ is described in~\cite{GWvol}.  
Our main interest is in the case where
 $\Xi$ is the zero locus of a minimal unit defining density~$\bmu$. In other words the singular metric~$g^o$ has asymptotically constant scalar curvature and~$\Xi$ is an asymptotically minimal surface:
$$
{\Sc^{g^o}}=1+{\mathcal O}(\sigma^{d})\, ,\qquad
H^{g^o}_\Xi\stackrel \Xi={\mathcal O}(\sigma^{d-1})\, .
$$

\medskip

 A key tool is  the following (single scale) recursive identity for derivatives 
of delta functions of unit  defining densities 
(see~\cite[Proposition 3.3]{GWvol})
\begin{equation}\label{Ls2deltas}
(\D_\bsigma)^k \delta(\bsigma)
=(d-k-1)(d-k-2)\cdots(d-2)\hh \delta^{(k)}(\bsigma)\, , \quad 
k\in \{1,\dots,d-1\}\, .
\end{equation}

Turning to the multiscale setting, our starting point is  the following 
identity for Laplace--Robin operators acting on the Heaviside function:
\begin{lemma}\label{heavylap}
Let $\bmu$ and $\bsigma$ be arbitrary weight $w=1$ densities, then
$$
\D_\bsigmas \theta(\bmu) =-
\bsigma\bS_\bmus \delta'(\bmu)
 +(d-2)\langle\bsigma,\bmu\rangle \, \delta(\bmu)\, .$$
\end{lemma}
\begin{proof}
This is a simple application of the distributional 
identity $d\theta(x)/dx=\delta(x)$ and the definition~\nn{Ldef} of the Laplace--Robin operator. Calling $\bsigma=[g;\sigma]$ and $\bmu=[g;\mu]$ one has
\begin{equation*}
\begin{split}
\D_\bsigmas \theta(\bmu)&=\D_\bsigmas [g;\theta(\mu)]=
[g;(d-2) \nabla_n\theta(\mu)-\sigma \Delta\theta(\mu)]\\
&=[g;(d-2) m.n\,  \delta(\mu) - \sigma (m^2 \delta'(\mu)+\delta(\mu)\Delta\mu)]\, .
\end{split}
\end{equation*}
Here, as usual, we denoted $n=\ext  \sigma$ and $m=\ext \mu$. To write the first term on the last line of the above display in terms of $\langle\bsigma,\bmu\rangle$, we would need to add 
$-\tfrac1d (\sigma \Delta \mu + \mu \Delta \sigma+2J\sigma\mu)$ to~$m.n$ (see Equation~\nn{thebracket}).
But $\mu \delta(\mu)=0=\mu^2\delta'(\mu)$, so this addition  is canceled by  a term $\tfrac{d-2}d \delta(\mu)\sigma\Delta\mu=
-\tfrac{d-2}d \delta'(\mu)\sigma\mu \Delta \mu 
+ \frac{2}{d}\delta'(\mu)\sigma\mu^2 J
$.
Writing the last term of the above display as $\sigma \delta'(\mu)\mu \Delta \mu$, 
 the coefficient of $-\sigma\delta'(\mu)$ is 
$m^2-\tfrac{2}{d}\mu(\Delta\mu+J\mu)$
which (see Equation~\nn{Scurvy}) correctly produces the required ${ {\mathcal S}}$-curvature contribution.
\end{proof}

The above lemma has the following corollary:
\begin{corollary}
Let $\bsigma$ be a weight $w=1$ density, then
$$
\D_\bsigmas \theta(\bsigma) =(d-1)
\bS_\bsigmas \delta(\bsigma)\, .$$
\end{corollary}

Exactly the same proof technique as used for the above lemma and corollary extends to the Laplace--Robin operator acting on differentiated delta functions:

\begin{lemma}\label{lapdelder}
Let $\bmu$ and $\bsigma$ be arbitrary  weight $w=1$ densities. Then, if $k\in{\mathbb Z}_{\geq 1}$,
$$
\D_\bmus \delta^{(k-1)}(\bsigma) =-
\bmu\hh\bS_\bsigmas \delta^{(k+1)}(\bsigma)
 +(d-2k-2)\langle\bsigma,\bmu\rangle \, \delta^{(k)}(\bsigma)\, .
$$
\end{lemma}

\begin{remark}
The central result~\nn{Ls2deltas} is a corollary to the above lemma, since replacing~$\bmu$ by $\bsigma$ and using $\bsigma\delta^{(k+1)}(\bsigma)=-(k+1)\delta^{(k)}(\bsigma)$ gives (for $k\in {\mathbb Z}_{\geq 0}$)
\begin{equation}\label{Dondel}
\D_\bsigmas \delta^{(k-1)}(\bsigma) = (d-k-1) \, \bS_\bsigmas\delta^{(k)}(\bsigma)\, .
\end{equation}
In the above, we have employed the notation 
$\delta^{(0)}(\bsigma):=\delta(\bsigma)$.
Also, calling
$\delta^{(-1)}(\bsigma):=\theta(\bsigma)$, Lemma~\ref{heavylap}
extends Lemma~\ref{lapdelder} to the case $k=0$. 
\end{remark}

\subsection{Multiple distributions}\label{superlaprob}
Lemmas \ref{heavylap} and \ref{lapdelder} provide the basic 
identities required to handle integrals of the type \eqref{Intpq},
which  include distributions with codimension two support that are built from distributions with  support along hypersurfaces.  The central relations we need to derive are Equations~\nn{master1} and~\nn{master2}; these underlie the key Proposition~\ref{cool}.

To begin with we observe that
generically
the product of a differentiated delta function multiplied by a Heaviside function lies in the span of images of Laplace--Robin operators:
\begin{proposition}\label{thetadelta}
Let $\bsigma$ be a unit defining density and $\bmu$ a minimal unit defining density. If $\tfrac d2\neq k\in \{1,\ldots, d-2\}$, then 
$$
\theta(\bmu)
\delta^{(k)}(\bsigma) 
=\frac{1}{d-k-1}\Big[
\D_\bsigmas\big(\theta(\bmu)\delta^{(k-1)}(\bsigma)\big)
-\frac{k-1}{
d-2k}\D_\bmus\big(\delta(\bmu)\delta^{(k-2)}(\bsigma)\big)\Big]\, ,
$$
where the second term is absent  when $k=1$.
\end{proposition}

\begin{proof}
Since our strategy relies on the Leibniz rules 
of Section~\ref{LRs} we first assume $k\neq \frac d2-1$ 
and use Equation~\nn{thebracket} to compute (again much along the lines of the proof of Lemma~\ref{heavylap})
\begin{equation*}
\begin{split}
\Langle \theta(\bmu),\delta^{(k-1)}(\bsigma)\Rangle
&=
\langle\bsigma,\bmu\rangle\, 
\delta(\bmu)\delta^{(k)}(\bsigma)
+\tfrac{k}{d-2}\bS_\bmu \delta'(\bmu) \delta^{(k-1)}(\bsigma)
=\tfrac{k}{d-2}\,  \delta'(\bmu) \delta^{(k-1)}(\bsigma)\, .
\end{split}
\end{equation*}
%
Here we used the unit and minimal unit defining properties imply $\Langle \bsigma,\bmu\Rangle\stackrel \bmuss\sim {\mathcal O}( \bsigma^{d-1})$.
 Combining this display with
 Lemma~\ref{heavylap}, Equation~\ref{Dondel} and  the Leibniz rule of Section~\ref{LRs}
 we find
 \begin{equation*}
\begin{split}
\D_\bsigmas\big(\theta(\bmu)\delta^{(k-1)}(\bsigma)\big)&=(k-1)\delta'(\bmu)\delta^{(k-2)}(\bsigma)+
(d-k-1)\theta(\bmu) \delta^{(k)}(\bsigma)\, ,
\end{split}
\end{equation*}
where the first term is absent when $k=1$.
Although the above display was computed when~$k\neq \frac d2-1$, it is not difficult to choose a scale $g\in\cc$ and verify that it holds for any $k\in \{1,\ldots, d-2\}$.

Next  we address the other term on the right hand side of the result.
Because 
$\delta(\bmu)$ has Yamabe weight in four dimensions, we first consider $d\neq 4$ and  compute (again using Equation~\nn{thebracket} and taking $k\geq 2$) 
\begin{equation}\label{LARA}
\Langle\delta(\bmu),\delta^{(k-2)}(\bsigma)\Rangle
=\langle\bsigma,\bmu\rangle \, \delta'(\bmu)\delta^{(k-1)}(\bsigma)
+\tfrac{k-1}{d-4}\,  \delta''(\bmu) \delta^{(k-2)}(\bsigma)
+\tfrac1{d-2k} \, \delta(\bmu)\delta^{(k)}(\bsigma)\, ,
\end{equation}
in order to use  the generalized Leibniz rule~\nn{Leibniz} (away from $k=\frac d2-1$ and $d=4$), which then gives
\begin{equation}\label{Lmdmds}
\D_\bmu\big(\delta(\bmu)\delta^{(k-2)}(\bsigma)\big)=(d-2k)\hh
\delta'(\bmu)\delta^{(k-2)}(\bsigma)
\, .
\end{equation}
In fact the last display holds for any integer $k\in \{2,\ldots, d-2\}$ and for $d=4$; this can be checked by direct computation in  a choice scale $g\in\cc$.
%
%
%
%
Combining the above displays completes the proof.
\end{proof}

\begin{corollary}\label{thetadeltacor}
Let $\bsigma$ be a unit  defining density and $\bmu$ a minimal unit  defining density.
If $\bm f$ is a weight $w=k+1-d\neq1-\tfrac d2$ density with $k\in\{1,\ldots, d-2\}$, then 
\begin{equation*}
\int_M 
\theta(\bmu)\delta^{(k)}(\bsigma) 
\bm f=
\frac1{d-k-1}\left[
\int_M \theta(\bmu)\delta^{(k-1)}(\bsigma) \D_\bsigmas\!\bm f
+(k\!-\!1)
\int_M \delta(\bmu)\delta^{(k-2)}(\bsigma) \, \Lodz_\bmu \bm f\right].
\end{equation*}
The second term on the right hand side is absent when $k=1$.
\end{corollary}
\begin{proof}
The above result is a direct corollary of Proposition~\ref{thetadelta}, the vanishing of $\partial M\cap \widetilde\Sigma$ (or the support of the integrands there) 
and the formal self-adjointness of the Laplace--Robin operator (see Equation~\nn{FSA}).  
\end{proof}

There is a critical value of $k$ excluded from  the above proposition and corollary, at which   the density $\bm f$ 
has
 bulk Yamabe weight $w=1-\frac d2$.
 At this weight the 
operator
 $\widehat \D_\bmu \bm f$ becomes ill-defined. 
 However, observe that in the above display, the product $\delta(\bmu)\widehat \D_\bmus$ appears;
 this suggests
   the replacement
\begin{equation*}
\widehat\D_\bmu \rightarrow \widetilde\D_\bmu \, ,
\end{equation*} 
when $k=\frac d2$,
where the  operator $\widetilde \D$ 
is defined in \eqref{Lodzhat}. This is executed as follows:

\begin{lemma}\label{crittheta}
Let $\bsigma$ be a unit  defining density and $\bmu$ a minimal unit defining density.
Suppose $d\geq 4$ is even, and let $\bm f$ be a weight $w=1-\frac d2$ density, then 
\begin{equation*}
\int_M 
\theta(\bmu)\delta^{(\frac d2)}(\bsigma) 
\bm f=
\frac2{d-2}\left[
\int_M \theta(\bmu)\delta^{(\frac d2-1)}(\bsigma) \D_\bsigmas\bm f
+\frac{d-2}2
\int_M \delta(\bmu)\delta^{(\frac d2-2)}(\bsigma) \, \widetilde\D_\bmu \bm f\right].
\end{equation*}

\begin{proof}
We begin by using $\delta^{(\frac d2)}(\bsigma) = \frac{2}{d-2}\D_\bsigmas \delta^{(\frac d2-1)}$ (see Equation~\nn{Dondel}) and formal self-adjointness of the Laplace--Robin operator to obtain
$$
\int_M 
\theta(\bmu)\delta^{(\frac d2)}(\bsigma) 
\bm f=
\frac2{d-2}
\int_M \delta^{(\frac d2-1)}(\bsigma) \D_\bsigmas\!\big(\theta(\bmu)\bm f\big)\, .
$$
Because $\bm f$ has critical weight we encounter the bulk Yamabe operator
$$
\D_\bsigmas  \big(\theta(\bmu)\bm f\big)
=-\bsigma \square\big(\theta(\bmu)\bm f\big)\, .
$$
Now we pick a $g\in \cc$ and compute
\begin{multline*}
\big(\Delta^g+[1-\tfrac d2] J\big)\big(\theta(\mu) f\big)=
\big(|m|_g^2\hh \delta'(\mu)+(\nabla^g.m)\hh \delta(\mu)\big) f
+ 2 \delta(\mu) \nabla_m f 
+\theta(\mu)\hh  \square f
\\
=\delta'(\mu)f + 2\delta(\mu)\big(\nabla_m-\tfrac{d-2}2\rho_\mu) f + \theta(\mu)\square f\, .
\end{multline*}
Here $m:=\ext \mu$ and the second line used that $\bmu$ is a unit conformal density so that $|m|^2_g
+2\rho_\mu \mu
=1+{\mathcal O}(\mu^d)$
with $\rho_\mu:=-\frac1d (\nabla^g.m+J\mu)$.
Using $\mu \delta(\mu)=0$, it follows that
$$
\D_\bsigmas  \big(\theta(\bmu)\bm f\big)
=-\bsigma \delta'(\bmu)\hh\bm f
-2\bsigma \delta(\bmu)\hh  \widetilde\D_\bmus \!\hh \bm f
+\theta(\bmu)\D_\bsigmas \!\bm f
\, ,
$$
and hence
\begin{multline*}
\int_M 
\theta(\bmu)\delta^{(\frac d2)}(\bsigma) 
\bm f=
\frac2{d-2}
\int_M \Big[\theta(\bmu)\delta^{(\frac d2-1)}(\bsigma) \D_\bsigmas \bm f
+(d-2)\delta(\bmu)\delta^{(\frac d2-2)}(\bsigma)\widetilde\D_\bmu \bm f 
\\
+\frac{d-2 }2\hh
\delta'(\bmu)\delta^{(\frac d2-2)}(\bsigma)
 \bm f
\Big]\, .
\end{multline*}
We still need to deal with the last term above  and find ourselves in the situation where Equation~\nn{Lmdmds}
cannot be used. Therefore we use the same method just employed to compute
\begin{equation*}\label{dpds}
\int_M \delta'(\bmu) \delta^{(\frac d2-2)}(\bsigma) \bm f =
\frac1{d-2}\int_M \delta(\bmu) \D_\bmus\big(
 \delta^{(\frac d2-2)}(\bsigma)\hh  \bm f \big)
 =-\int_M\delta(\bmu) \Lodz_\bmu\big(
 \delta^{(\frac d2-2)}(\bsigma)\hh  \bm f \big)
 \, .
\end{equation*}
The last integrand is easily calculated in some scale $g\in \cc$ and gives
$$
\delta(\mu)\big(\nabla_m-(d-2)\rho_\mu\big)
\big(\delta^{(\frac d2-2)}(\sigma) f\big)=
\delta(\mu)\delta^{(\frac d2-2)}(\sigma)\big(\nabla_m-\tfrac{d-2}2\rho_\mu\big )f\, .
$$
Here we used that $m.n+\sigma \rho_\mu$ (where $n:=\ext \sigma$) vanishes along $\Xi$ to order ${\mathcal O}(\sigma^{d-1})$
and $\sigma \delta^{(\frac d2-1)}(\sigma)=-\frac{d-2}{2}\delta^{(\frac d2-2)}(\sigma)$.
The above display corresponds to $\delta(\bmu) \delta^{(\frac d2-2)}(\bsigma) \widetilde\D_\bmu \bm f$ so that
\begin{equation}\label{snappo}
\int_M \delta'(\bmu) \delta^{(\frac d2-2)}(\bsigma) \bm f =
-\int_M \delta(\bmu) \delta^{(\frac d2-2)}(\bsigma) \widetilde\D_\bmu \bm f\, ,
\end{equation}
and this completes the proof.
\end{proof}
\end{lemma}

Corollary \ref{thetadeltacor} and its critical extension, 
Lemma \ref{crittheta} are the two basic integrated
recursive relations required to handle integrals involving the product
of distributions~$\theta(\bmu)\delta^{(k)}(\bsigma)$. 
Adopting a unified notation by introducing the operator family
\begin{equation}\label{Lprime} 
\D^\prime_\bmus \bm f :=
\left\{ 
\begin{array}{cc} 
(d+w-2)\,\widehat\D_\bmu \bm f\, ,
& w\neq 1-\frac d2\, ,\\[2mm]
\frac{d-2}{2}\,\widetilde\D_\bmu\, \bm f\, ,
&w=1-\frac d2\, ,
\end{array} \right.
\end{equation}
both cases can be described by a  single
expression
\begin{equation}\label{master1}
\int_M 
\theta(\bmu)\, \delta^{(k)}(\bsigma)\bm f
= \frac{1}{d-k-1}\int_M \Big[
\theta(\bmu) \delta^{(k-1)}(\bsigma) 
\D_\bsigmas
+\delta(\bmu) \delta^{(k-2)}(\bsigma) 
\D^\prime_\bmus \Big] \bm f ~,
\end{equation}
where $k\in\{1,\ldots, d-2\}$  and the second term on the right hand side is absent when~$k=1$.

We will later need the following technical Lemma giving the algebra obeyed by $\D^\prime_\bmu$ and $\bsigma$ along $\Xi$.

\begin{lemma}\label{techlemm}
Let $\bm f$ be a weight $w-1\neq 2-d$ density, $\bsigma$ be a unit  defining density and~$\bmu$ a minimal unit defining density.
 Then along $\Xi$, the operator $\D^\prime_\bmu$ obeys
$$
(d+w-3)\hh\D^\prime_\bmu (\bsigma \bm f)\stackrel{\Xi}= 
(d+w-2)\hh\bsigma\D^\prime_\bmu \bm f
+{\mathcal O}(\bsigma^{d-1})\, .
$$
Moreover, when $w=3-d$ one has $\D^\prime_\bmu (\bsigma \bm f)\stackrel{\Xi}={\mathcal O}(\bsigma)$.
\end{lemma}
\begin{proof}
The simplest proof is based on a choice of scale $g\in \cc$ where $\bm f=[g;f]$. In that case, for all weights $w$ (including $w=1-\frac d2$ and $w=3-d$) we have  that  $\D^\prime_\bmu (\bsigma\bm f)$ along~$\Xi$ is given by
\begin{multline*}
(d+w-2)(\nabla_m+w \rho_\mu) (\sigma f)=
(d+w-2)(\sigma \nabla_m f+m.n f + w \rho_\mu \sigma f)
\\[1mm]
=(d+w-2)\hh\sigma (\nabla_m  + (w-1) \rho_\mu  )f + {\mathcal O}(\sigma^{d-1})\, .
\end{multline*}
In the above, we used that $\mu$ vanishes along $\Xi$, as well as the minimal condition~\nn{mn}.
It is not difficult to assemble the quoted results from this display.
\end{proof}

\medskip
For the particular case of the volume anomaly in 
three dimensions (see Section~\ref{3dV}), we shall 
also need the action of the operator $\D^\prime$ on 
the log of a true scale $\bm \tau$. 
Hence in that case we extend Definition~\ref{Lprime} 
to read
\begin{equation}\label{Lprimelog}
\D^\prime_\bmu \log\bm\tau := 
(d-2) \widehat\D_\bmu \log\bm\tau=\bm \D_\bmus  \log\bm\tau\, .
\end{equation} 
%

\medskip

We now proceed to develop the analogous machinery 
needed for handling  densities  
integrated against the product distribution $\delta(\bmu)\delta^{(k)}(\bsigma)$. Performing the $\delta(\bmu)$ integration these
give integrals  along the hypersurface $\Xi$. 
Therefore, to deal with the operators that appear in this 
case, we rely on  the notion of tangentiality developed 
in Section~\ref{LapRob}, see in particular 
Definition~\ref{TangDef}.

\begin{definition}\label{LTdef}
Let $\bsigma$ be a unit  defining density 
and $\bmu$ a minimal unit defining density.
The \emph{tangential} 
\emph{Laplace--Robin 
operator} 
for the hypersurface $\Xi=\mathcal Z(\bmu)$
is defined as follows:

In the case where $\bm f$ is a weight $w\neq 1-\tfrac d2,2-\tfrac d2$ 
density, 
$$\D^{\!\rm T}_\bsigmas:\Gamma(\ce M[w])\to \Gamma(\ce M[w-1])$$
according to 
\begin{equation}\label{LTnoncrit}
\D^{\!\rm T}_\bsigmas \bm f :=\Bigg[ 
\frac{d+2w-3}{d+2w-2} \, \D_\bsigmas 
-w\big(\Lodz_\bmu \langle\bsigma,\bmu\rangle\big)
+\frac{\bsigma\D_\bmus^2}{(d+2w-4)(d+2w-2)} \Bigg] \bm f \, .
\end{equation}
 When the density $\bm f$ has bulk Yamabe weight 
$w=1-\frac d2$,  we define the map 
$$\D^{\!\rm T}_\bsigmas:
\Gamma(\ce M[1-\tfrac d2])\to \Gamma_\Xi(\ce M[-\tfrac d2])$$
by (see Display~\nn{funnycod} and Equation~\nn{funnycombination})
\begin{equation}\label{LTcrit1}
\D^{\!\rm T}_\bsigmas \bm f:= \Big[ \D_\bsigmas 
+\tfrac{d-2}2\big(\Lodz_\bmu \langle\bsigma,\bmu\rangle\big)
+\bsigma\widehat \D_\bmu \widetilde\D_\bmu
-\widetilde \D_\bsigmas \Big] \bm f \, .
\end{equation}
%
For $\bm f$ of weight $w=2-\frac d2$ we define 
$$\D^{\!\rm T}_\bsigmas:\Gamma(\ce M[2-\tfrac d2])\to \Gamma_\Xi(\ce M[1-\tfrac d2])$$ 
by
\begin{equation}\label{LTcrit2}
\D^{\!\rm T}_\bsigmas \bm f:= \Big[\widehat \D_\bsigmas
+\tfrac{d-4}2\big(\Lodz_\bmu \langle\bsigma,\bmu\rangle\big)
+\bsigma \widetilde \D_\bmu \widehat \D_\bmu\Big] \bm f \, . 
\end{equation}
\end{definition}
\begin{remark}\label{Images}
Note that although the codomain of the map $\D_\bsigmas^{\!\rm T}$ depends on whether acting at a critical weight or not, 
we are  mostly interested in  computing of $\D_\bsigmas^{\!\rm T} \bm f\big|_\Xi$.
\end{remark}

The tangential Laplace--Robin operator can also be defined acting on log densities:

\begin{definition} \label{logLT}
Let $\bsigma$ and $\bmu$ be  unit defining 
and minimal unit defining densities respectively, 
 and let
 $\bm \tau$ be an arbitrary true scale.
 Then, if $d\neq 4$,  we define the  
  \emph{tangential Laplace--Robin 
operator} acting on log densities 
$$\D^{\!\rm T}_\bsigmas:\Gamma({\mathcal F} M[1])\to \Gamma(\ce M[-1])$$
by
\begin{equation*}
\D^{\!\rm T}_\bsigma \log\bm\tau :=
(d-3)\,\widehat\D_\bsigma\log \bm \tau
-(\widehat\D_\bmu \langle \bsigma,\bmu \rangle) 
+ \bsigma \widehat \D_\bmu^2 \log \bm \tau\, .
\end{equation*}
In the case when $d=4$, the {tangential Laplace--Robin 
operator}  operator acting on log densities
$$\D^{\!\rm T}_\bsigmas:\Gamma({\mathcal F} M[1])\to \Gamma_\Xi(\ce M[-1])$$
is defined as
\begin{equation*}
\D^{\!\rm T}_\bsigma \log\bm\tau :=
\widehat\D_\bsigma\log \bm \tau
-(\widehat\D_\bmu \langle \bsigma,\bmu \rangle)
+ \bsigma \widetilde \D_\bmu \widehat \D_\bmu \log \bm \tau\, .
\end{equation*}
\end{definition}
\begin{remark}
Observe that, just as in the case of Definition~\ref{LTdef} 
for the tangential Laplace--Robin operator acting on standard 
densities, the codomain of the  map $\D^{\!\rm T}_\bsigma$ acting on 
$\log \bm \tau$ is forced to  be differently defined in dimension four, since in this case 
$\widehat \D_\bmu \log \bm \tau$ has critical Yamabe 
weight $w=1-\frac d2=-1$ (see Remark~\ref{Images}).
\end{remark}

The key property of the Laplace--Robin operator $\D^{\!\rm T}_\bsigmas$ is that it is tangential along $\Xi$ in the sense of Definition~\ref{TangDef} and the obvious variant for codomains $\Gamma_\Xi(\ce M[w])$ :

\begin{lemma}
Let $\bsigma$ be a unit  defining density 
and $\bmu$ a minimal unit defining density.
Then the operators $\D^{\!\rm T}_\bsigmas$ of Definitions~\ref{LTdef} and~\ref{logLT} are tangential along $\Xi$.
\end{lemma}
\begin{proof}
Tangentiality of the operator $\D^{\!\rm T}_\bsigmas$ acting at 
weights $w\neq 1-\frac d2$ and $2-\frac d2$ is proved 
in~\cite[Proposition 4.7]{GW161} (see also~\cite{GLW}). 
Hence only the two critical weights remain.
For the case $w=1-\frac d2$, consider $\bm f'$ to be 
an arbitrary weight $w=-\tfrac d2$ density. Then, the 
non-vanishing contributions to $\D^{\!\rm T}_\bsigma (\bmu \bm f')$
along $\Xi$ are given by
\begin{equation*}
\D^{\!\rm T}_\bsigma (\bmu \bm f')\big|_\Xi
= \Big(\D_\bsigmas -\widetilde \D_\bsigmas
+\bsigma\widehat \D_\bmu \widetilde\D_\bmu\Big)
(\bmu \bm  f')\big|_\Xi ~.
\end{equation*}
We need to show that the quantity in the above display in fact vanishes. For that,
picking~$g\in \cc$, a direct calculation using
the various definitions of the Laplace--Robin 
operator and its decorated extensions introduced 
in Sections~\ref{LapRob}~and~\ref{LRs}, yields
\begin{align*}
\D_\bsigma (\bmu \bm f')\big|_\Xi&= 
[g; -\sigma f' \Delta \mu - 2\sigma \nabla_m f']~, \\
-\widetilde \D_\sigma (\bmu \bm f')\big|_\Xi &=
[g; \rho_\mu \sigma f']~, \\
\bsigma \widehat \D_\bmu \widetilde \D_\bmu
(\bmu \bm f')\big|_\Xi&= 
[g;   2\sigma \nabla_m f'
- (d+1)\rho_\mu\sigma f']~,
\end{align*}
with $m:=\ext \mu$, 
and where we have used Equation~\eqref{Snrs} to express 
$m^2 = 1- 2\rho_\mu \mu+{\mathcal O}(\mu^d)$. Adding the three pieces above (remembering that $\Delta \mu|_\Xi = -d \rho_\mu$) we obtain 
$
\D^{\!\rm T}_\bsigma (\bmu \bm f')\big|_\Xi= 0
$,
as required from Definition~\ref{TangDef}.
A completely analogous computation,  
 now using a density
$\bm f'$ of weight $w=1-\frac d2$, proves tangentiality of~$\D^{\!\rm T}_\bsigmas$ at  weight $w=2-\frac d2$. The proof of tangentiality in the log-density case follows from the previous cases specialized to weight $w=0$, because  if $0<f\in C^\infty M$ and $f|_\Sigma=1$ we may write $f= e^{\bmu \bm \omega}$
where $\bm \omega\in \Gamma(\ce M[-1])$ and use $\log (f\btau) = \log \btau + \bmu \bm \omega$.
This reduces the required computations to the previous cases.
\end{proof}

Because the operator $\D^{\!\rm T}_\bsigma$ is tangential, we may define an analog of the Laplace--Robin operator along the hypersurface $\Xi$ by the following holographic formula:
$$
{\mathscr L}^\Xi_\bsigmas:\Gamma(\ce \Xi[w])\to \Gamma(\ce \Xi[w-1])
$$
with
\begin{equation}\label{curlyL}
{\mathscr L}^\Xi_\bsigmas
\bm f_\Xi = \D^{\!\rm T}_\bsigma
\bm f\big|_\Xi\, ,
\end{equation}
where $\bm f_\Xi \in \Gamma(\ce \Xi[w])$ and $\bm f\in \Gamma(\ce M[w])$ is any extension of $\bm f_{\Xi}$ to $M$.
The operator ${\mathscr L}^\Xi_\bsigmas$
 underlies an $\mathfrak{sl}(2)$ algebra analogous to that constructed for the Laplace--Robin operator in Section~\ref{LapRob}. The key relation for this is given below:

\begin{lemma}\label{Iwasatheorem}
Let $\bm f_\Xi\in \Gamma(\ce \Xi[w])$. Then 
$$
{\mathscr L}^\Xi_\bsigmas (\bsigma_{_{\!\Xi}} \bm f_\Xi) - \bsigma_{_{\!\Xi}}{\mathscr L}^\Xi_\bsigmas \bm f_\Xi = (d+2w-1) \bm f_\Xi
+{\mathcal O}(\bsigma_{_{\!\Xi}}^{d-1})\, ,
$$
where $\bsigma_{_{\!\Xi}}$ is the restriction of $\bsigma$ to $\Xi$. 
\end{lemma}

\begin{proof}
There many ways to prove this result, an expedient method is to 
first establish the result along the interior of $\Xi$ because this allows us to calculate using the scale determined by the singular metric $g^o$ for which (away from $\Sigma$) $\bsigma=[g^o;1]$. In particular, 
in this scale, when $w\neq 1-\frac d2, 2-\frac d2$, using Equation~\nn{LTnoncrit} we find
\begin{multline*}
\D^{\!\rm T}_\bsigma
\bm f
\stackrel{\Xi}
=
(d+2w-3)\big(
-\tfrac wd \J^{g^o}\big)f
-\tfrac{d+2w-3}{d+2w-2} (\Delta + w \J^{g^o}) f
-wf\nabla_m(\rho_\mu -\tfrac \mu d\J^{g^o})\\[1mm]
+\big(\nabla_m+(w-1)\rho_\mu\big)
\big(\nabla_m + w \rho_\mu -\tfrac\mu{d+2w-2}(\Delta + wJ^{g^o})\big)f
\end{multline*}
$$
=-\big(\Delta-\nabla_m^2+(2w-1)H_\Xi^{g^o}\nabla_{ m}\big) f+\Big(w(w-1)(H_\Xi^{g^o})^2-\tfrac{2w(d+w-2)}{d}\, J^{g^o}\Big)\hh f\, .
$$
Here we used that in this scale $n=\nabla\sigma=\nabla\, 1 =0$ and $\rho_\sigma =\rho_{\sss1}= -\tfrac 1d\J^{g^o}$. We also used that $\rho_\mu|_\Xi=-H_\Xi^{g^o}$. 
The second equality of the above display shows that the singularities at weights $1-\frac d2$, $2-\frac d2$ of  $\D^{\!\rm T}_\bsigma$ as defined by Equation~\nn{LTnoncrit} are removable. 
Indeed, it is straightforward to use Equations~\nn{LTcrit1} and~\nn{LTcrit2}
to establish that $\D^{\!\rm T}_\bsigma
\bm f$ along $\Xi$ is as stated above for these critical weights.
In the $g^o$ scale, the above display also gives the result for $\bsigma\D^{\!\rm T}_\bsigma
\bm f$ along $\Xi$. Moreover, to compute
$\D^{\!\rm T}_\bsigma
(\bsigma \bm f)$ we only need write out the above display replacing $w$ by $w+1$ because $\bsigma \bm f\in \Gamma(\ce M[w+1])$. Hence for the difference of these we find
$$
\bsigma\D^{\!\rm T}_\bsigma
\bm f
-
\D^{\!\rm T}_\bsigma
(\bsigma\bm f)
\stackrel{\Xi}
=
\frac{2J^{g^o}}d\, (d+2w-1)f
+2H_\Xi^{g^o}(\nabla_{\hat m}-w H_\Xi^{g^o})f\, .
$$
In the interior $- \tfrac 2d\J^{g^o} = \bS_{\bsigmas}$ and $-H^{g^o}_\Xi=\langle \bsigma,\bmu\rangle$. The former of these equals $1+{\mathcal O}(\bsigma^d)$ while along $\Xi$, the latter is ${\mathcal O}(\bsigma^{d-1}_{_{\!\Xi}})$.
This establishes the quoted result on the interior of~$\Xi$. 
Since both the left and right hand sides of this result are smoothly defined along~$\Xi$, the above interior computation suffices to establish their equality along all of~$\Xi$.

\end{proof}

The tangential Laplace--Robin operator~$\D^{\!\rm T}_\bsigma$ introduced in Definition~\ref{LTdef} is built from formally self adjoint Laplace--Robin
operators; see Equation~\eqref{FSA}. 
When integrating over manifolds without boundary (or integrands with no support there), we encounter the  
formal  adjoint~$(\D^{\!\rm T}_\bsigma)^\dagger$  of the tangential operator~$\D^{\!\rm T}_\bsigmas$. This is 
easily computed using formal self-adjointness of the $\D$--operator:
\begin{equation}\label{LTdagger}
(\D^{\!\rm T}_\bsigma)^\dagger  \bm f=  \Bigg[
\frac{d+2w+1}{d+2w} \, \D_\bsigmas 
+(d+w-1)\big(\Lodz_\bmu \langle\bsigma,\bmu\rangle\big)
+\frac{\D_\bmus^2\circ \bsigma}{(d+2w+2)(d+2w)}\Bigg]\bm f  ~,
\end{equation}
when $\bm f$ is a  density of weight $w\neq -1-\frac d2,-\frac d2$.

Remarkably, the product $\delta(\bmu)\delta^{(k)}(\bsigma)$ 
can be expressed  in terms of  the formal adjoint given in~\eqref{LTdagger} acting on 
a product of delta functions with one fewer derivatives. This result, given in the lemma below, may also be viewed as 
an generalization of  the key recursive 
 relation~\eqref{Dondel} to  products of distributions. 

%
\begin{proposition}\label{LST}
Let $\bsigma$ be a unit  defining density 
and $\bmu$ a minimal unit  defining density. 
If  $\tfrac d2,\tfrac d2-1\neq k\in \{1,\ldots, d-3\}$, then
\begin{equation*}
\delta(\bmu)\delta^{(k)}(\bsigma)=
\frac{1}{d-k-2} \big(\D_\bsigma^{\!\rm T}\big)^\dagger \big(\delta(\bmu)\delta^{(k-1)}(\bsigma)\big) .
\end{equation*}
\end{proposition}
 
\begin{proof}
We proceed along similar lines to the proof of 
Proposition~\ref{thetadelta}. First we use the generalized Leibniz rule~\nn{Leibniz}, Lemma~\ref{lapdelder} and
Equation~\nn{LARA}, avoiding special dimensions and values of $k$,
to compute 
\begin{multline}\label{intermediate}
\D_\bsigma\big(\delta(\bmu)\delta^{(k-1)}(\bsigma)\big)=
(d-k-3)\delta(\bmu)\delta^{(k)}(\bsigma)\\
+(d-4)\langle\bsigma,\bmu\rangle\delta'(\bmu)\delta^{(k-1)}(\bsigma)
+(k-1)\delta''(\bmu)\delta^{(k-2)}(\bsigma)\, .
\end{multline}
The above display in fact holds in all dimensions and for $k\in \{2,\ldots, d-3\}$.
When $k=1$, the last term is absent. 
Equation~\nn{Lmdmds} gives 
$
\D_\bmu\big(\delta(\bmu)\delta^{(k-1)}(\bsigma)\big)
=(d-2k-2)\delta'(\bmu)\delta^{(k-1)}(\bsigma)
$,
so it remains to employ our Leibniz identities to calculate (when $k\neq 1$)
\begin{multline*}
\D_\bmu^2\big(\delta(\bmu)\delta^{(k-2)}(\bsigma)\big)=(d-2k)\D_\bmu\big(\delta'(\bmu)\delta^{(k-2)}(\bsigma)\big)\\
=(d-2k)\big[
\delta(\bmu)\delta^{(k)}(\bsigma)
+(d-2k)\langle\bsigma,\bmu\rangle\delta'(\bmu)\delta^{(k-1)}(\bsigma)
 \\
+(d-2k-1)\delta''(\bmu)\delta^{(k-2)}(\bsigma)\big]\, ,
\end{multline*}
where again the last term is absent when $k=1$.
In the above we used  the identity
$$\, 
\bmu\Langle\delta'(\bmu),\delta^{(k-2)}(\bsigma)\Rangle
\!=-2\langle\bsigma,\bmu\rangle\delta'(\bmu)\delta^{(k-1)}(\bsigma)
-\tfrac{2}{d-2k}\delta(\bmu)\delta^{(k)}(\bsigma) 
-\tfrac{3(k-1)}{d-6} \delta''(\bmu)
\delta^{(k-2)}(\bsigma) \, ,
$$ 
which can be derived from  Equation~\nn{thebracket}. However once again, the display before last in fact holds for $k\in \{2,\ldots, d-3\}$.

Simple algebra now gives 
\begin{multline*}
\delta(\bmu)
\delta^{(k)}(\bsigma)=
\frac{1}{d-2k-2}\Big(\Big[
\frac{d-2k-1}{d-k-2}
\D_\bsigmas-\langle\bsigma,\bmu\rangle\D_\bmus
\Big]\big(\delta(\bmu)\delta^{(k-1)}(\bsigma)\big)\\
-\frac{k-1}{(d-k-2)(d-2k)}\D_\bmus^2
\big(\delta(\bmu)\delta^{(k-2)}(\bsigma)\big)\Big)~.
\end{multline*}
Here, using that 
$\D_\bmu\big(\langle\bsigma,\bmu\rangle  
\delta(\bmu) \delta^{(k-1)}(\bsigma)\big) =0$
together with the Leibniz rule \eqref{Leibniz} (for which one needs to use that $\bmu\Langle\langle\bsigma,\bmu\rangle, \delta(\bmu)\delta^{(k-1)}(\bsigma)\Rangle=-(\widehat \D_\bmu \langle\bsigma,\bmu\rangle)\,  \delta(\bmu)\delta^{(k-1)}(\bsigma)\hh$),  
we can then rewrite the second term in the above
expression as 
\begin{equation*}
\langle\bsigma,\bmu\rangle \D_\bmu 
\big( \delta(\bmu) \delta^{(k-1)}(\bsigma)\big)
=- (d-2k-2) \big(\widehat \D_\bmu \langle\bsigma,\bmu\rangle\big)
\delta(\bmu) \delta^{(k-1)}(\bsigma)~,
\end{equation*}
while (for $k\neq 1$) the third term can be combined with the first two
by means of $\bsigma\delta^{(k-1)}(\bsigma)=-(k-1)\delta^{(k-2)}(\bsigma)$,
producing 
\begin{equation*}
\begin{split}
\delta(\bmu)\delta^{(k)}(\bsigma)=
\frac{1}{d-k-2}\Bigg[\frac{d-2k-1}{d-2k-2} \, \D_\bsigmas 
+(d-k-2)\big(\widehat \D_\bmu \langle\bsigma,\bmu\rangle\big)\hspace{1cm}
\\\hspace{1cm}
+\, \frac{\D_\bmus^2\circ \bsigma}{(d-2k)(d-2k-2)} \Bigg]
\big(\delta(\bmu)\delta^{(k-1)}(\bsigma)\big)\, .
\end{split}
\end{equation*}
Expressing the factors within the square bracket in terms
of the weight $w=-k-1$, one identifies the formal adjoint 
operator \eqref{LTdagger}, and thus completes the proof.

\end{proof}

\begin{remark}
Proposition~\ref{LST} may also be read as saying
$$
\frac{1}{d-k-2}
\big(\D_\bsigma^{\!\rm T}\big)^\dagger 
\, 
\big(\delta(\bmu)\delta^{(k-1)}(\bsigma)\big) =
\frac{1}{d-k-1}\,  \delta(\bmu)\, \D_\bsigmas\delta^{(k-1)}(\bsigma)\, .
$$
Notice that the right hand side integrated against a density-valued test function $\bm f$ is independent of the extension of $\bm f$ off of the zero locus of $\bmu$. This explains why the formal adjoint of a tangential operator
must appear on the left hand side. Moreover, the above display
suggests an obvious extension to higher codimension problems where the anchoring submanifold is the zero locus of a set of defining densities $\{\bmu_1,\ldots,\bmu_\ell\}$ and one deals with products of delta distributions $\delta(\bmu_1)\cdots \delta(\bmu_\ell)$.
\end{remark}

We now develop the  integrated analog of Proposition~\ref{LST} and extend this to include  critical weights.

\begin{proposition}\label{bigcor}
Let $\bsigma$ be a unit defining density 
and $\bmu$ a minimal unit defining density.
Then if $\bm f$ is a weight $w=k+2-d$ density with 
$k\in
 \{1,\ldots, d-3\}$ and $d\geq 4$,
\begin{equation}\label{master2}
\int_M 
\delta(\bmu)\delta^{(k)}(\bsigma) 
\bm f=\frac{1}{d-k-2}\, 
\int_M \delta(\bmu)\delta^{(k-1)}(\bsigma) \, 
\D_\bsigma^{\!\rm T}\, \bm f\, .
\end{equation}
\end{proposition}

\begin{proof}
The proof for non-critical weights
$w\neq 1-\tfrac d2,2-\tfrac d2$ follows
directly from Proposition~\ref{LST}.
For critical weights, first observe that
by virtue of Equation~\nn{Dondel} and formal self-adjointness of $\D_\bsigmas$ we have
\begin{equation}\label{Startpt}
\int_M \delta(\bmu)\delta^{(k)}(\bsigma) \bm f
=\frac1{d-k-1}\int_M \delta^{(k-1)}(\bsigma) \D_\bsigmas \!\big(\delta(\bmu)\bm f\big)\, .
\end{equation}
We now break the computation of $ \D_\bsigmas \!\big(\delta(\bmu)\bm f\big)$ into the two cases of interest.

\medskip

{\it Case (i)} $w=1-\frac d2$ and thus $k=\frac d2-1$:
We begin by computing $
\D_\bsigma\big(\delta(\bmu)\bm f\big)
$ using the method of proof of Lemma~\ref{crittheta}. We pick a  scale $g\in \cc$ and find \begin{multline}
-2f (\nabla_n-\rho_\sigma)\delta(\mu)
-2 \delta(\mu) (\nabla_n+[1-\tfrac d2]\rho_\sigma)f
-\sigma (\Delta^g-\tfrac d2 J)\big(\delta(\mu)f\big)\\[1mm]
=-2f(m.n+\mu \rho_\sigma+\sigma \rho_\mu)\delta'(\mu)
-2 \delta(\mu) (\nabla_n+[1-\tfrac d2]\rho_\sigma)f
-\sigma (m^2+2\mu \rho_\mu )\delta''(\mu)f
\\
-2\sigma \delta'(\mu) (\nabla_m+[1-\tfrac d2]\rho_\mu) f
-\sigma \delta(\mu)  (\Delta^g+[1-\tfrac d2] J) f\, .
\label{just a comp}
\end{multline}
Thus, using that $\bsigma$ and $\bmu$ are unit and minimal unit defining densities respectively, 
we have
$$
\D_\bsigma\big(\delta(\bmu)\bm f\big)=
\delta(\bmu)(\D_\bsigmas 
-2\widetilde\D_\bsigmas) \bm f 
-2\delta'(\bmu)\big(\bsigma \hh\widetilde\D_\bmus \bm  +
 \langle\bsigma,\bmu\rangle \big)\bm f
-\bsigma \delta''(\bmu) \bm f \, .
$$
Using the identities $\bm \sigma \delta^{(\frac d2-2)}(\bsigma)=-\frac12(d-4)\, \delta^{(\frac d2 -3)}(\bsigma)$ and $\D_\bsigma \bm f = -\bsigma \square \bm f$,
we so far have
\begin{multline*}
\int_M \delta(\bmu)\delta^{(\frac d2-1)}(\bsigma) \bm f
=\frac2{d}\int_M
\Big(-2\delta^{(\frac d2-2)}(\bsigma)\big[
\delta(\bmu)\hh\widetilde\D_\bsigmas 
+ 
\delta'(\bmu)
  \langle\bsigma,\bmu\rangle\big] \bm f
\\[1mm]
\qquad\qquad\qquad
+\frac12 (d-4)\hh \delta^{(\frac d2-3)}(\bsigma)\Big[
\delta(\bmu)\hh \square
+2\delta'(\bmu)\hh\widetilde\D_\bmus 
+\delta''(\bmu)\Big]
\bm f
\Big)\, .
\end{multline*}
 We can handle the terms involving $\delta^{(\frac d2-2)}(\bsigma)
\delta'(\bmu)$
and $\delta^{(\frac d2-3)}(\bsigma)
\delta'(\bmu)$ in the above expression
 using Equations~\nn{snappo}
 and~\nn{Lmdmds}, respectively;  remembering that $ \langle\bsigma,\bmu\rangle|_\Xi={\mathcal O}(\bsigma^{d-1})$ and that $\D_\bmus$ is formally self-adjoint, we thus have 
\begin{multline*}
\int_M \delta(\bmu)\delta^{(\frac d2-1)}(\bsigma) \bm f
=\frac2{d}\int_M
\Big(
-2\delta^{(\frac d2-2)}(\bsigma)\delta(\bmu)\Big[\hh
\widetilde\D_\bsigmas 
-
\big(\Lodz_\bmus  \langle\bsigma,\bmu\rangle\big)\Big] \bm f
\\[1mm]
\qquad\qquad
+\frac12(d-4)\hh \delta^{(\frac d2-3)}(\bsigma)\Big[\delta(\bmu)
\big(\square
-2\Lodz_\bmu\hh\widetilde\D_\bmus \big)
+\delta''(\bmu)\Big]
\bm f
\Big)\, .
\end{multline*}
To handle the last term above,
we recall that $\D_\bmu \delta'(\bmu)=(d-3)\delta''(\bmu)$ and so 
perform a  computation similar to that of Display~\nn{just a comp} to obtain 
\begin{multline*}
\D_\bmus\big(\delta^{(\frac d2-3)}(\bsigma)\bm f\big)=
\delta^{(\frac d2-3)}(\bsigma)\big(\D_\bmus
-(d-4)\hh\widetilde\D_\bmus\big) \bm f \\[1mm]
-\hh \delta^{(\frac d2-2)}(\bsigma)\big(2\bmu\hh \widetilde\D_\bsigmas +(d-4)
 \langle\bsigma,\bmu\rangle \big)\bm f
-\bmu\hh  \delta^{(\frac d2-1)}(\bsigma)\bm f  \, .
\end{multline*}
Using Equations~\nn{Dondel},~\nn{snappo}
 and~\nn{Lmdmds} again, as well as $\bmu \delta'(\bmu)=-\delta(\bmu)$ and $\delta(\bmu) \hh \Lodz_\bmus \D_\bmus \bm f = -\delta(\bmu) \square \bm f$, gives
 \begin{multline*}
 \int_M \delta^{(\frac d2-3)}(\bsigma)\delta''(\bmu) \bm f =
 \frac1{d-3} \int_M \delta(\bmu)
 \Big(
 \delta^{(\frac d2-3)}(\bsigma) \big(\square
+(d-4)
\hh\Lodz_\bmus\widetilde\D_\bmus\big) \bm f 
\\[1mm]
\qquad\qquad 
+
\delta^{(\frac d2-2)}(\bsigma) 
\big(2 \widetilde\D_\bsigmas \bm f
 +(d-4)\bm f  \Lodz_\bmus
 \langle\bsigma,\bmu\rangle \big)
 +\delta^{(\frac d2-1)}(\bsigma) \bm f
\Big)\, .
 \end{multline*}
 Collating the above calculations yields
\begin{align*}
 \int_M 
\delta(\bmu)\hh\delta^{(\frac d2-1)}(\bsigma) 
\bm f=
\frac{1}{d-2}
\int_M\delta(\bmu) \Big(\delta^{(\frac d2-2)}(\bsigma) 
\Big[
&-2\hh\widetilde\D_\bsigmas
+(d-2)\big(\Lodz_\bmu \langle\bsigma,\bmu\rangle\big)
\Big]\\
&-(d-4)\delta^{(\frac d2-3)}(\bsigma) \, 
\big[\Lodz_\bmu \widetilde \D_\bmu
-\square\big]\Big)\bm f~.
 \end{align*}
Here, using the identity 
$\delta^{(\frac d2-3)}(\bsigma)=
-\frac{2}{d-4}\bsigma \delta^{(\frac{d}{2}-2)}(\bsigma)$
and the fact that acting on a critical density 
$\D_\bsigma \bm f=-\bsigma \square \bm f$, we can 
rewrite the previous display as
\begin{align*}
\int_M \delta(\bmu)\delta^{(\frac d2-1)}(\bsigma)\bm f
=\frac{2}{d-2} \int_M\!\delta(\bmu)\delta^{(\frac d2-2)}(\bsigma) 
\Big[ \D_\bsigma 
+\tfrac{d-2}{2}\big(\Lodz_\bmu \langle\bsigma,\bmu\rangle\big)
+\big(\bsigma \widehat \D_\bmu \widetilde \D_\bmu\!
-\!\widetilde \D_\bsigma \big)\Big]\bm f .
 \end{align*}
Then  the quoted 
result follows  by virtue of Equation \eqref{LTcrit1} in 
Definition~\ref{LTdef}.
  
\medskip 

{\it Case (ii)} $w=2-\frac d2$ and $k=\frac d2$: Note that this case only pertains to $d>4$. We begin by first noting that
$$
\D_\bsigmas  \big(\delta(\bmu)\bm f\big)
=-\bsigma \square\big(\delta(\bmu)\bm f\big)\, .
$$
Now the proof again proceeds along the lines of
Lemma~\ref{crittheta}. Choosing a scale $g\in\cc$ we compute $\square \big(\delta(\bmu)\bm f\big)$ and find (using $\delta(\mu)=-\mu\delta'(\mu)$)
\begin{multline*}
\big(\Delta+[1-\tfrac d2] J\big)\big(\delta(\mu) f\big)=
\big(m^2 \delta''(\mu)+\nabla.m\hh \delta'(\mu)\big) f
+ 2 \delta'(\mu) \nabla_m f 
+\delta(\mu)\hh  \big(\Delta+[1-\tfrac d2] J\big) f\, .\\[1mm]
=\delta''(\mu)(m^2+2\rho_\mu\mu)f + \delta'(\mu)\big[2(\nabla_m+[2-\tfrac d2]\rho_\mu) f -
\mu (\Delta + [2-\tfrac d2]J)f\big]\, .
\end{multline*}
Thus, using the unit conformal property of $\bmu$ and Equation~\nn{Dondel},
it follows that
$$
\, 
\, \D_\bsigmas \! \big(\delta(\bmu)\bm f\big)=
-\bsigma \delta''(\bmu)\bm f
-2\bsigma \delta'(\bmu) \Lodz_\mu  \bm f=-
\frac\bsigma{d-2}\hh
\Big(
\frac{1}{d-3}\hh
 \bm f
\D_\bmus 
+2 (\Lodz_\mu  \bm f) \Big)\D_\bmu\delta(\bmu)\, .
$$
In turn via formal self-adjointness of $\D_\bmus$ and the identity $\bsigma \delta^{(\frac d2-1)}(\bsigma)=-\frac12 (d-2)\delta^{(\frac d2-2)}(\bsigma)$, it follows from Equation~\nn{Startpt} that
\begin{multline}\label{1/2way}
\int_M \delta(\bmu)\delta^{(\frac d2)}(\bsigma) \bm f
=\frac{1}{d-2}\int_M 
\delta(\bmu)
\Big(\frac{1}{d-3}\D_\bmus^2\big(
\delta^{(\frac d2-2)}
(\bsigma)\bm f\big)
+2 
\D_\bmu\big(\delta^{(\frac d2-2)}
(\bsigma)
 \Lodz_\bmus  \bm f\big)
\Big)
\, .
\end{multline}
Care is now needed because the delta distribution $\delta^{(\frac d2-2)}(\bsigma)$ has Yamabe weight $1-\frac d2$. A simple computation in a choice of scale shows that because the defining density $\bmu$ is minimal unit, we have 
\begin{equation}\label{libel}
\delta(\bmu)
\bm \updelta_{\rm R}^{\sss\Xi}
\delta^{(k)}(\bsigma)=0\, ,
\end{equation}
for any positive integer $k\leq d-3$. Here $\bm \updelta_{\rm R}^{\sss\Xi}$ is the Robin operator defined in Equation~\nn{ROBIN}.
In particular this implies that $\delta(\bmu)\hh\widetilde\D_\bmus \hh\delta^{(\frac d2-2)}(\bsigma)=0$.
Then using this fact in Equation~\nn{Leibniz''} applied to the last term in Equation~\nn{1/2way} we have
\begin{multline*}
 \Lodz_\bmu \big(\delta^{(\frac d2-2)}(\bsigma)\hh \Lodz_\bmus\bm f\big)=
\Big(\widetilde\D_\bmus\delta^{(\frac d2-2)}(\bsigma)-\tfrac 12\bmu\square\delta^{(\frac d2-2)}(\bsigma)\Big)\Lodz_\bmus\bm f 
\\[1mm]+ \delta^{(\frac d2-2)}(\bsigma)\hh \big(\widetilde\D_\bmus
\Lodz_\bmus \bm f -\tfrac 12\bmu\square\Lodz_\bmus \bm f \big) 
+\bmu \, \,  
\Langle\hspace{-3.35mm}- \hspace{-.8mm} \delta^{(\frac d2-2)}(\bsigma), \Lodz_\bmus\bm f\hh\rangle\hspace{-2.2mm}-\, ,
\end{multline*}
so
$$
\delta(\bmu)\D_\bmu\big(\delta^{(\frac d2-2)}
(\bsigma)
 \Lodz_\bmus  \bm f\big)
=-(d-2)\delta(\bmu)\delta^{(\frac d2-2)}
(\bsigma)\widetilde\D_\bmu\Lodz_\bmu \bm f\, .
$$
This handles the last term in Equation~\nn{1/2way} so we now have
\begin{multline*}\label{2/3way}
\int_M \delta(\bmu)\delta^{(\frac d2)}(\bsigma) \bm f
=\frac{1}{d-2}\int_M 
\delta(\bmu)
\Big(\frac{1}{d-3}\D_\bmus^2\big(
\delta^{(\frac d2-2)}
(\bsigma)\bm f\big)
-2(d-2) 
\delta^{(\frac d2-2)}
(\bsigma)
\widetilde\D_\bmus
 \Lodz_\bmus  \bm f
\Big)
\, .
\end{multline*}
To treat the first term on the right hand side of the above display, we first note that
$\D_\bmus^2\big(
\delta^{(\frac d2-2)}
(\bsigma)\bm f\big)
=(d-2)(d-4)\hh\Lodz_\bmu^2
\big(
\delta^{(\frac d2-2)}
(\bsigma)\bm f\big)
$, so
\begin{equation}
\label{stuff}
\delta(\bmu) \D_\bmus^2\big(
\delta^{(\frac d2-2)}(\bsigma)\bm f\big)
=(d-2)(d-4)\hh \delta(\bmu)\hh \bm\updelta_{\rm R}^{\sss\Xi}\hh 
\Lodz_\bmus
\big(
\delta^{(\frac d2-2)}(\bsigma)\bm f\big)\, .
\end{equation}
Note that because the Robin operator $\bm{\updelta}_{\rm R}^{\sss\Xi}$ is first order, it obeys the standard Leibniz rule.
Then we  use Equation~\nn{Leibniz'} to compute 
\begin{multline*}
\Lodz_\mu \big(\delta^{(\frac d2-2)}(\bsigma)\hh \bm f\big)=
\Big(\big(\widetilde\D_\bmus-\tfrac1{d-4}\D_\bmus\big)\delta^{(\frac d2-2)}(\bsigma)\Big)\bm f + \delta^{(\frac d2-2)}(\bsigma)\hh \Lodz_\bmus \bm f +\frac{2\bmu}{d-4} \: \,  
\, \Langle\hspace{-2.6mm}- \hspace{-.1mm}\,\! \delta^{(\frac d2-2)}(\bsigma), \bm f\Rangle\\[1mm]
=
\frac{1}{d-4} \bmu \delta^{(\frac d2)}(\bsigma) \bm f
+
\delta^{(\frac d2-2)}(\bsigma)\hh \Lodz_\bmu \bm f
+
\delta^{(\frac d2-1)} 
(\bsigma)
\Big[
\frac{2\bmu}{d-4} \hh \Lodz_\bsigma \bm f
+
\langle \bsigma,\bmu\rangle\hh \bm f\Big]\, .
\end{multline*}
The terms in square brackets in the last line above come from the first and last terms on the right hand side of the previous line. This requires a computation in a choice of scale of the type performed above, relying on Equation~\nn{takeoff}.
The remaining terms on 
 last line above used $\D_\bmus \delta^{(\frac d2-2)}(\bsigma)=-\bmu \square\delta^{(\frac d2-2)}(\bsigma)$ and 
that $\bm\sigma$ is unit defining so 
$$
\square\,  \delta^{(\frac d2-2)}(\bsigma) = \delta^{(\frac d2)}(\bsigma) \, .
$$
Now using Equations~\nn{stuff} and~\nn{libel},
%
and $\delta(\bmu)\delta_{\, \rm R}^\Xi\bmu=\delta(\bmu)$ it follows that
\begin{multline*}
\delta(\bmu) \D_\bmus^2\big(
\delta^{(\frac d2-2)}(\bsigma)\bm f\big)
=(d-2) 
\delta(\bmu)\delta^{(\frac d2)}(\bsigma) \bm f+
(d-2)(d-4)\hh \delta(\bmu)\hh
\delta^{(\frac d2-2)}(\bsigma)\hh \widetilde\D_\bmus\Lodz_\bmu \bm f
\\[1mm]
+(d-2)\delta(\bmu)
\delta^{(\frac d2-1)} 
(\bsigma)
\Big[
  2\Lodz_\bsigma \bm f
+(d-4)\bm f\Lodz_\bmus
\langle \bsigma,\bmu\rangle\hh \Big]
\, .
\end{multline*}
Orchestrating the above gives
\begin{multline*}
\int_M \delta(\bmu)\delta^{(\frac d2)}(\bsigma) \bm f
=\int_M 
\delta(\bmu)
\Big(-
\frac{d-2}{d-4}
\hh \hh
\delta^{(\frac d2-2)}(\bsigma)\hh \widetilde\D_\bmus\Lodz_\bmu \bm f
\\[1mm]
+\frac{1}{d-4}
\delta^{(\frac d2-1)} 
(\bsigma)
\Big[
  2\, \Lodz_\bsigma \bm f
+(d-4)\bm f\Lodz_\bmus
\langle \bsigma,\bmu\rangle\hh \Big]
\Big)
\, .
\end{multline*}
Using the identity
$\delta^{(\frac d2-2)}(\bsigma)=
-\frac{2}{d-2}\bsigma \delta^{(\frac{d}{2}-1)}(\bsigma)$
one recognizes (from Equation~\eqref{LTcrit2})
 the action of the tangential operator on
$\bm f$. The final result then follows. 
\end{proof}


We have now established the two main recursive relations
needed to compute integrals of the type~\eqref{Intpq} and in turn the regulated volume and 
area expansions in Equations~\nn{reg_vol} and~\nn{AREA}. The   proposition below gives our result for such integrals. Its proof follows simply by iterating
the recursions given in Equations~\eqref{master1}~and~\eqref{master2}.
\begin{proposition}\label{cool}
Let $\bsigma$ be a unit defining density and 
$\bmu$ a minimal unit defining density.
Then, if $\bm f$ is a weight $w\neq0$ density
and $1\leq k\leq d-3, d-2$, respectively, the following
relations hold
\begin{align*}
\int_M \delta(\bmu)\, \delta^{(k)}(\bsigma)\bm f
&=~\frac{(d-k-3)!}{(d-3)!} \int_M \delta(\bmu)\,\delta(\bsigma)
(\D^{\!\rm T}_\bsigma)^k \bm f ~,\\[3mm]
\int_M \theta(\bmu)\, \delta^{(k)}(\bsigma)\bm f
&=~\frac{(d-k-2)!}{(d-2)!}\,  \Bigg[ \int_M \theta(\bmu) \delta(\bsigma) \, \D_\bsigmas^{k}\bm f \\
&\hspace{2.7cm}+(d-2)\, \sum_{j=0}^{k-2}\int_M
\delta(\bmu) \delta(\bsigma) \, (\D^{\!\rm T}_\bsigma)^j\, 
\D^\prime_\bmu\,  \D^{k-2-j}_\bsigma\bm f \Bigg]~,
\end{align*}
where the second term on the right hand side of the last displayed equality is absent when~$k=1$.
\end{proposition}

\section{Holographic Formul\ae} 
\label{Sec:HolFor}
\subsection{Divergences}
We now deduce  results in arbitrary dimensions for the 
volume and area divergences of the regulated volume~$\Vol_\varepsilon$ and area~$\Area_\varepsilon$ defined in Section~\ref{Sec:RenVA}. Recall that in distributional form their coefficients are
given (from Equations~\nn{volumedivergences}
and~\nn{areadivergences}) by: 
\begin{align*}
v_k=\frac{(-1)^{d-1-k}}{(d-1-k)!k} \, \int_M\frac{
\theta(\bmu)\, \delta^{(d-1-k)}(\bsigma)}{\btau^k}\, , 
\quad
a_k=\frac{(-1)^{d-2-k}}{(d-2-k)!k}
\int_M \frac{\delta(\bmu)\delta^{(d-2-k)}(\bsigma)}{\btau^{k}}\, .
\end{align*}
The following pair of theorems express the  above expressions as integrals over $\widetilde\Sigma$ and $\Lambda$,
with local holographic formul\ae\ for integrands.
 These follow immediately from Proposition~\ref{cool} by setting    $\bm f=\btau^{-k}$, where the regulator $\btau$ is an  arbitrary
true scale.

\begin{theorem}\label{Vdiv} 
The 
divergences  in the regulated volume $\Vol_\varepsilon$
determined by the sequence of embeddings $\Lambda \hookrightarrow \Sigma \hookrightarrow (M,\cc)$ and the regulating hypersurface $\Sigma_\varepsilon$
are local integrals given by
\begin{align*}
{\rm Poles}(\Vol_\varepsilon)=\sum_{k={d-1}}^1
\frac{(-1)^{d-k-1}}{(d-k-1)!\hh k\hh \varepsilon^k} \, \frac{(k-1)!}{(d-2)!}\, 
\Bigg[\int_{\widetilde\Sigma} {\bm v}_k 
+&(d-2)\int_\Lambda {\bm  v}'_{k}\hh \Bigg]~.
\end{align*}
The integrands are given 
in terms of
the  unit  and minimal unit  defining 
densities $\bsigma$ and~$\bmu$ for $\Sigma$ and $\Xi$, respectively, and the 
regulator $\btau$, by the holographic formul\ae
\begin{equation*}
\bm v_k
=\D_\bsigmas^{d-1-k}\btau^{-k}\Big|_{\widetilde\Sigma}~,\quad
\bm v'_{k\leq d-3}=\sum_{j=0}^{d-k-3}
(\D_\bsigma^{\!\rm T})^j\, \D_\bmu^\prime \, \D^{d-k-3-j}_\bsigma \btau^{-k}
\Big|_{\Lambda}\, ,\quad \bm v'_{d-2}=0=\bm v'_{d-1} ~.  
\end{equation*}


\end{theorem}


\begin{theorem}\label{Adiv}
The 
divergences  in the regulated area $\Area_\varepsilon$
determined by the sequence of embeddings $\Lambda \hookrightarrow \Sigma \hookrightarrow (M,\cc)$ and the regulating hypersurface $\Sigma_\varepsilon$,
are the local integrals
\begin{equation*}
{\rm Poles}(\Area_\varepsilon)=\sum_{k={d-2}}^1
\frac{(-1)^{d-k-2}}{(d-k-2)!k\varepsilon^k} \,
 \frac{(k-1)!}{(d-3)!} \int_\Lambda {\bm  a}_k \, .
\end{equation*}
The integrands are given  holographically
in terms of
the  unit  and minimal unit  defining 
densities $\bsigma$ and~$\bmu$ for $\Sigma$ and $\Xi$, respectively, and the 
regulator $\btau$, by
\begin{equation*}
\bm a_k=(\D_\bsigma^{\!\rm T})^{d-k-2} \, \btau^{-k}
\Big|_{\Lambda}\, .
\end{equation*}
%
\end{theorem}
\begin{remark}
Note that the results stated in the two previous Theorems hold 
upon performing delta integrations according to Equation~\eqref{deltasurface}; 
the measure factors $\sqrt{{\bm {\mathcal S}}}$ usually incurred when integrating along submanifolds, 
with respect to their induced conformal structures, are both one because 
 the defining densities $\bsigma$ 
and $\bmu$ are  conformal unit and minimal conformal unit, 
respectively.
\end{remark}

\subsection{Anomalies}
\label{Sec:Anomaly}
We now study  the critical $\log \varepsilon$ 
divergences appearing in the regulated volume and area expansions. In particular, in this section we prove Theorems~\ref{Vanomaly} and~\ref{Aanomaly}.
In arbitrary dimensions, the volume and area anomalies are given, 
respectively, by the integrals 
\begin{equation*}
\mathcal V=\frac{(-1)^d}{(d-1)!} \int_M
\theta(\bmu)\, \delta^{(d-1)}(\bsigma) ~,\quad 
\mathcal A=\frac{(-1)^{d-1}}{(d-2)!}
\int_M\delta(\bmu)\delta^{(d-2)}(\bsigma)\, ;
\end{equation*}
see Section~\ref{Sec:RenVA}.
In the volume case, the recursion~\eqref{Dondel}
gives 
%
%
%
%
\begin{equation*}
\D_\bsigmas \delta^{(d-2)}(\bsigma) = 0 ~,
\end{equation*}
and so cannot be used to handle $\delta^{(d-1)}(\bsigma)$.
Instead one must introduce  
a $\log$ 
density into the problem (see~\cite{GWvol} and~\cite{GWvolII}).
This  relies on the following proposition.
from~\cite[Proposition 4.4]{GWvolII}
specialized to assume that $\bsigma$ is a unit defining density and written in the notation of Section~\ref{LRs}. 
%
\begin{proposition}\label{logt}
Let $f$ be a  weight zero density,
where $f\delta^{(d-1)}(\bm \sigma)$ is compactly supported,
and let~$\bm \t$ be any 
true scale, then 
%
\begin{equation*}
\int_M f \delta^{(d-1)}(\bm \sigma) =  
-\int_M \delta^{(d-2)}(\bm \sigma) 
\big(f \D_\bsigma \log \bm \t 
+ \btau^{-1} \D_{\bsigmas,\btaus} f \big) ~.
\end{equation*}
\end{proposition}
\noindent
Note that the proof of~\cite{GWvolII} 
was for the case that $f$ is a smooth function, but it is easily seen to  
   extend  to the 
particular instance where $f=\theta(\bmu)$, because $\widetilde \Sigma$ is 
compact.
 Thus
the proposition can be applied to recast the 
$\mathcal V$-integral as   
\begin{equation}\label{Vnabla}
\mathcal V =- \frac{(-1)^d}{(d-1)!} \int_M 
\delta^{(d-2)}(\bsigma)\Big(\theta(\bmu) \D_\bsigma \log \bm \t 
+ \btau^{-1} \D_{\bsigmas,\btaus} \theta(\bmu) \Big)  \, .
\end{equation}
The differentiated Dirac delta distribution 
$\delta^{(d-2)}(\bsigma)$  in the first term on the right hand side above may now handled using Equation~\eqref{master1} and Proposition~\ref{cool}, while the second term is discussed below.

As for the area anomaly $\mathcal A$, 
the distributional identity $\delta^{(d-2)}(\bsigma)=
\frac{1}{1-d} \bsigma \delta^{(d-1)}(\bsigma)$
implies that
\begin{equation*}
\mathcal A = \frac{(-1)^d}{(d-1)!}
\int_M \bsigma\delta(\bmu)\delta^{(d-1)}(\bsigma)\, ,
\end{equation*}
which in turn can be rewritten using 
Proposition~\ref{logt} for the case where 
$f = \bsigma \delta(\bmu)$, as
\begin{equation}\label{Anabla}
\mathcal A= \frac{(-1)^{d-1}}{(d-1)!} \int_M 
\delta^{(d-2)}(\bsigma)\Big(\bsigma\delta(\bmu) \D_\bsigma \log \bm \t 
+ \btau^{-1} \D_{\bsigmas,\btaus} \big(\bsigma\delta(\bmu)\big) \Big)  \, .
\end{equation}
Here we have again extended the
scope of Proposition~\ref{logt}
to the case where $f$ is a weight zero, distribution-valued density.
The proof applies {\it mutatis mutandis} to this case because $\Lambda=\partial\widetilde \Sigma$ is closed.
Again, the first term on the right hand side can be treated using Proposition~\ref{cool}.
For the remaining terms involving  
$\D_{\bsigma, \btau}$ operators in Equations~\nn{Vnabla} and~\nn{Anabla}, we need the 
following 
lemma:
\begin{lemma}\label{Lst}
Let $\bsigma$ be a weight one density,  $\bmu$ a hypersurface defining density, 
and $\bm \t$ a true scale. Then
\begin{align*}
\btau^{-1} \D_{\bsigmas,\btaus}\theta(\bmu) &=
\delta(\bmu)\Big( \langle\bsigma,\bmu\rangle-
 \frac{\bsigma}{d-2} \D_\bmu \log\bm \tau\Big)\, , \\[2mm]
\btau^{-1} \D_{\bsigmas,\btaus} (\bsigma\delta(\bmu)) &=
\delta(\bmu) \Big(\bm{\mathcal S}_\bsigmas- \frac{\bsigma}{d-2} \D_\bsigma \log \bm \t  \Big) 
+ \bsigma\delta'(\bmu)\Big(\langle\bsigma,\bmu\rangle -\frac{\bsigma}{d-2} \D_\bmu \log\bm \tau \Big)  \, .
\end{align*}
\end{lemma}
\begin{proof}
The first equality can be demonstrated directly from the definition 
of the operator~$\D_{\bsigmas,\btaus}$ acting on the weight zero
density $\theta(\bmu)$. Indeed,~\eqref{DD} implies that
\begin{equation*}
\btau^{-1} \D_{\bsigmas,\btaus} \theta(\bmu) 
=\big[g; \nabla_n \theta(\mu) - \sigma\tau^{-1} \nabla_k \theta(\mu) \big] 
= \big[g; \delta (\mu) \big( n.m - \sigma\tau^{-1} m.k \big) \big]\, ,
\end{equation*}
where we have denoted the triple of one-forms 
$(m, n, k):=(\ext \mu, \ext \sigma, \ext \tau)$  and used the distributional identity $\theta'(x)=\delta(x)$.
Remembering the distributional identity $x\delta(x)=0$, and that
%
 $
 \langle\bsigma,\bmu\rangle
:=\big[g; \langle\sigma,\mu\rangle  \big]
=\big[g; \sigma \rho_\mu+ n.m  + \mu \rho_\sigma\big]
$
%
as well as the definition of $\widehat\D_\bmu \log \btau$ in Equation~\nn{Llog1} and~\nn{LODZ},
we can write the expression 
above~as
\begin{equation*}
\btau^{-1} \D_{\bsigmas,\btaus} \theta(\bmu) 
=\big[g; \delta (\mu) \big( \langle\sigma,\mu\rangle 
- \sigma( \nabla_m \log \tau + \rho_\mu) \big) \big]
= \delta (\bmu) \Big( \langle\bsigma,\bmu\rangle 
- \bsigma \widehat\D_\bmu \log \btau \Big)\, ,
\end{equation*}
as required.

The second equality of the lemma can be obtained via a similar route: From~\eqref{DD} and recalling that the product
$\bsigma \delta(\bmu)$ has weight zero, it follows that
\begin{equation*}
\btau^{-1} \D_{\bsigmas,\btaus} (\bsigma \delta(\bmu))
=\big[g; \nabla_n (\sigma \delta(\mu))  
- \sigma\tau^{-1} \nabla_k (\sigma \delta(\mu)) \, \big]~.
\end{equation*}
%
The above can be further developed using our  distributional calculus:
\begin{align*}
\nabla_n &(\sigma \delta(\mu)) - \sigma\tau^{-1} \nabla_k (\sigma \delta(\mu))
= \delta(\mu) \big(\nabla_n \sigma - \sigma \tau^{-1} \nabla_k \sigma \big)
+  \sigma \big(\nabla_n \delta(\mu) - \sigma \tau^{-1} \nabla_k \delta(\mu) \big)\\
&=\delta(\mu) \big(n^2- \sigma \nabla_n \log\tau \big)
+ \sigma \delta'(\mu) \big(n.m - \sigma \nabla_m \log \tau \big) \\
&= \delta(\mu) \big(\mathcal S_\sigma - \sigma\rho_\sigma-\sigma(\nabla_n \log\tau+\rho_\sigma)\big)
+ \sigma \delta'(\mu) \big(\langle\sigma,\mu\rangle -\mu\rho_\sigma- \sigma (\nabla_m \log \tau+\rho_\mu)\big) ~.
\end{align*}
Observe that the second and sixth terms above 
cancel each other.
Next, using the action of the Laplace--Robin operator 
on log-densities \eqref{Ldef}, the last expression of the 
previous display may  be written as
\begin{align*}
\delta(\mu) \Big[
\mathcal S_\sigma\! -\! \sigma \big(\widehat {\rm L}_\bsigma \log\tau\!-\! \frac{\sigma}{d} (\Delta\log\tau + J) \big)\Big]
+ \sigma \delta'(\mu) \Big[
\langle\sigma,\mu\rangle\! - \!\sigma \big(\widehat {\rm L}_\bmu \log\tau
\!-\! \frac{\mu}{d} (\Delta\log\tau + J) \big)\Big] \, .
\end{align*} 
Once again, thanks to the identity $x\delta'(x)=-\delta (x)$, the terms 
with coefficient $1/d$ cancel, 
so we obtain
%
\begin{equation*}
\btau^{-1} \D_{\bsigmas,\btaus} (\bsigma \delta(\bmu))
=\delta(\bmu) \Big(\mathcal S_\bsigma - \bsigma\widehat \D_\bsigma \log \bm \t  \Big)
+\bsigma\delta'(\bmu)\Big(\langle\bsigma,\bmu\rangle-\bsigma\widehat\D_\bmu \log\bm \tau\Big)  ~,
\end{equation*}
again as required.
\end{proof}

We are now ready to complete the proofs of our first two main results.

\begin{proof}[Proof of Theorems~\ref{Vanomaly} and~\ref{Aanomaly}]
We begin with the volume anomaly. By virtue of Lemma \ref{Lst}, 
the volume anomaly integral \eqref{Vnabla} 
may be written as
\begin{align*}
 \mathcal V&=  \frac{(-1)^{d-1}}{(d-1)!} \int_M 
 \delta^{(d-2)}(\bsigma)\Big(\theta(\bmu) \D_\bsigma \log \bm \t
 +\delta(\bmu) \langle\bsigma,\bmu\rangle
 -\frac{\bsigma}{d-2}\delta(\bmu) \D_\bmu \log\bm \tau \Big)~.
\end{align*}
Here the second term under the integral vanishes due to the 
minimal surface condition (as formulated in Equation~\nn{thebr}) while the last one can 
be rewritten  by means  of the identity 
$\bsigma \delta^{(d-2)}(\bsigma)=
-(d-2) \delta^{(d-3)}(\bsigma)$; this gives
%
%
\begin{equation}\label{Vfinal}
\mathcal V= \frac{(-1)^{d-1}}{(d-1)!} \int_M 
 \Big( \theta(\bmu) \delta^{(d-2)}(\bsigma) \D_\bsigma 
 +\delta(\bmu) \delta^{(d-3)}(\bsigma) \D_\bmu \Big) \log \bm \t ~.
 \end{equation}
Likewise,  Lemma~\ref{Lst} applied   to the surface 
anomaly integral~\eqref{Anabla} yields
\begin{align*}
\mathcal A= \frac{(-1)^{d-1}}{(d-1)!} \int_M 
\delta^{(d-2)}(\bsigma)\bigg[\bsigma\delta(\bmu) \D_\bsigma \log \bm \t 
+ \bsigma\delta'(\bmu)\Big(\langle\bsigma,\bmu\rangle
-&\frac{\bsigma}{d-2} \D_\bmu \log\bm \tau\Big)\\ 
+ \delta(\bmu)\Big(1-&\frac{\bsigma}{d-2} 
\D_\bsigma \log \bm \t \Big)\bigg]  ~,
\end{align*}
where we used that  $\bS_\bsigma=1+{\mathcal O}(\bsigma^d)$. Applying the same set of distributional identities as used earlier, the integral above 
can be re-expressed   as
\begin{align}\label{almostA}
\mathcal A=\frac{(-1)^{d}}{(d-1)!} \int_M \Big[
(d-3)\delta(\bmu)\delta^{(d-3)}(\bsigma)\D_\bsigma \log\bm\t 
&+(d-2) \delta'(\bmu) \delta^{(d-3)}(\bsigma)\langle\bsigma,\bmu\rangle\\
-\, \delta(\bmu) \delta^{(d-2)}(\bsigma)
&+(d-3)  \delta'(\bmu) \delta^{(d-4)}(\bsigma) \D_\bmu \log\bm\t\Big]~. \nonumber
\end{align}
Observe that the third integrand is 
proportional to the original anomaly~\eqref{arnomaly}, subtracting this from the left hand side yields 
$(1-\frac{1}{d-1})\mathcal A$ there. The second term on the right can 
be handled using~\eqref{Lmdmds}, 
which implies
\begin{equation*}
\delta'(\bmu)\delta^{(d-3)}(\bsigma) 
=-\frac{1}{d-2} \D_\bmu 
\big(\delta(\bmu)\delta^{(d-3)}(\bsigma) \big)~.
\end{equation*}
When $d\neq 4$, 
the very last term can also be treated using~\eqref{Lmdmds}.  When $d=4$, however, the 
density $\delta(\bsigma)$ has critical weight and the integrated formula~\eqref{snappo} is needed. This distinction between critical ($d=4$) and non--critical ($d\neq4$) 
weights fits with that for the tangential 
Laplace--Robin operator $\D^{\!\rm T}_\bsigma$ acting on $\log$--densities
as given in Definition~\ref{logLT}.
%
Indeed, by using this definition directly in 
Equation~\eqref{almostA} and performing computations that rely on formal self-adjointness of $\D_\bmus$ and the identity $\bsigma \delta^{(d-3)}(\bsigma) = -(d-3) \delta^{(d-4)}(\bsigma)$, 
we obtain
\begin{align}\label{Afinal}
\mathcal A= \frac{(-1)^{d-2}}{(d-2)!} \int_M \delta(\bmu)
 \delta^{(d-3)}(\bsigma) \D^{\!\rm T}_\bsigma \log\bm\t ~.
\end{align}
Equation~\nn{Afinal} is a nice analog of the key Lemma 3.8 of~\cite{GWvol}.

Applying Proposition~\ref{cool} to  
Equations \eqref{Vfinal} and \eqref{Afinal} 
completes the proof of our results for the volume and area anomalies.
\end{proof}

\medskip

The regulated volume and area anomalies ${\mathcal V}$ and ${\mathcal A}$ do not depend on the choice of regulator $\bm \tau$, even though their corresponding extrinsic $Q$ and $T$ curvature integrands do. Remember that the choice of a true scale $\btau=[g;1]$ is equivalent to a choice of 
   metric $g$ in the conformal class~$\cc$. 
  (In what follows we sometimes abbreviate our notation for the dependence of ~$Q$ and~$T$ curvatures and related operators on the underlying data that determines them, for example writing $\bm T_{\sss\Lambda}^g$ for $\bm T_{\Lambda\hookrightarrow \Sigma \hookrightarrow (M,\cc)}^{\btaus}$.) 
   Changing the regulator $\bm \tau$ to
a new true scale $\hat\btau=e^f\btau$ where $f\in C^\infty M$,  the identity $\log (e^f \btau)=\log \btau + f$ implies
\begin{align}\label{transform}
&\bm Q_{\Sigma \hookrightarrow (M,\cc)}^{\hat\btaus}
-\bm Q_{\Sigma \hookrightarrow (M,\cc)}^{\btaus}=\D^{d-1}_{\bsigmas} f\big|_\Sigma\, ,\nonumber
\\[2mm]
&\bm T_{\Lambda\hookrightarrow \Sigma \hookrightarrow (M,\cc)}^{\hat\btaus} -
\bm T_{\Lambda\hookrightarrow \Sigma \hookrightarrow (M,\cc)}^{\btaus} 
= \sum_{j=0}^{d-3} (\D^{\!\rm T}_\bsigma)^j\, 
\D^\prime_\bmu \bm \D_\bsigma^{d-j-3}f\hh\Big|_{\Lambda} \, ,
\\[2mm]
&
\bm Q_{\Lambda\hookrightarrow \Sigma \hookrightarrow (M,\cc)}^{\hat\btaus}
-
\bm Q_{\Lambda\hookrightarrow \Sigma \hookrightarrow (M,\cc)}^{\btaus} =(\D_\bsigma^{\!\rm T})^{d-2} f \big|_{\Lambda}\, .\nonumber
\end{align}
The above formul\ae\ encode how extrinsic (and submanifold) 
$Q$ and $T$ curvatures transform when 
moving to a conformally related metric $e^{-2f} g$. Necessarily the sum of the integral over $\widetilde\Sigma$ of the right hand side of the first equation displayed above plus $(d-2)$ times the integral over $\Lambda=\partial \widetilde\Sigma$ of the right hand side of the second equation   must vanish because ${\mathcal V}$ does not depend on~$f$. Similarly, the integral of the right hand side of the third equation over $\Lambda$ vanishes because~${\mathcal A}$ is also independent of $f$. The three operators acting on $f$ above are conformally invariant
and canonically determined by the structure $\Lambda\hookrightarrow \Sigma\hookrightarrow (M,\cc)$, and are therefore of independent interest. They are the subject of the following section.



\section{Extrinsic Conformal Laplacian Powers and Associated Boundary Operators}
\label{Sec:GJMS}

Theorems~\ref{Vanomaly}  and~\ref{Aanomaly} suggest that, given the data of a minimal defining density $\bsigma$ and unit minimal defining density $\bmu$ determined by the sequence of embeddings $\Lambda \hookrightarrow \Sigma \hookrightarrow (M,\cc)$, we define the following operators 
\begin{eqnarray*}
\bm {\mathcal P}_{\Sigma \hookrightarrow (M,\cc)}^{(k)}\:&:=&\:\:\: \D^{k}_\bsigma \: :\Gamma\big(\ce M\big[\tfrac{k-d+1}2\big]\big)\to \Gamma\big(\ce M\big[\tfrac{-k-d+1}2\big]\big)\, ,\quad k\in\{1,\ldots, d-1\}\, ,
\\[2mm]
\bm {\mathcal U}_{\Lambda\hookrightarrow \Sigma \hookrightarrow (M,\cc)}\!\!&:=
&\sum_{j=0}^{d-3} (\D^{\!\rm T}_\bsigma)^j\, 
\D^\prime_\bmu \bm \D_\bsigma^{d-j-3}:\Gamma(\ce M[0])\to \Gamma(\ce M[2-d])
\, ,
\\[2mm]
\bm {\mathcal P}^{(k)}_{\Lambda\hookrightarrow \Sigma \hookrightarrow (M,\cc)}\!\!&:=&\:\: \big(\D^{\!\rm T}_\bsigma\big)^{k} :\Gamma\big(\ce M\big[\tfrac{k-d+2}2\big]\big)\to
\Gamma\big(\ce M\big[\tfrac{-k-d+2}2\big]\big)\, ,\quad  k\in\{1,\ldots, d-2\}\, .
\end{eqnarray*}

The first important property of these operators are that they are tangential (see Definition~\ref{TangDef}
and the text that follows there).
This means that they give holographic formul\ae\
for operators along $\Sigma$ or $\Lambda$. These tangentiality properties are stated  in the next proposition.

\begin{proposition}\label{proptang}
The operator $\bm {\mathcal P}_{\Sigma \hookrightarrow (M,\cc)}^{(k)}$ is tangential along $\Sigma={\mathcal Z}(\bsigma)$, the operator $\bm {\mathcal U}_{\Lambda\hookrightarrow \Sigma \hookrightarrow (M,\cc)}$ is tangential  with respect to $\Sigma$ along $\Lambda={\mathcal Z}(\bsigma,\bmu)$, and the operator~$\bm {\mathcal P}^{(k)}_{\Lambda\hookrightarrow \Sigma \hookrightarrow (M,\cc)}$ is tangential along $\Lambda={\mathcal Z}(\bsigma,\bmu)$.
\end{proposition}

\begin{proof}
Tangentiality of the operator $\bm {\mathcal P}_{\Sigma \hookrightarrow (M,\cc)}^{(k)}$ 
is a special case of a result proved in \cite[Theorem 4.1]{GW} 
using the $\mathfrak{sl}(2)$ algebra in Equation~\nn{thesl2}
and its enveloping algebra identities~\nn{envelop}. The result there pertains to general tractor bundles and so applies to conformal densities. The quantity $I^2$ in that article is the ${ {\mathcal S}}$-curvature, which equals one to order ${\mathcal O}(\bsigma^d)$
in the present context because $\bsigma$ is a unit defining density
(see Equation~\nn{asymWill}). It is not difficult to verify (again using the  $\mathfrak{sl}(2)$ algebra) that the ${ {\mathcal S}}$-curvature
can be replaced by unity without destroying tangentiality so long as the 
power~$k$ of the Laplace--Robin operator is in $\{1,\ldots, d-1\}$.

Tangentiality of the operator $\bm {\mathcal P}^{(k)}_{\Lambda\hookrightarrow \Sigma \hookrightarrow (M,\cc)}$
along $\Lambda$ follows from a similar argument to that for~$\bm {\mathcal P}_{\Sigma \hookrightarrow (M,\cc)}^{(k)}$. The codimension two submanifold $\Lambda$ is the intersection of $\Xi$ and~$\Sigma$ so it suffices to check tangentiality to each of these hypersurfaces along $\Lambda$ separately.
Tangentiality  to~$\Xi$ is guaranteed because already the operator $\D^{\!\rm T}_\bsigma$ is tangential  to $\Xi$. 
Then  
Lemma~\ref{Iwasatheorem}
provides  the analog of the $\mathfrak{sl}(2)$ algebra used in the previous argument.

Only the proof of tangentiality of the operator $\bm {\mathcal U}_{\Lambda\hookrightarrow \Sigma \hookrightarrow (M,\cc)}$
to $\Sigma$ along $\Lambda$ remains. For that 
we need to show that 
 $$\sum_{j=0}^{d-3} (\D^{\!\rm T}_\bsigma)^j\, 
\D^\prime_\bmu \bm \D_\bsigma^{d-j-3}(\bsigma\bm f)$$
vanishes along $\Lambda$ for any smooth weight $-1$ density $\bm f$. We first use the $\mathfrak {sl}(2)$ algebra identities~\nn{envelop} to show
$$
 \bm \D_\bsigma^{d-j-3}( \bsigma \bm f) = \bsigma
 \bm \D_\bsigma^{d-j-3} \bm f+
(d-j-3)(j+2)\bm \D_\bsigma^{d-j-4}  \bm f
+{\mathcal O}(\bm \sigma^{j+4})\, .
$$
The term ${\mathcal O}(\bm \sigma^{j+4})$ appears because $\bm {\mathcal S}_\bsigmas=1+{\mathcal O}(\bsigma^d)$.
Now, along $\Xi$ when $j\neq 0$, Lemma~\ref{techlemm} gives
$$
\D^\prime_\bmu 
(\bsigma \bm \D_\bsigma^{d-j-3}\bm f)\stackrel{\Xi}= 
\frac{j+1}{j}\hh
\hh\bsigma\D^\prime_\bmu  \D_\bsigma^{d-j-3}\bm f
+{\mathcal O}(\bsigma^{d-1})\, .
$$
When $j=0$, Lemma~\ref{techlemm} says $\D^\prime_\bmu 
(\bsigma \bm \D_\bsigma^{d-3}\bm f)\stackrel\Xi={\mathcal O}(\bsigma)$, so this term in the sum does not contribute along $\Lambda$.

To compute $(\D^{\!\rm T}_\bsigma)^j(\bsigma\D^\prime_\bmu 
 \bm \D_\bsigma^{d-j-3}\bm f)$ along $\Xi$ for $j\geq 1$,  we note that along $\Xi$ the operator~$\D^{\!\rm T}_\bsigma$ equals~${\mathscr L}^\Xi_\bsigmas$ (see Equation~\nn{curlyL}), which obeys Lemma~\ref{Iwasatheorem}. Thus, using again 
the $\mathfrak {sl}(2)$ algebra identities~\nn{envelop} adjusted to account for the factor $d+2w-1$ (rather than $d+2w$) appearing in Lemma~\ref{Iwasatheorem}, along $\Lambda$ we have
$$
(\D^{\!\rm T}_\bsigma)^j(\bsigma\D^\prime_\bmu 
 \bm \D_\bsigma^{d-j-3}\bm f)
\stackrel\Lambda=
-j
(d-j-2)
(\D^{\!\rm T}_\bsigma)^{j-1}\D^\prime_\bmu 
 \bm \D_\bsigma^{d-j-3}\bm f
\, .
$$
Assembling the above three displays gives
\begin{multline*}
\sum_{j=0}^{d-3} (\D^{\!\rm T}_\bsigma)^j\, 
\D^\prime_\bmu \bm \D_\bsigma^{d-j-4}(\bsigma\bm f)
\stackrel\Lambda=
-\sum_{j=1}^{d-3}
(j+1)
(d-j-2)\hh
(\D^{\!\rm T}_\bsigma)^{j-1}\D^\prime_\bmu 
 \bm \D_\bsigma^{d-j-3}\bm f
\, 
\\ + \sum_{j=0}^{d-4} (d-j-3)(j+2)
(\D^{\!\rm T}_\bsigma)^j\, 
\D^\prime_\bmu
\bm \D_\bsigma^{d-j-4}  \bm f\, .
\end{multline*}
Reindexing one of the summations shows that the right hand side above vanishes as required.
\end{proof}

As a consequence of this proposition, we can define the following differential operators holographically:  \begin{eqnarray}
\mathbf{P}_{\Sigma \hookrightarrow (M,\cc)}^{(k)}:\Gamma\big(\ce \Sigma\big[\tfrac{k-d+1}2\big]\big)&\to& \Gamma\big(\ce \Sigma\big[\tfrac{-k-d+1}2\big]\big)\, ,\qquad k\in\{1,\ldots, d-1\}\, ,
\nonumber\\[2mm]
\stackrel{\rotatebox{90}{$\in$}}
{{\bm f}^{\!\!\!\phantom A}}\quad\qquad &\mapsto&
\quad
\stackrel{\rotatebox{90}{$\in$}}{(\D^{k}_\bsigma {\bm f}_{\rm ext})\big|_\Sigma}\, ,
\nonumber\\[7mm]
\mathbf{U}_{\Lambda\hookrightarrow \Sigma \hookrightarrow (M,\cc)}:\Gamma(\ce \Sigma[0])
&\to&
\qquad\quad
 \Gamma\big(\ce \Sigma[2-d]\big)|_\Lambda
\, ,
\label{PUPdef}
\\[2mm]
\stackrel{\rotatebox{90}{$\in$}}
{{ f}^{\!\!\!\phantom A}}\quad\quad &\mapsto&
\Big(
\sum_{j=0}^{d-3} (
\stackrel{\rotatebox{90}{$\in$}}{
\D^{\!\rm T}_\bsigma)^j\, 
\D^\prime_\bmu \bm \D_\bsigma^{d-j-3}
}\! {f}_{\rm ext}\Big)\Big|_\Lambda\, ,
\nonumber\\[7mm]
\mathbf{P}^{(k)}_{\Lambda\hookrightarrow \Sigma \hookrightarrow (M,\cc)}:\Gamma\big(\ce \Lambda\big[\tfrac{k-d+2}2\big]\big)
&\to&
\Gamma\big(\ce \Lambda\big[\tfrac{-k-d+2}2\big]\big)\, ,\quad  k\in\{1,\ldots, d-2\}\, ,
\nonumber\\[2mm]
\stackrel{\rotatebox{90}{$\in$}}
{{\bm f}^{\!\!\!\phantom A}}\quad\qquad &\mapsto&
\quad
\stackrel{\rotatebox{90}{$\in$}}{\big((
\D^{\!\rm T}_\bsigma\big)^{k} 
 {\bm f}_{\rm ext}\big)\big|_\Lambda}
\, .\nonumber
\end{eqnarray}
In the above the subscript ``ext'' denotes an arbitrary smooth extension to the bulk manifold $M$. 
 The operators $\mathbf{P}_{\Sigma \hookrightarrow (M,\cc)}^{(k)}$ are the
 extrinsically coupled, conformal Laplacian powers
  introduced in~\cite{GW15} (in that work these operators are also generalized to act not only on densities but also on general tractors, moreover the integer $k$ can be extended to include any positive even $k$). Also here we have defined the operators $\mathbf{ P}^{(k)}_{\Sigma \hookrightarrow (M,\cc)}$ with an additional factor of $(-1)^k$ as compared to~\cite{GW} in order simplify later expressions.
 For  the case when the  interior conformal class of metrics includes a formal Poincar\'e--Einstein metric (in the sense of~\cite{FGast,FGrnew}),  it was proved in~\cite{GW}
that the operator  $\mathbf{ P}_{\Sigma \hookrightarrow (M,\cc)}^{(k)}$ reduces to  the Laplacian powers of~\cite{GJMS}.
Those results were summarized in Theorem~\ref{oldiebutgoodie}.

The operator $\mathbf {P}_\Sigma:=\mathbf { P}_{\Sigma \hookrightarrow (M,\cc)}^{(d-1)}$ is of particular significance because it determines the conformal transformation property of the  ${\bm Q}_{\Sigma\hookrightarrow(M,\cc)}$ curvature. It naturally pairs with the operator~$\mathbf { U}_{\Lambda\hookrightarrow \Sigma\hookrightarrow (M,\cc)}$ because together they satisfy Theorem~\ref{intPU}, whose proof we are now ready to give.


\begin{proof}[Proof of Theorem~\ref{intPU}]
That $\mathbf{U}_{\Lambda\hookrightarrow \Sigma \hookrightarrow (M,\cc)}$ is canonically determined by the stated embedding data, follows by uniqueness of the unit and minimal unit defining densities $\bsigma$ and $\bmu$ up to higher order terms in $\bsigma$, which cannot contribute by virtue of the $\mathfrak{sl}(2)$ algebra of Equation~\nn{thesl2} and its analog along $\Xi$ in Lemma~\ref{Iwasatheorem}.

As explained at the end of Section~\ref{Sec:Anomaly}, because the volume anomaly ${\mathcal V}$ is independent of the choice of regulator $\bm \tau$, it follows from Theorem~\ref{Vanomaly}, upon replacing $\btau$ by $e^f \btau$ where $f\in C^\infty M$ and the identity $\log (e^f\btau) = f + \log \btau$, that
$$
\int_{\widetilde\Sigma}
\D_{\bsigmas}^{d-1} f + 
(d-2)
\int_{\Lambda} \sum_{j=0}^{d-3} (\D^{\!\rm T}_\bsigma)^j\, 
\D^\prime_\bmu \bm \D_\bsigma^{d-j-3}f=0\, .
$$
The integral formula of the theorem now follows directly by using the definition of the $(
\mathbf {P}_{\Sigma \hookrightarrow (M,\cc)},\mathbf {U}_{\Lambda\hookrightarrow\Sigma \hookrightarrow (M,\cc)})$-pair 
given in Equation~\nn{PUPdef}
(and text directly thereafter).

The first of the leading derivative terms 
 can be deduced from a counting argument and the divergence theorem which
 implies that
$$
\int_{\widetilde\Sigma} \Delta_\Sigma \tilde f= \int_{\Lambda}
\nabla_{\hat m} \tilde f\,  ,
$$
for any smooth function  $\tilde f$ on  ${\widetilde\Sigma}$ where $\hat m$ is the outward unit normal to $\Lambda$.
Since the smooth function $f$ is arbitrary
and $\mathbf{P}_{\Sigma \hookrightarrow (M,\cc)}$ has leading derivative term $\big((d-2)!!\big)^2\big(\Delta_\Sigma\big)^{\frac {d-1}2}$, this implies that
the operator $\mathbf {U}_{\Lambda\hookrightarrow\Sigma \hookrightarrow (M,\cc)}$
has the leading derivative term stated   modulo terms that are total derivatives. To establish that the latter are of lower transverse order, it suffices to write out  
the holographic formula for $\mathbf U_{\Lambda\hookrightarrow \Sigma \hookrightarrow (M,\cc)}$ in a flat limit. More precisely for $M={\mathbb R}^d$ with its standard conformally flat structure,  take $\Sigma$ to be a hyperplane. Moreover, since we are considering the leading derivative terms, we take $\Lambda$ to be a hyperplane in $\Sigma$. It is then easy to verify that the minimal surface $\Xi$ is the hyperplane in $M$ intersecting $\Sigma$ along $\Lambda$ at right angles. We leave it to the reader to compute $\mathbf {U}_{\Lambda\hookrightarrow\Sigma}$ in this setting and  complete the proof.

%
%
%
%
%
%
%

\end{proof}

\begin{remark}
Observe that the operator $\mathbf U_{\Lambda\hookrightarrow \Sigma \hookrightarrow (M,\cc)}$
is order $d-2$ in derivatives $\nabla_{\hat m}$ normal to~$\Lambda$. Choosing $(M,g^o)$ to be Poincar\'e--Einstein,
Theorem~\ref{intPU} solves holographically the problem of finding
a canonical, conformally invariant 
operator of this order acting on functions, determined by the embedding $\Lambda\hookrightarrow (\Sigma, \cc_\Sigma)$ when $\dim (\Sigma)$ is even.

\end{remark}

We can also now give the proof of Theorem~\ref{Plambdaop}.


\begin{proof}[Proof of Theorem~\ref{Plambdaop}]
A somewhat tedious direct computation in a choice of scale shows that along $\Xi$, the operator $\D_\bsigmas^T$ differs only by terms involving multiplication by extrinsic curvatures from the Laplace--Robin operator intrinsic to $\Xi$ depending on the defining density $\sigma|_\Xi$. 
(Alternatively, this result is an immediate consequence of the fact that the Laplace--Robin operator can be constructed from the Thomas $D$-operator of~\cite{Thomas,BEG} as well as the results for the Thomas $D$-operator along hypersurfaces given in~\cite[Section 4.2]{GW161}.) Canonical determinedness of 
$\mathbf{P}_{\Lambda\hookrightarrow \Sigma \hookrightarrow (M,\cc)}^{(k)}$ and its leading derivative behavior can then be proved following {\it mutatis mutandis}  the proof method of Theorem~\ref{oldiebutgoodie}. The displayed integral formula follows via the reasoning given below Equation~\nn{transform}. (This logic is also the same as that used to prove the analogous statement with boundary in Theorem~\nn{intPU}.)
\end{proof}

\section{Examples}\label{Sec:Examples}

We now give the proofs of Theorems~\ref{exvol} and~\ref{exarea}, which give examples of our results  when the bulk $M$ is a three- or four-manifold.
The proofs are direct computations  of the holographic formul\ae\  of Theorems~\ref{Vdiv} and~\ref{Adiv} determining the divergences in the regulated volume and area expansions, as well as Theorems~\ref{Vanomaly} and~\ref{Aanomaly} which give the extrinsic $(Q, T)$-curvature pair. 
The critical extrinsic Laplacian powers and the corresponding conformally invariant boundary operators are directly computed from Definition~\ref{PUPdef}.
In general, $Q$ and $T$ curvatures yield integrated conformal invariants so involve a mixture of terms that are either dependent or independent of the choice of $g\in \cc$ (or equivalently a true scale $\btau$). We will use a mixed bolded and unbolded notation to keep track of these dependencies. We also suppress the full extrinsic embedding data dependence appearing as subscripts on curvatures and operators.
We organize these calculations by dimensionality and start with dimension three where we encounter at most two powers of the  Laplace--Robin $\D_{_\pdot}$ operator. We will conclude by computing four-dimensional quantities that involve the action of no more than three $\D_{_\pdot}$ operators.

\subsection{Three bulk dimensions}
\label{3dV}
\subsubsection{Divergences}
In this case ${\rm dim}(\Lambda)=1$. The non-critical divergences in the regulated volume expansion~\eqref{reg_vol} are determined by Theorem~\ref{Vdiv}. These are given by 
$$
{\rm Poles}(\Vol_\varepsilon) = \frac{1}{2\varepsilon^2}\bigg(\int_{\widetilde\Sigma} {\bm v}_2 
+\int_\Lambda {\bm  v}'_{2}\hh \bigg)-\frac{1}{\varepsilon}\bigg(\int_{\widetilde\Sigma} {\bm v}_1 
+\int_\Lambda {\bm  v}'_{1}\hh \bigg)~.
$$
Here, the two integrals along $\Lambda$ vanish because ${\bm  v}'_{2}=0={\bm  v}'_{1}$. Next, computing in the scale $\btau= [g; 1]$ we find
$$
{\bm  v}_{2}= 1~, \qquad
{\bm  v}_{1}=\D_\bsigmas \btau^{-1}\Big|_{\widetilde\Sigma}=-H_{\Sigma}~.
$$
Thus
$$
{\rm Poles}(\Vol_\varepsilon) = \frac{1}{2\varepsilon^2}\int_{\widetilde\Sigma} \ext V_{g_\Sigma} +\frac{1}{\varepsilon}\int_{\widetilde\Sigma} \ext V_{g_\Sigma}  H^g_{\Sigma}~.  
$$

For the divergences in the regulated area expansion~\eqref{AREA}, Theorem~\ref{Adiv}  gives
$$
{\rm Poles}(\Area_\varepsilon)=\frac{1}{\varepsilon} \int_\Lambda \ext V_{g_\Lambda}~,
$$
where $g_\Lambda$ is the metric induced along $\Lambda$ by $g$. 

\subsubsection{$(\bm Q^g_\Sigma,\bm T^g_{\Lambda})$ pair}
The corresponding extrinsically coupled $Q$-curvature along $\Sigma$ involves two powers of the Laplace--Robin operator  and was calculated in~\cite{GWvol} and is given by
\begin{equation*}
\bm Q^g_{ \Sigma 
} = J^{g_\Sigma}-\frac12\,
\bm K_\Sigma\, ,
\end{equation*}
where $\bm K_\Sigma$ denotes the {\it rigidity density} of a conformally embedded hypersurface $\Sigma$ which is defined  (in any dimension $d\geq 3$) by
$$
\bm K_\Sigma:=
\bm\IIo_{ab}^\Sigma\,\bm \IIo_\Sigma^{ab}
\in \Gamma(\ce \Sigma[-2])\, .
$$
From Theorem~\ref{Vanomaly} and using Equations~(\ref{Lprimelog},\ref{Llog1}), the transgression  is  
\begin{equation*}
\bm T_{ \Lambda
} 
= \D^\prime_\bmu \log\bm\tau\big|_{\Lambda}
= \D_\bmu\log\bm\tau\big|_{\Lambda}
\stackrel{g}=\rho_\mu \big|_{\Lambda}
=-H_\Xi\big|_\Lambda ~.
\end{equation*}
In the above, $g $ is the metric determined by $\btau$.
The Chern--Gau\ss--Bonnet 
theorem states that the Euler characteristic~$\chi_{\widetilde\Sigma}$ 
of a surface~$\widetilde\Sigma$ is given by 
\begin{equation*}
2\pi \chi_{\widetilde\Sigma} = \int_{\widetilde\Sigma} \J^{g_\Sigma}
+\int_{\partial\widetilde\Sigma}\big(- H_{\partial \widetilde\Sigma\hookrightarrow\widetilde\Sigma}\big) \, ,
\end{equation*}
where the mean curvature $H_{\partial \widetilde\Sigma\hookrightarrow\widetilde\Sigma}$ is given by the  divergence of any extension of the inward unit conormal to $\partial \widetilde\Sigma$ (this quantity is minus the geodesic curvature of $\partial \widetilde\Sigma$).
Thus,  using that Lemmas~\ref{IIonn} and~\ref{H2H} give that $H_{\Lambda\hookrightarrow \Sigma}=H_{\Xi}|_\Lambda$ (note that the defining density~$\bmu$ is positive on the interior of $\Xi$ so that $\ext \mu$ is the inward conormal), Theorem~\ref{Vanomaly} then yields the volume anomaly
%
\begin{equation}\label{3-manifold-anom}
\mathcal V = \pi \chi_{\widetilde\Sigma}
-\frac 14\int_{\widetilde\Sigma} \bm\IIo_{ab}^\Sigma\,\bm \IIo_\Sigma^{ab}\, .
\end{equation}
This establishes the regulated volume expansion and $(\bm Q,\bm T)$ curvatures of Theorem~\ref{exvol} when $d=3$.

\subsubsection{$(\mathbf P_\Sigma,\mathbf  U_{\Lambda})$ pair}

The extrinsic conformal Laplacian $\mathbf P_\Sigma$ was computed in~\cite{GW161} and equals the Laplacian intrinsic to $\Sigma$, 
$$
\mathbf P_\Sigma=\bm \Delta_\Sigma\, .
$$
We have written the Laplace  operator in bold
because, for surfaces,  it is conformally invariant acting on weight zero densities $f_\Sigma$. When $f$ is any smooth extension of $f_\Sigma$
to~$M$, the corresponding $\mathbf U$ operator has holographic formula
$$
\mathbf U_{\Lambda}f_\Sigma = \D^\prime_\bmus f\big|_\Lambda\, .
$$
Using Equation~\nn{Lprime}, and choosing a scale $g\in \cc$, the above equals $\nabla_m f$. Hence we have
$$
\mathbf U_{\Lambda} = \bm\nabla_{\bm{ \hat m}}\, .
$$
This is manifestly conformally invariant and, because $\bm{\hat m}$ is the inward pointing unit normal, the divergence theorem  gives
$
\int_{\widetilde\Sigma}
\mathbf P_{\Sigma} f 
+  \int_{\Lambda}\mathbf U_{\Lambda}f=0$
in concordance with Theorem~\ref{intPU}. This establishes the $(\mathbf P_\Sigma,\mathbf  U_{\Lambda})$  operator pair when $d=3$ in Theorem~\ref{exvol}.

\subsubsection{$\bm Q_{\Lambda}$ and $\mathbf P_{\Lambda}$}
Our holographic formula for the ${ Q}$-curvature of $\Lambda$ in three bulk dimensions is given by (see Theorem~\ref{Aanomaly})
$$
\bm Q_{\Lambda}=\D_\bsigma^{\!\rm T}\log \btau\big|_\Lambda=-\widehat\D_\bmu \langle \bsigma,\bmu \rangle\big|_\Lambda=-\hh 
\bm\IIo^\Sigma_{\bm{ \hat\bm  m}   {\bm{\hat  m}}}|_\Lambda\, .
$$
In the above we used Definition~\nn{logLT}
and Lemma~\nn{C}.
Hence the area anomaly is given by
\begin{equation*}
\mathcal A =\int_\Lambda \bm\IIo^\Sigma_{\bm{ \hat\bm  m}   {\bm{\hat  m}}}\, .
\end{equation*}
Observe that this vanishes when the embedding of $\Sigma$ is umbilic, and thus also when the singular metric $g^o$ is (asymptotically) Poincar\'e--Einstein. Moreover, the above integrand is manifestly conformally invariant, and hence it is possible that the corresponding extrinsic Laplacian power  may vanish; indeed this is the case:
$$
\mathbf P_{\Lambda}=0\, .
$$
To see this, it is easy to verify using Definition~\ref{LTdef}
that $\D_\bsigma^{\!\rm T}f|_\Lambda=0$ for $f\in C^\infty M$, and hence that Theorem~\ref{Plambdaop} gives a vanishing result for $\mathbf P_{\Lambda}$.

%

\newpage

\subsection{Four bulk dimensions}
\subsubsection{Divergences}
Theorem~\ref{Vdiv} indicates that the divergences in the regulated volume expansion are
$$
{\rm Poles}(\Vol_\varepsilon) = 
\frac{1}{3\varepsilon^3}\int_{\widetilde\Sigma} {\bm v}_3
-\frac{1}{4\varepsilon^2}\int_{\widetilde\Sigma} {\bm v}_2
+\frac{1}{4\varepsilon}\bigg(\int_{\widetilde\Sigma} {\bm v}_1 
+2\int_\Lambda {\bm  v}'_{1}\hh \bigg)~, 
$$
where ${\bm  v}'_{3}=0={\bm  v}'_{2}$. In the $\bm \tau$ scale, the expansion coefficients are
$$
{\bm v}_3=1~, \qquad 
{\bm v}_2=\D_\bsigmas \btau^{-2}\Big|_{\widetilde\Sigma}
=-4H^g_{\Sigma}~, \qquad 
{\bm v}_1=\D_\bsigmas^{2}\btau^{-1}\Big|_{\widetilde\Sigma}
=-2J^g_{\Sigma}~.
$$
(These have been calculated in~\cite{GWvol}.) The boundary contribution is computed as follows:
$$
{\bm  v}'_{1}= \D_\bmu^\prime \btau^{-1}\big|_{\Lambda}
=\widetilde\D_\bmu\, \bm\tau^{-1}\big|_{\Lambda}
=-\rho_{\mu}\big|_{\Lambda}
=H_{\Lambda\hookrightarrow \Sigma}^{g_\Sigma}\, ,
$$
where we have made use of Definitions~\ref{Lprime} and~\ref{Lodzhat} as well as the identity $H_{\Lambda\hookrightarrow \Sigma}=H_{\Xi}|_\Lambda$ (which in turn follows from Lemmas~\ref{IIonn} and~\ref{H2H}) to compute the action of the critical operator $\D^\prime_\bmus$ on the bulk Yamabe weight $-1$ density~$\btau^{-1}$. Altogether this gives 
$$
{\rm Poles}(\Vol_\varepsilon) = \frac{1}{3\varepsilon^3} \int_{\widetilde\Sigma} \ext V^{g_\Sigma}
+
\frac{1}{\varepsilon^2}
\int_{\widetilde\Sigma} \ext V^{g_\Sigma}
H_{\Sigma}^g
-\frac{1}{2\varepsilon}
\int_{\widetilde\Sigma} \ext V^{g_\Sigma}
J_{\Sigma}^g
+\frac{1}{2\varepsilon}\int_\Lambda  
dV^{g_\Lambda}H_{\Lambda\hookrightarrow \Sigma}^{g_\Sigma}~.
$$

\medskip

Turning to the divergences in the $d=4$ area expansion, Theorem~\ref{Adiv} says that
$$
{\rm Poles}(\Area_\varepsilon)= \frac{1}{2\varepsilon^2}\int_\Lambda\bm a_2 -\frac{1}{\varepsilon}\int_\Lambda\bm a_1\, ,
$$
where the expansion coefficients
$$
\bm a_2=1~, \qquad
\bm a_1= \D_\bsigma^{\!\rm T} \, \btau^{-1}\big|_{\Lambda}
=\big(\Lodz_\bmu \langle\bsigma,\bmu\rangle -\widetilde \D_\bsigmas \big)\btau^{-1}\big|_{\Lambda}
=\bm\IIo^\Sigma_{\bm{ \hat\bm  m}   {\bm{\hat  m}}} - H_{\Sigma}~.
$$
In the computation of $\bm a_1$ we  used the definition of the critical tangential operator $\D^{\rm T}_\bsigmas$ given in Equation~\ref{LTcrit1} as well as Lemma~\ref{C}. Thus we have
$$
{\rm Poles}(\Area_\varepsilon)=
\frac1{2\varepsilon^2}
\int_\Lambda \ext V_{g_\Lambda}
+\frac{1}{\varepsilon}\int_\Lambda \ext V_{g_\Lambda}
\big(H_{\Sigma}-\bm\IIo^\Sigma_{\bm{ \hat\bm  m}   {\bm{\hat  m}}} \big)~.
$$

\subsubsection{$(\bm Q_\Sigma,\bm T_{\Lambda})$ pair}
Here $\Lambda$ is a surface without boundary.
In four dimensions, the extrinsic $Q$-curvature involves three powers of the Laplace--Robin operator and was calculated in~\cite{GGHW15,GWvol}. It is given for some choice of $g\in \cc$ by 
\begin{equation*}
\bm Q_{\Sigma} = -4\nabla^a_\Sigma\nabla^b_\Sigma \,\bm{\IIo}_{ab}^\Sigma
-8\, \bm{\IIo}^{ab}_\Sigma\, \bm{\mathcal F}_{ab}^\Sigma ~,
\end{equation*}
where $\bm{\mathcal F}_{ab}^\Sigma\in \Gamma(\odot^2 T^*M[0])$ is the {\it Fialkow tensor}~\cite{Fialkow} (see also~\cite{YuriThesis,Grant,Stafford,GW15})
 and equals (in any dimension $d\geq 4$)
\begin{align}\label{Fialkow}
\bm {\mathcal F}_{ab}^\Sigma
&:=\frac1{d-3}\Big(\bm \IIo_{\!ac}^\Sigma\, \bm g_\Sigma^{cd} \hh \hh\bm \IIo_{bd}^\Sigma
-\frac1{2(d-2)}\hh \bm{  g}_{ab}^\Sigma\,  \bm{K}_\Sigma
-\bm W_{\!cabd}\,\bm{\hat n}^c\bm{\hat n}^d\Big)~.\nonumber
\end{align}

The transgression again follows from Theorem~\ref{Vanomaly} and  Equations~(\ref{Lprime} ,\ref{Lprimelog}):
\begin{equation*} 
\bm T_{\Lambda\hookrightarrow \Sigma} = \big(\D^\prime_\bmu \D_\bsigma 
+ \D_\bsigma^{\!\rm T} \D^\prime_\bmu\big) \log\bm\tau \Big|_{\Lambda}
=\big(\widetilde\D_\bmu \D_\bsigma 
+ 2\D_\bsigma^{\!\rm T} \widehat\D_\bmu\big) \log\bm\tau \Big|_{\Lambda}\, .
\end{equation*}
Using~\nn{Llog1}, the first term on the right is
\begin{multline*}
\widetilde\D_\bmu \D_\bsigma \log\bm\tau \big|_{\Lambda}
\stackrel g= \big(\nabla_m - \rho_\mu\big) \big(2\rho_\sigma- \sigma J  \big)\big|_\Lambda
\nonumber\\[1mm]
=\big(2\nabla_m \rho_\sigma - 2\rho_\mu\rho_\sigma\big) \big|_\Lambda
=-\hat m^a \nabla_\Sigma^b\hh  \IIo_{ab}^\Sigma+2 \Rho_{\hat m \hat n} - 2 H_\Sigma H_\Xi\, .
\end{multline*}
In the above $g$ is the metric determined by $\btau$. 
We also used $\nabla_m\sigma|_\Lambda=m.n|_\Lambda=0$, and
that $\nabla_m \rho_\sigma \stackrel\Lambda= - \nabla^\Sigma_{\hat m } H_\Sigma$ 
as well as
 the trace of the $d=4$ Codazzi--Mainardi equation (see~\cite[Equation (2.9)]{GW15}) along $\Lambda$: 
$$
\hat m^a
\nabla^b_\Sigma \hh\IIo^\Sigma_{ab} =
2\big(\nabla^\Sigma_{\hat m} H_\Sigma + \Rho_{\hat m \hat n}\big)\, .
$$
The remaining term in $\bm T_{\Lambda}$ can be similarly  handled using Equation~\nn{LTcrit1}:
\begin{align*}
2\D_\bsigma^{\!\rm T} \widehat\D_\bmu \log\bm\tau \Big|_{\Lambda}
&=2\Big(\big(\Lodz_\bmu \langle\bsigma,\bmu\rangle\big)
-\widetilde \D_\bsigmas\Big)
\widehat\D_\bmu \log\bm\tau \Big|_{\Lambda}\\
&\stackrel g= 2\Big(
\IIo^\Sigma_{\hat m\hat m}
- \nabla_{\hat n}+\rho_\sigma \Big)
\big(\rho_\mu -\frac 12\mu J\big) \Big|_{\Lambda} \\
&=-2\hh\hh\IIo^\Sigma_{\hat m\hat m} H_\Xi 
+\hat n^a \nabla_\Xi^b \IIo_{ab}^\Xi-2 \Rho_{\hat m \hat n}
+ 2H_\Sigma H_\Xi \, .
\end{align*}
Orchestration gives
$$
\bm T_{\Lambda}\stackrel g=-{\hat m}^a \nabla_\Sigma^b \hh \IIo_{ab}^\Sigma
+{\hat n}^a \nabla_\Xi^b \hh\hh\IIo_{ab}^\Xi-2\hh \hh \IIo^\Sigma_{{\hat m\hat m}} H_\Xi\, .
$$
We want to write this expression in terms of the embedding sequence $\Lambda\hookrightarrow \Sigma \hookrightarrow (M,\cc)$.
Recalling that Lemmas~\ref{IIonn} and~\ref{H2H} imply that $H_{\Lambda\hookrightarrow \Sigma}=H_{\Xi}|_\Lambda$, so we only need deal with the second term in the above
(note that along $\Lambda$, the unit conormal $ {\hat m}$ is determined by the embedding $\Lambda\hookrightarrow \Sigma$). For this  we 
first note that the trace of the $d=4$ Codazzi--Mainardi equation (see for example~\cite[Equation (2.9)]{GW15}) along $\Lambda$
gives 
$$
\hat n^a
\nabla^b_\Xi \hh\IIo^\Xi_{ab} =
2\big(\nabla^\Xi_{\hat n} H_\Xi + \Rho_{\hat m \hat n}\big)\, .
$$
Therefore, we need to study $\nabla^\Xi_{\hat n} H_\Xi$ which we can write along $\Lambda$ in terms of our canonical extensions $n$ and $m$ of $\hat n$ and $\hat m$, as (using that $\rho_\mu=-\frac 14  (\nabla.m+\mu \J)$ and $m.n|_\Lambda=0$)
\begin{multline*}
\frac 14 \nabla_n (\nabla.m+\mu \J)
\stackrel\Lambda=
\frac14\Big( \nabla^b \nabla_n m_b
-Ric_{mn}-(\nabla_b n^a)\nabla_a m^b \Big)\\
\stackrel\Lambda=
\frac14
\Big(
\Delta(m.n)
-m^a \Delta n_a
-2\Rho_{\hat m\hat n}
-2(\IIo_{ab}^\Sigma+g_{ab}H_\Sigma )
(\IIo_\Xi^{ab}+g^{ab} H_\Xi)
\Big)\, .
\end{multline*}
Hence
$$
\nabla^\Xi_{\hat n}H_\Xi +\Rho_{\hat n \hat m}\stackrel\Lambda=
\frac12 \Rho_{\hat m \hat n}
-\frac12 \, \IIo^\Sigma_{ab}\hh\IIo_\Xi^{ab}-2H_\Sigma H_{\Lambda\hookrightarrow \Sigma}+
\frac14 \Delta (m.n) -\frac14 m^a \Delta n_a\, .
$$
The last term above is easily computed using similar techniques:
$$
m^a \Delta n_a = m^a \nabla_b \nabla_a n^b = \nabla_m \nabla.n + \Ric_{mn}
\stackrel\Lambda=4\nabla^\Sigma_{\hat m} H_\Sigma+2\Rho_{\hat m \hat n}\, .
$$
Moreover,
\begin{multline*}
\Delta(m.n)
=\Delta\big(-\mu \rho_\sigma -\sigma \rho_\mu + \mu\hh {\mathcal C} + {\mathcal O}(\sigma^3)\big)
\\\stackrel\Lambda=
8 H_\Sigma H_\Xi
+2\nabla_{\hat m}^\Sigma H_\Sigma 
+2\nabla^\Xi_{\hat n} H_\Xi
+4 H_\Xi {\mathcal C} + 2 \nabla_{\hat m}^\Sigma{\mathcal C}\, .
\end{multline*}
Here we used $\Delta \mu\stackrel\Xi = 4 H_\Xi$  and 
$\Delta \sigma\stackrel\Sigma= 4 H_\Sigma$. It follows, using the traced Codazzi--Mainardi equation along $\Sigma$,  that 
\begin{align*}
\nabla^\Xi_{\hat n}H_\Xi
+ \Rho_{\hat m \hat n}
&\stackrel\Lambda=
-\nabla_{\hat m}^\Sigma H_\Sigma 
- \Rho_{\hat m\hat n}
+(\nabla^\Sigma_{\hat m} + 2 H_{\Lambda\hookrightarrow \Sigma})\hh \hh\IIo_{\hat m \hat m}^\Sigma
- \IIo^\Sigma_{ab}\hh\IIo_\Xi^{ab}
\\
&
\stackrel\Lambda=
-\frac12 \hat m^a
\nabla^b_\Sigma \hh\IIo^\Sigma_{ab} 
+(\nabla^\Sigma_{\hat m} + 2 H_{\Lambda\hookrightarrow \Sigma})\hh \hh\IIo_{\hat m \hat m}^\Sigma
- \IIo^\Sigma_{ab}\hh\IIo_\Xi^{ab}
\,.
\end{align*}
We still need to develop the last term above. For that we compute using the same methodology as above as follows:
\begin{multline*}
\IIo^{ab}_\Sigma \hh\IIo_{ab}^\Xi
\stackrel\Lambda=
\IIo^{ab}_\Sigma
(\nabla_a m_b + g_{ab}\rho_\mu)
\stackrel\Lambda=
\IIo^{ab}_\Sigma \nabla_a m_b
\stackrel\Lambda=
\IIo^{ab}_\Sigma \big([\nabla_a-m_a\nabla_m] m_b
+\frac 12 m_a \nabla_b m^2\big)
\\[1mm]
\stackrel\Lambda=
\IIo^{ab}_\Sigma \big([\nabla^\Sigma_a-m_a\nabla^\Sigma_m] m_b
+\frac 12 m_a \nabla_b m^2\big)
\stackrel\Lambda=
\IIo^{ab}_\Sigma 
\big(\II_{ab}^{\Lambda\hookrightarrow \Sigma}+\hat m_a\hat m_b H_\Xi\big)
=
\IIo^{ab}_\Sigma \hh \IIo_{ab}^{\Lambda\hookrightarrow \Sigma}\, .
\end{multline*}
Thus, noting that the definition of the Robin operator in Equation~\nn{ROBIN} gives
$\nabla^\Sigma_{\hat m} \hh
\IIo_{\hat m \hat m}^\Sigma
+  H_{\Lambda\hookrightarrow \Sigma}\hh \hh\IIo_{\hat m \hat m}^\Sigma\stackrel\Lambda =\bm \updelta_{\rm R}^{\sss\Lambda\hookrightarrow \Sigma} \, \bm \IIo^\Sigma_{\bm{\hat m}\bm {\hat m}}
$,
we finally obtain
$$
\bm T_{\Lambda}
\stackrel g=
-2 \bm{ \hat m}^a
\nabla^b_\Sigma \hh\bm \IIo^\Sigma_{ab}
+2
\bm \updelta_{\rm R}^{\sss\Lambda\hookrightarrow \Sigma} 
\, \bm \IIo^\Sigma_{\bm{\hat m}\bm {\hat m}}\
-2\hh \hh\bm \IIo^\Sigma_{ab}\hh\hh\bm \IIo_{\Lambda\hookrightarrow \Sigma}^{ab}
\, .
$$
Remembering that $\hat m$ is the inward normal to $\Lambda$, then the divergence theorem together with 
Theorem~\ref{Vanomaly} yields the volume anomaly
\begin{equation}\label{4-manifold-anom}
\mathcal V = 
\frac23\, \int_{\widetilde\Sigma}\bm{\IIo}^{ab}_\Sigma\, \bm{\mathcal F}_{ab}^\Sigma
-\frac13
\int_\Lambda
\Big(
\bm \updelta_{\rm R}^{\sss\Lambda\hookrightarrow \Sigma} 
\, \bm \IIo^\Sigma_{\bm{\hat m}\bm {\hat m}}\
-\hh \hh\bm \IIo^\Sigma_{ab}\hh\hh\bm \IIo_{\Lambda\hookrightarrow \Sigma}^{ab}\Big)\, .
\end{equation}
Here the integrands of the above result are manifestly conformally invariant because the scale dependent terms in the ${T}$-curvature conspire to precisely cancel those in the extrinsic ${Q}$-curvature. 
Notice that this anomaly vanishes when $\Sigma$ is umbilic and in particular  when the bulk is Poincar\'e--Einstein.

%
%
%
%
%
%
%
%

%
%

\subsubsection{$(\mathbf P_\Sigma,\mathbf  U_{\Lambda})$ pair}

The extrinsic conformal Laplacian power was calculated in \cite{GW161} (see also~\cite{GGHW15}) and, acting on weight zero densities,  is given by 
(not forgetting a factor $(-1)^{d-1}=-1$ accounting for differing sign  conventions)
$$
\mathbf P_\Sigma=8\bm \nabla_a^\Sigma \circ\bm \IIo^{ab}_\Sigma \circ \bm \nabla^\Sigma_b\, .
$$
Observe that although $\Sigma$ has dimension three so that there is no intrinsic conformally invariant Laplacian power at this weight~\cite{GJMS}, the above operator is Laplacian-like with the trace-free second fundamental form appearing in place of the inverse metric. We have used a bold notation for both gradient operators because the gradient of a function and the divergence of a weight $-3$ vector are conformally invariant operations in three dimensions. 
It follows from the above display and the divergence theorem that
if $f\in C^\infty M$, then
$$
\int_{\widetilde\Sigma}
\mathbf P_{\Sigma} f = -8\int_\Lambda
\bm {\hat m}_a\hh  \bm \IIo^{ab}_\Sigma \bm \nabla_b^\Sigma f\, ,
$$
where $\bm {\hat m}$ is the inward unit normal to $\Lambda$.

The $\mathbf U_\Lambda$ operator associated to $\mathbf P_{\Sigma}$ can be computed from Equation~\nn{PUPdef}. 
First note that when $d=4$ we have
$$
\sum_{j=0}^{d-3} (
\D^{\!\rm T}_\bsigma)^j\, 
\D^\prime_\bmu \bm \D_\bsigma^{d-j-3}
=\D^{\!\rm T}_\bsigma\D^\prime_\bmu
+\D^\prime_\bmu  \D_\bsigma
\stackrel\Lambda
=
2(\bm{\mathcal C}
-\widetilde \D_\bsigmas)
\widehat\D_\bmu
+
\widetilde\D_\bmu \D_\bsigma\, . 
$$
To achieve the last equality we used the definitions of  $\D_\bmus^\prime$ and $\D^{\!\rm T}_\bsigma$  in Equations~\nn{Lprime} and~\nn{LTcrit1} of $\D^{\!\rm T}_\bsigma$, as well as Equation~\nn{thebr} in conjunction with the identity $\widehat\D_\bmu \bmu \stackrel\Xi = 1$.
We now choose a scale $g\in\cc$ and write out the operators appearing on the right hand side above. Working along $\Lambda$ this gives
$$
2({\mathcal C}\!-\!\nabla_n+\rho_\sigma)\nabla_m 
+(\nabla_m-\rho_\mu)(2\nabla_n-\sigma \Delta)
\stackrel\Lambda=
2[\nabla_m,\nabla_n]
-2H_\Sigma \nabla_m \!+ \! 2 H_\Xi \nabla_n
+2{\mathcal C}\nabla_m\, .
$$
Recalling that $\nabla_m n_a \stackrel\Lambda=m^b \IIo^\Sigma_{ba}+m_a H_\Sigma$ and
$\nabla_n m_a \stackrel\Lambda=n^b \IIo^\Xi_{ba}+n_a H_\Xi$, the commutator term yields
$
[\nabla_m,\nabla_n]
\stackrel\Lambda=
m^b \IIo^\Sigma_{ba}\nabla^a_\Sigma
+H_\Sigma \nabla_m
-n^b \IIo^\Xi_{ba}\nabla^a_\Xi- H_\Xi\nabla_n
$, where along $\Lambda$ and acting on scalars we have used that $\nabla^a_\Xi
=\nabla^a-m^a \nabla_m$ and $\nabla^a_\Sigma
=\nabla^a-n^a \nabla_n$.
In concordance with the tangentiality result of Proposition~\ref{proptang}, applying this commutator result to the above display, all instances of the gradient operator $\nabla_n$ along the conormal to $\Sigma$ cancel, and we are left with the operator
$$
2m^b \IIo^\Sigma_{ba}\nabla^a_\Sigma
-2n^b \IIo^\Xi_{ba}\nabla^a_\Xi+2{\mathcal C}\nabla_m\, .
$$ 
Using the identity of Equation~\ref{Rodsequation} and Lemma~\ref{C}, we have the conformally invariant result for the operator $ \mathbf U_{\Lambda}$
$$
\mathbf U_{\Lambda} = 4\bm {\hat m}^a\bm \IIo^\Sigma_{ab} \bm \nabla_\Sigma^b\, .
$$
Indeed, as proved in Theorem~\ref{intPU}
$$
\int_{\widetilde\Sigma}
\mathbf P_{\Sigma} f 
+ 2 \int_{\Lambda}\mathbf U_{\Lambda}f=0\, .$$
We observe that just as for the ${\mathbf P}_\Sigma$ operator, the ${\mathbf U}_\Lambda$ operator is obtained from its $d=3$ counterpart $\bm\nabla_{\bm{ \hat m}}$
by replacing the inverse metric with the trace-free second fundamental form.

\subsubsection{$\bm Q_{\Lambda}$ and 
$\mathbf P_{\Lambda}$}

The holographic formula for the ${\bm Q}$ curvature of $\Lambda$ in four bulk dimensions is given by (see Theorem~\ref{Aanomaly})
\begin{equation*}
\bm Q_{\Lambda}=
(\D_\bsigma^{\!\rm T})^2\log\bm \t \big|_\Lambda\, .
\end{equation*}
In four dimensions  $\D_\bsigma^{\!\rm T} \log\bm\tau$ 
has bulk Yamabe  weight $w=-1$, so we must use Equation~\nn{LTcrit1} of Definition~\ref{LTdef} as well as Definition~\ref{logLT}.  This yields
\begin{eqnarray*}
\D_\bsigma^{\!\rm T}{}^2\log\bm \t 
&\stackrel\Lambda=& \Big((\widehat\D_\bmus \langle\bsigma,\bmu\rangle)
-\widetilde\D_\bsigmas \Big)
\Big(\widehat\D_\bsigmas \log\bm \t  - \widehat\D_\bmus \langle\bsigma,\bmu\rangle
+\bsigma\widetilde \D_\bmus \widehat\D_\bmu \log\bm \t  \Big)\\[1mm]
&\stackrel\Lambda=&
-\widetilde \D_\bmu \widehat\D_\bmu \log\bm \t 
-\widetilde\D_\bsigmas 
\widehat\D_\bsigma \log\bm \t
+\widetilde\D_\bsigmas
\widehat\D_\bmu \langle\bsigma,\bmu\rangle
+
\bm{\mathcal C}\, (\widehat\D_\bsigma \log\bm \t  - \bm{\mathcal C})
\, .
\end{eqnarray*}
%
Here $\bm{\mathcal C}$ is as given in Lemma~\ref{C} and we have also used $\widetilde\D_\bsigmas\bsigma|_\Sigma = 1$.
We now calculate each of the  four terms above for the choice of $g\in \cc$ determined by $\btau$. First, 
using Equations~\nn{LODZ} and~\nn{Llog1} and $\nabla_m \mu|_\Xi=1$,
\begin{equation*}
-\widetilde \D_\bmu \widehat\D_\bmu \log\bm \t \big|_\Lambda
= -\big(\nabla_m - \rho_\mu \big)\big( \rho_\mu- \tfrac 12 \mu J\big)\stackrel\Lambda
=-\nabla_m \rho_\mu +\rho_\mu^2 +\frac 12 J ~.
\end{equation*}
Similarly
\begin{equation*}
-\widetilde \D_\bsigmas \widehat\D_\bsigmas \log\bm \t \stackrel\Lambda
=-\nabla_n \rho_\sigma +\rho_\sigma^2 +\frac 12 J ~.
\end{equation*}
For the second last of the abovementioned four terms we  use that $\widehat\D_\bsigma \log\bm \t\stackrel\Lambda=
\rho_\sigma$. That leaves 
the term%
\begin{align*}
\widetilde\D_\bsigmas \widehat\D_\bmu \langle\bsigma,\bmu\rangle 
&\stackrel\Lambda= \nabla_n \nabla_m \big(m.n +\sigma\rho_\mu+ \mu\rho_\sigma\big) -\rho_\sigma {\mathcal C}\\
&\stackrel\Lambda= \nabla_n\nabla_m m.n  + \rho_\mu\nabla_n m.n 
+ \nabla_m\rho_\mu +\nabla_n \rho_\sigma
 -\rho_\sigma {\mathcal C}\, .\end{align*}
Here we used $\nabla_n \mu = m.n\stackrel\Lambda=0$.
Using that 
 $\nabla_m m = \frac 12 \nabla m^2 =-\nabla (\mu\rho_\mu)+{\mathcal O}(\mu^3)$
we have
$$
\nabla_n\nabla_m m.n
\stackrel\Lambda=
\nabla_n\big(-\nabla_n(\mu\rho_\mu)
+m^a \nabla_m n_a\big)
\stackrel\Lambda=
-\rho_\mu\nabla_n m.n
+\nabla_n(m^a \nabla_m n_a)\, .
$$
Collating the above computations we find
$$
(\D_\bsigma^{\!\rm T})^2\log\bm \t \stackrel\Lambda 
=J + H_\Sigma^2 + H_\Xi^2- \mathcal C^2 + \nabla_n(m^a\nabla_mn_a)\, .
$$
We now focus on the last term on the line above:
\begin{align*}
\nabla_n(m^a \nabla_m n_a)
&\stackrel\Lambda=
\IIo^\Xi_{\hat n a}
g^{ab}\IIo^\Sigma_{\hat m b}
+m^a [\nabla_n,\nabla_m]n_a
-m^a\nabla_m \nabla_a(\sigma\rho_\sigma)\\[1mm]
&\stackrel\Lambda=
\IIo^\Xi_{\hat n a}
g^{ab}\IIo^\Sigma_{\hat m b}
+m^a(\nabla_n m^c) \nabla_c n_a
-m^a (\nabla_m n^b) \nabla_b n_a\\
&\qquad\qquad+R_{\hat n\hat m\hat m\hat n}
+({\mathcal C} + H_\Sigma)H_\Sigma
\\[1mm]
&\stackrel\Lambda=
2\hh\IIo^\Xi_{\hat na}g^{ab}\IIo^\Sigma_{\hat mb}
-
\IIo^\Sigma_{\hat ma}g^{ab}\IIo^\Sigma_{\hat mb}
-H_\Sigma\hh{\mathcal C}
-W_{\hat n\hat m\hat n\hat m}
-\Rho_{\hat m \hat m}-\Rho_{\hat n\hat n}\, .
\end{align*}
In the above we used $\nabla n \stackrel\Sigma = \IIo^\Sigma + g H_\Sigma$
as well as the analogous relation and~$\nabla m$. Also $\nabla_n n= -\nabla (\sigma\rho_\sigma)+{\mathcal O}(\sigma^3)$.
This implies that $m^a\nabla_m n_a\stackrel\Sigma=
\IIo_{\hat m\hat m}^\Sigma+H_\Sigma\stackrel\Lambda={\mathcal C} + H_\Sigma$, which was used to reach the second equality.
Moreover $n|_\Sigma = \hat n$, $m|_\Xi = \hat m$, $\rho_\sigma|_\Sigma=-H_\Sigma$ and $\rho_\mu|_\Xi=-H_\Xi$.
To achieve the last line we additionally employed Equation~\nn{superclaim}.
Altogether we now get
\begin{multline*}
(\D_\bsigma^{\!\rm T})^2\log\bm \t \stackrel\Lambda 
=J-\Rho_{\hat m \hat m}-\Rho_{\hat n\hat n} + H_\Sigma^2 + H_\Xi^2 
\\[1mm]+2\hh\hh\IIo^\Xi_{\hat na}g^{ab}\IIo^\Sigma_{\hat mb}
-
\IIo^\Sigma_{\hat ma}g^{ab}\IIo^\Sigma_{\hat mb}
-W_{\hat n\hat m\hat n\hat m}
-{\mathcal C}({\mathcal C}+H_\Sigma)\, .
\end{multline*}
Employing Equation~\nn{Rodsequation} and that
 $H_{\Lambda\hookrightarrow \Sigma}=H_{\Xi}|_\Lambda
$
 (see Lemma~\ref{H2H}), we then have
$$
{Q}_{\Lambda}=
H_{\Lambda\hookrightarrow \Sigma}^2
+\Rho_{ab}\hh g^{ab}_\Lambda
+H_\Sigma^2
 - H_{\Sigma}{\mathcal C}
 -2{\mathcal C}^2
 -3\hh \IIo^\Sigma_{\hat ma}g^{ab}_\Lambda\IIo^\Sigma_{\hat mb}
 -W_{\hat n\hat m\hat n\hat m}\, .$$
Along $\Sigma$,  the Fialkow--Gau\ss\ Equation~\cite[Equation 2.7]{GW15} gives
 $$
 \Rho_{ab}-\hat n_a \Rho_{\hat n b} 
 -\hat n_b \Rho_{\hat a}+\hat n_a\hat n_b \Rho_{\hat n\hat n}
 =\Rho^\Sigma_{ab}-H_\Sigma \IIo_{ab}^\Sigma-\frac12 (g_{ab}-\hat n_a \hat n _b) H_\Sigma^2 +{\mathcal F}_{ab}^\Sigma\, ,
 $$
 where the 
 Fialkow tensor in $d=4$ is given by
 $
 {\bm {\mathcal F}}^\Sigma_{ab}=
 \bm \IIo^\Sigma_{ac} \hh \bm g_\Sigma^{cd}\hh\bm \IIo^\Sigma_{bd}-\frac14 \bm g^\Sigma_{ab}\hh\hh
 \bm K_\Sigma
+\bm W_{\bm{\hat n}a\bm {\hat n}b}$.
 Note that $\bm g_\Sigma^{ab} {\bm{\mathcal  F}}_{ab}^\Sigma=\frac14 \bm K_\Sigma$.
 Thus, along $\Lambda$,
 \begin{align*}
 P_{ab}g^{ab}_\Lambda&=J_\Sigma-P^\Sigma_{\hat m \hat m}
 +H_\Sigma{\mathcal C}
 - H_\Sigma^2
+\frac14 K_\Sigma
- {\mathcal F}^\Sigma_{ {\hat m}{\hat m}}
\, .
 \end{align*}
 Using
 $
 W_{\hat n\hat m\hat n\hat m}=
 {\mathcal F}^\Sigma_{ {\hat m}{\hat m}}
-\IIo^\Sigma_{\hat ma}g_\Lambda^{ab}\IIo^\Sigma_{\hat mb}-{\mathcal C}^2+\frac14 K_\Sigma
 $,
 this gives
 \begin{align*}
  Q_{\Lambda}
 &=
 H_{\Lambda\hookrightarrow \Sigma}^2
+J_\Sigma-P^\Sigma_{\hat m \hat m}
- 2{\mathcal F}^\Sigma_{ {\hat m}{\hat m}}
 -2\hh \IIo^\Sigma_{\hat ma}g_\Lambda^{ab}\IIo^\Sigma_{\hat mb}
 -{\mathcal C}^2\, .
 \end{align*}
A necessary condition for the singular metric $g^o$ to be Poincar\'e--Einstein is that the embedding $\Sigma\hookrightarrow (M,\cc)$ is umbilic so that $\, \bm \IIo_{ab}^\Sigma=0$.
Moreover the Fialkow tensor vanishes, this is easily verified by demonstrating that $\bm W_{\bm {\hat n} a\bm {\hat n}b}\stackrel\Sigma=0$ for Poincar\'e--Einstein structures. In that case only the first three terms on the right hand side of the above display survive and the area anomaly (see Theorem~\ref{Aanomaly}) is given by ${\mathcal A}=\frac12\int_\Lambda\big(H_{\Lambda\hookrightarrow \Sigma}^2
+J_\Sigma-P^\Sigma_{\hat m \hat m}\big)$. 
This is in concordance with the original expression for the log coefficient of Graham and Witten~\cite{GrahamWitten} (noting that their mean curvature is the sum, not average of the eigenvalues of the second fundamental form).

To see that $\bm Q_{\Lambda}^g$ is an extrinsically coupled $Q$-curvature type invariant for the submanifold $\Lambda$, we recall that the Gau\ss\ equations imply
$$
J_\Sigma-\Rho^\Sigma_{\hat m \hat m} = J_\Lambda -H_{\Lambda\hookrightarrow \Sigma}^2
+ \frac12 K_{\Lambda\hookrightarrow \Sigma}\, ,
$$
where $J_\Lambda :=\frac 12 Sc_\Lambda$. 
Thus we have
$$
{\bm Q}^{g}_{\Lambda}=J_\Lambda
+ \frac12 \bm K_{\Lambda\hookrightarrow \Sigma}
- 2\bm {\mathcal F}^\Sigma_{ \bm {\hat m}\bm {\hat m}}
 -2\hh \bm\IIo^\Sigma_{\bm{\hat m}a}\bm g_\Lambda^{ab}\bm \IIo^\Sigma_{\bm{\hat m}b}
 -\bm {\mathcal C}^2\, .
$$

Hence, using the Gau\ss--Bonnet theorem,  the area anomaly is given by 
$$
{\mathcal A}=\pi \chi_\Lambda + 
 \frac14 
 \int_\Lambda 
 \Big[
 \bm K_{\Lambda\hookrightarrow \Sigma}
 - 4\bm {\mathcal F}^\Sigma_{ {\hat m}{\hat m}}
 -4\hh \bm \IIo^\Sigma_{\hat ma}\bm g^{ab}_\Lambda \bm \IIo^\Sigma_{\hat mb}
  -2\bm{\mathcal C}^2
\Big]\, .
$$
The first term above is proportional to the integral over the intrinsic $Q$-curvature of the manifold $\Lambda$
and gives the three manifold anomaly in the regulated volume when the bulk is Poincar\'e--Einstein. 

\medskip

Finally we turn to the operator $\mathbf P_\Lambda$.
From Theorem~\ref{Plambdaop} and Equations~(\ref{LTcrit1},\ref{LTcrit2}) of Definition~\ref{LTdef}, the extrinsic Laplacian power associated to the submanifold $Q$-curvature~$\bm Q_{\Lambda}$ 
has holographic formula
$$
\mathbf P_{\Lambda}\stackrel\Lambda=
\big(\D^{\!\rm T}_\bsigmas{}\big)^2= \big( \bm {\mathcal C}
-\widetilde \D_\bsigmas \big)\circ \big(
\widehat \D_\bsigmas
+\bsigma \widetilde \D_\bmu \widehat \D_\bmu
\big)
= 
\bm {\mathcal C}\widehat \D_\bsigmas
-\widetilde \D_\bsigmas 
\widehat \D_\bsigmas
- \widetilde \D_\bmu \widehat \D_\bmu
\, .
$$
Proceeding in a choice of $g\in \cc$ the above operator becomes (along $\Lambda$)
\begin{multline*}
{\mathcal C} 
\nabla_n
-(\nabla_n-\rho_\sigma)
\circ\big(\nabla_n-\tfrac12 \sigma
\Delta\big)
-
(\nabla_m-\rho_\mu)
\circ\big(\nabla_m-\tfrac12 \mu
\Delta\big)\\=
\Delta-\nabla_n^2-\nabla_m^2
+({\mathcal C}-H_\Sigma) \nabla_n
-H_\Xi \nabla_m\, .
\end{multline*}
It is not difficult to verify that the Laplacian $\Delta_\Lambda$ along $\Lambda$ acting on scalars has holographic formula 
\begin{multline*}
g^{ab}(\nabla_a - m_a \nabla_m - n_a \nabla_n)
 (\nabla_b - m_b \nabla_m - n_b \nabla_n)
 \\[1mm]
 \stackrel\Lambda=
 \Delta 
  -\nabla_m^2
 -\nabla.m \nabla_m
 -\nabla_n^2
-\nabla.n \nabla_n \qquad \qquad\qquad\qquad
\\
\qquad\qquad\qquad
 +(\nabla_m m^a)\nabla_a
 +m^a (\nabla_m m_a) \nabla_m
 +m^a (\nabla_m n_a) \nabla_n
 \\
 +(\nabla_n n^a) \nabla_a
 +n^a (\nabla_n m_a)\nabla_m
+n^a (\nabla_n n_a)\nabla_n\, .
\end{multline*}
Note that $m^a \nabla_m m_a = \frac12 \nabla_m m^2 \stackrel\Xi= H_\Xi
\stackrel\Xi= \frac 14 \nabla.m$
and 
 $\nabla_m m_a = \frac12 \nabla_a m^2
\stackrel\Xi= \hat m_a H_\Xi$. Analogous identities hold replacing $m$ with $n$.
Also $m^a m^b \nabla_a n_b\stackrel \Lambda= \mathcal C + H_\Sigma$ and
$n^a n^b \nabla_a m_b\stackrel \Lambda=  H_\Xi$. This establishes that
$$
\mathbf P_{\Lambda}=\bm \Delta_\Lambda\, .
$$

\newcommand{\msn}[2]{\href{http://www.ams.org/mathscinet-getitem?mr=#1}{#2}}
\newcommand{\hepth}[1]{\href{http://arxiv.org/abs/hep-th/#1}{arXiv:hep-th/#1}}
\newcommand{\maths}[1]{\href{http://arxiv.org/abs/math/#1}{arXiv:math/#1}}
\newcommand{\mathph}[1]{\href{http://lanl.arxiv.org/abs/math-ph/#1}{arXiv:math-ph/#1}}
\newcommand{\arxiv}[1]{\href{http://arxiv.org/abs/#1}{arXiv:#1}}

\addtocontents{toc}{\SkipTocEntry}
\section*{Acknowledgements}
C.A. would like to thank the hospitality of QMAP at UC Davis 
and the Department of Mathematics at King's College London 
during earlier stages of this project.
A.W.~was also supported by a Simons Foundation Collaboration Grant for Mathematicians ID 317562, and  thanks the University of Auckland for warm hospitality.
A.W. and A.R.G.
 gratefully acknowledge support from the Royal Society of New Zealand via Marsden Grant 16-UOA-051.

\end{document}